\documentclass[12pt,a4paper,twoside]{book}

\usepackage{pgf,pgfarrows,pgfnodes,pgfautomata,pgfheaps,pgfshade}
\usepackage{amsmath,amssymb,amsfonts,amsthm,comment}
\usepackage[latin1]{inputenc}
\usepackage{colortbl}
\usepackage[english]{babel}

\usepackage{lmodern}
\usepackage[T1]{fontenc} 

\usepackage{times}
\usepackage{tikz}
\usepackage{bm}
\usepackage{color}
\usepackage{tabu}
\usepackage{graphicx}
\usepackage{float}
\usepackage[latin1]{inputenc}
\usepackage{setspace}
\usepackage{afterpage}

\def\Z{\ifmmode{\mathbb Z}\else{$\mathbb Z$}\fi}
\def\C{\ifmmode{\mathbb C}\else{$\mathbb C$}\fi}
\def\Q{\ifmmode{\mathbb Q}\else{$\mathbb Q$}\fi}
\def\K{\ifmmode{\mathbb K}\else{$\mathbb K$}\fi}
\def\P{\ifmmode{\mathbb P}\else{$\mathbb P$}\fi}
\def\R{\ifmmode{\mathbb R}\else{$\mathbb R$}\fi}
\def\H{\ifmmode{\mathbb H}\else{$\mathbb H$}\fi}
\def\g{\ifmmode{\mathfrak g}\else {$\mathfrak g$}\fi}
\def\h{\ifmmode{\mathfrak h}\else {$\mathfrak h$}\fi}
\def\a{\ifmmode{\mathfrak a}\else {$\mathfrak a$}\fi}
\def\k{\ifmmode{\mathfrak k}\else {$\mathfrak k$}\fi}
\def\p{\ifmmode{\mathfrak p}\else {$\mathfrak p$}\fi}
\def\b{\ifmmode{\mathfrak b}\else {$\mathfrak b$}\fi}
\def\n{\ifmmode{\mathfrak n}\else {$\mathfrak n$}\fi}
\def\m{\ifmmode{\mathfrak m}\else {$\mathfrak m$}\fi}
\def\t{\ifmmode{\mathfrak t}\else {$\mathfrak t$}\fi}
\def\O{\ifmmode{\mathcal{O}}\else {$\mathcal{O}$}\fi}
\def\W{\ifmmode{\mathcal{W}}\else {$\mathcal{W}$}\fi}
\def\so{\ifmmode{\mathfrak {so}(n)} \else {$\mathfrak {so} (n)$}\fi}
\def\soc{\ifmmode{\mathfrak {so}(n,\C)} \else {$\mathfrak {so} (n,\C)$}\fi}
\def\u {\ifmmode{\mathfrak {u}(n)} \else {$\mathfrak {u} (n)$}\fi}
\def\su {\ifmmode{\mathfrak {su}(n)} \else {$\mathfrak {su} (n)$}\fi}
\def\sp {\ifmmode{\mathfrak {sp}(n)} \else {$\mathfrak {sp} (n)$}\fi}
\def\spc {\ifmmode{\mathfrak {sp}(n,\C)} \else {$\mathfrak {sp} (n,\C)$}\fi}
\def\slr {\ifmmode{\mathfrak {sl}(n,\R)} \else {$\mathfrak {sl} (n,\R)$}\fi}
\def\sl {\ifmmode{\mathfrak {sl}(n,\C)} \else {$\mathfrak {sl} (n,\C)$}\fi}
\def\slh {\ifmmode{\mathfrak {sl}(m,\H)} \else {$\mathfrak {sl} (n,\H)$}\fi}
\def\sup {\ifmmode{\mathfrak {su}(p,q)} \else {$\mathfrak {su} (p,q)$}\fi}
\def\gl {\ifmmode{\mathfrak {gl}(n,\R)} \else {$\mathfrak {gl} (n,\R)$}\fi}
\def\glc{\ifmmode{\mathfrak {gl}(n,\C)} \else {$\mathfrak {gl} (n,\C)$}\fi}
\def\glh{\ifmmode{\mathfrak {gl}(n,\H)} \else {$\mathfrak {gl} (n,\H)$}\fi}

\def\SLR{{SL(n,\R)}}
\def\SLH{{SL(m,\H)}}
\def\SLC{{SL(n,\C)}}
\def\SUP{{SU(p,q)}}

\DeclareMathOperator{\tr}{tr}

\DeclareMathOperator{\ad}{ad}

\DeclareMathOperator{\fix}{Fix}

\newtheorem{thm}{Theorem}
\newtheorem{prop}[thm]{Proposition}
\newtheorem{defi}[thm]{Definition}
\newtheorem{lemma}[thm]{Lemma}
\newtheorem{cor}[thm]{Corollary}

\begin{document}
\frontmatter
\begin{titlepage}
		\vspace*{-3.5cm}
		\begin{center}
			\includegraphics[scale=0.3]{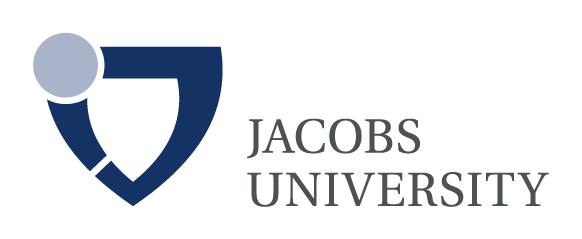}~\\[2cm]
			\textbf{\huge Schubert slices in the combinatorial \bigskip \newline geometry of flag domains }\\[0.5cm]
			
			\large by\\[0.5cm]
			
			\Large \textbf{Ana-Maria Brecan}\\[1cm]
			
			\large A thesis submitted in partial fulfillment\\ of the requirements for the degree of\\[0.5cm]
			
			\Large \textbf{Doctor of Philosophy\\ in Mathematics}
		\end{center}
		
		\vspace*{.5cm}
		\flushleft
		\noindent
		\large \textbf{Dissertation Committee:}\\[0.2cm]
		
		\noindent
		\large Prof. Dr. Alan Huckleberry\\
		\normalsize School of Engineering and Science, Jacobs University Bremen, Germany\\[0.2cm]
		
		\noindent
		\large Prof. Dr. Ivan Penkov\\
		\normalsize School of Engineering and Science, Jacobs University Bremen, Germany\\[0.2cm]
		
		\noindent
		\large Prof. Dr. Joseph A. Wolf\\
		\normalsize Departement of Mathematics, University of California, Berkeley, USA\\[2.5cm]

		\noindent
		\large \textbf{Date of Defense:} October 20th, 2014\\[0.2cm]
		
		\rule{\textwidth}{0.1pt}~\\
		\center School of Engineering and Science
	\end{titlepage}
\newpage
\afterpage{\null\newpage}

\clearpage
	\vspace*{\stretch{1}}
	\begin{minipage}{\textwidth}
		\begin{center}
			\Large \textbf{Statutory Declaration}\\[2cm]
		\end{center}
		I, Ana-Maria Brecan, hereby declare that I have written this PhD thesis independently, 
		unless where clearly stated otherwise. I have used only the sources, the data and the support 
		that I have clearly mentioned. This PhD thesis has not been submitted for conferral of degree 
		elsewhere. \\[1cm]
		
		Bremen, \today \\[1cm]
		\begin{minipage}{0.2\textwidth}
			Signature 
		\end{minipage}
		\begin{minipage}{0.8\textwidth}
			\rule{\textwidth}{0.1pt}
		\end{minipage}
	\end{minipage}
	\vspace{\stretch{3}}
	\clearpage~	
\newpage
\afterpage{\null\newpage}	
\newpage
\vspace*{5cm}
\begin{center}
\textbf{\huge Abstract}
\end{center}
\noindent
\bigskip
\newline
Flag domains are open orbits of real semisimple Lie groups in flag manifolds of their complexifications. Certain group theoretically defined compact complex submanifolds, which are regarded as cycles, are of basic importance for their complex geometric and representation theoretic properties. It is known that there are optimal Schubert varieties which intersect the cycles transversally in finitely many points and in particular determine them in homology. Here we give a precise description of these Schubert varieties in terms of certain subsets of the Weyl group and compute their total number for all the real forms of $\SLC$. Furthermore, we give an explicit description of the points of intersection in terms of flags and their number.
\newpage
\afterpage{\null\newpage}
\newpage
\vspace*{7cm}
\begin{center}
\textbf{To my father for allowing me to think about mathematics and life}
\end{center}

\newpage
\afterpage{\null\newpage}

\newpage
\begin{center}
\textbf{\huge Acknowledgement}
\end{center}
\noindent
\bigskip
\newline
I would like to express my deepest gratitude to my PhD adviser Prof. Dr. Alan Huckleberry for offering me the opportunity to carry out my PhD work under his valuable guidance. Prof. Huckleberry has been a continuous source of inspiration not only for me but for many of my colleagues and friends. I greatly appreciate his endless support and advice. 
\noindent
\bigskip
\newline
I am very thankful to Prof. Dr. Ivan Penkov and Prof. Dr. Joseph A. Wolf who agreed to review my thesis and also for their time and interest in accepting to be in my oral defense committee. 
\noindent
\bigskip
\newline
I would like to thank the Deutsche Forschungsgemeinschaft for financially supporting this PhD project.
\noindent
\bigskip
\newline
I would like to thank both Prof. Alan Huckleberry and Prof. Ivan Penkov for organising such a lively seminar in Algebra, Lie theory and Geometry and making my time as a graduate student at Jacobs University Bremen a very fruitful one. 
\noindent
\bigskip
\newline
In addition, I am grateful to Johanna Hennig, Dr. Keivan Mallahi-Karai, Prof. Dr. Vera Serganova and Prof. Dr. Gregg Zuckerman for discussing my research ideas and thoughts and to Prof. Dr. Peter Oswald for his helpful advice during my first year as a graduate student. 
\noindent
\bigskip
\newline
Writing my thesis was not all about doing mathematical research, but it also included learning a lot about myself. Irina Chiaburu, Victoria Glasenapp, Ioana Ivanciu, Christina Moeller and Marie-Luise Rose, you kept my spirits up and made me laugh, thank you for that and all the happy moments spent together. As to my family, I leaned on you for support and you offered your advice which I gladly took. I appreciate everything that you have done for me. 
\noindent
\bigskip
\newline
I wish to express my warm thanks to all the students, college office team, college masters and managers, resident associates and friends in College Nordmetall, my home away from home. I truly enjoyed the time spent as a resident associate of the college together with our team Sanja, Mircea and Pavel.
\noindent
\bigskip
\newline
Utz Ermel, I am most grateful and extremely happy that you were and are by my side, helping me in so many different ways, while loving me so fully. 

\tableofcontents
\mainmatter

\definecolor{cd70000}{RGB}{215,0,0}
\definecolor{c007c00}{RGB}{0,124,0}
\definecolor{c006a00}{RGB}{0,106,0}
\definecolor{c4d4d4d}{RGB}{77,77,77}
\definecolor{c5b62ff}{RGB}{91,98,255}
\definecolor{c006b00}{RGB}{0,107,0}
\definecolor{c0099ff}{RGB}{0,153,255}

\definecolor{c0000e1}{RGB}{0,0,225}
\definecolor{cff0000}{RGB}{255,0,0}

\chapter{Introduction}
\section{Background}
Here we deal with complex flag manifolds $Z=G/P$, with $G$ a complex semisimple Lie group and $P$ a complex parabolic subgroup, and consider the action of a real form $G_0$ of $G$ on $Z$. The real form $G_0$ is the connected real Lie group associated to the fixed point Lie algebra $\mathfrak{g}_0$ of an antilinear involution $\tau:\mathfrak{g}\rightarrow \mathfrak{g} $. It is known that $G_0$ has only finitely many orbits in $Z$ and therefore it has at least one open orbit. These and many other results relevant at the foundational level are proved in \cite{Wolf1969}. They are also summarised in the research monograph \cite{Fels2006}.
 Our work here is motivated by recent developments in the theory of cycle spaces of such a flag domain D. These arise as follows. Consider a choice $K_0$ of a maximal compact subgroup of $G_0$, i.e. $K_0$ is given by the fixed point set of a Cartan involution $\theta:G_0\rightarrow G_0$. Then $K_0$ has a unique orbit in $D\subset Z$, denoted by $C_0$, which is a complex submanifold of $D$. Equivalently, if $K$ denotes the complexification of $K_0$, one could look at $C_0$ as the unique minimal dimensional closed $K$-orbit in $D$. If $\dim C_0=q$, then $C_0$ can be regarded as a point in the Barlet space associated to $D$, namely $C^q(D)$, \cite{Barlet1975}. By definition, the objects of this space are formal linear combinations  $C=n_1C_1+\dots +n_kC_k$, with positive integral coefficients, where each $C_j$ is an irreducible $q$-dimensional compact subvariety of $D$. In this context $C_0$ is called the \textit{base cycle} associated to $K_0$. It is known that $C^q(D)$ is smooth at $C_0$ and thus one can talk about the irreducible component of $C^q(D)$ that contains $C_0$. 
\noindent
\bigskip
\newline
It is a basic method of Barlet and Koziarz, \cite{Barlet2000}, to transform functions on transversal slices to the cycles to functions on the cycle space. In the case at hand these transversal slices can be given using a special type of Schubert varieties which are defined with the help of the Iwasawa decomposition of $G_0$ (see part II of \cite{Fels2006} and the references therein). Recall that this is a global decomposition that exhibits $G_0$ as a product $K_0A_0N_0$, where each of the members of the decomposition are closed subgroups of $G_0$, $K_0$ is a maximal compact subgroup and $A_0N_0$ is a solvable subgroup. The Iwasawa decomposition is used to describe a type of Borel subgroups of $G$ which in a sense are as close to being real as possible. We define an \textit{Iwasawa-Borel} subgroup $B_I$ of $G$ to be a Borel subgroup that contains an Iwasawa-component $A_0N_0$ and we call the closure of an orbit of such a $B_I$ in $Z$ an \textit{Iwasawa-Schubert variety}. The Iwasawa-Borel subgroup can be equivalently obtained at the level of complex groups as follows. If $(G,K)$ is a symmetric pair, i.e. $K$ is defined by a complex linear involution, and $P=MAN$ where as usual $M$ is the centraliser of $A$ in $K$, then any such $B_I$ is given by choosing a Borel subgroup in $M$ and adjoining it to $AN$.
 
\noindent
\bigskip
\newline
The following result, \cite[p.101-104]{Fels2006}, has provided the motivation for our work.
\begin{thm}
If $S$ is an Iwasawa-Schubert variety such that $\dim S+ \dim C_0=\dim D$ and $S\cap C_0\ne \emptyset$ then the following hold:
\begin{enumerate}
\item{S intersects $C_0$ in only finitely many points $z_1,\dots, z_{d_s}$,}
\item{For each point of intersection the orbits $A_0N_0.z_j$ are open in $S$ and closed in $D$,}
\item{The intersection $S\cap C_0$ is transversal at each intersection point $z_j$.}
\end{enumerate}
\end{thm}
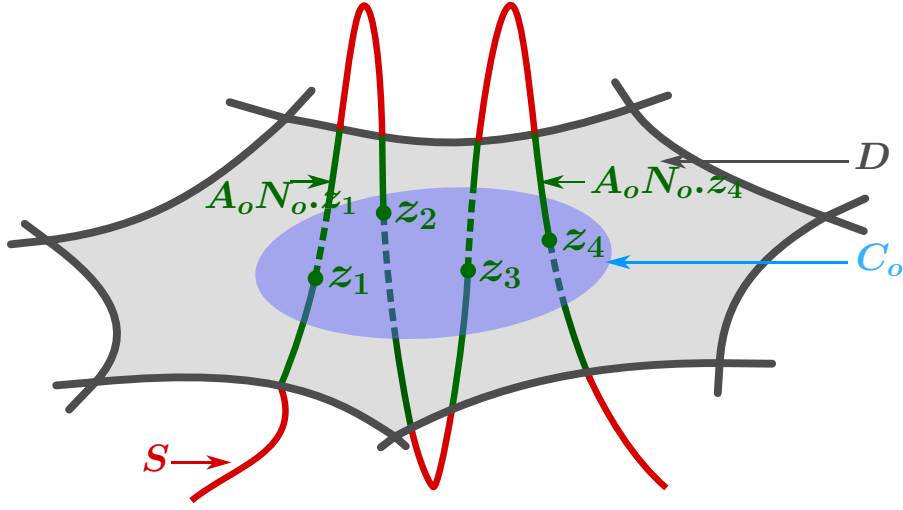
\begin{figure}
\begin{tikzpicture}[y=0.45pt,x=0.45pt,yscale=-1, inner sep=0pt, outer sep=0pt]
  \path[draw=cd70000,line join=miter,line cap=butt,miter limit=4.00,line
    width=2.423pt] (163.5588,658.8550) .. controls (198.5261,629.8707) and
    (258.5571,617.0950) .. (237.5293,562.4550);
  \path[draw=c007c00,line join=miter,line cap=butt,miter limit=4.00,line
    width=2.423pt] (237.3031,563.5435) .. controls (248.7112,537.0908) and
    (260.6735,499.9406) .. (266.4373,472.3315);
  \path[draw=c006a00,dash pattern=on 4.85pt off 2.42pt,line join=miter,line
    cap=butt,miter limit=4.00,line width=2.423pt] (323.4438,417.0863) .. controls
    (324.2845,441.3556) and (327.0978,472.5030) .. (331.0435,503.9956) .. controls
    (331.7694,509.7894) and (332.5336,515.5949) .. (333.3309,521.3714);
  \path[draw=c006a00,line join=miter,line cap=butt,miter limit=4.00,line
    width=2.423pt] (333.3309,521.3714) .. controls (337.0477,548.2996) and
    (341.4837,574.5976) .. (346.1088,596.1445);
  \path[draw=cd70000,line join=miter,line cap=butt,miter limit=4.00,line
    width=2.423pt] (346.1088,596.1445) .. controls (352.5108,625.9687) and
    (362.2141,647.0452) .. (364.9968,647.3812) .. controls (367.7795,647.7171) and
    (374.1445,619.4431) .. (380.1868,583.1480);
  \path[draw=c007c00,line join=miter,line cap=butt,miter limit=4.00,line
    width=2.423pt] (380.1868,583.1480) .. controls (386.5679,544.8170) and
    (392.5892,497.5398) .. (393.6510,465.5674);
  \path[draw=c006a00,fill=black,dash pattern=on 4.85pt off 2.42pt,line
    join=miter,line cap=butt,miter limit=4.00,fill opacity=0.000,line
    width=2.423pt] (461.6731,440.5972) .. controls (465.3945,460.7340) and
    (469.9549,481.1617) .. (475.6524,501.2191);
  \path[draw=c006a00,fill=black,line join=miter,line cap=butt,miter
    limit=4.00,fill opacity=0.000,line width=2.423pt] (474.4265,496.6222) ..
    controls (480.0166,516.3015) and (487.9273,540.2215) .. (495.9885,558.5639);
  \path[draw=cd70000,line join=miter,line cap=butt,miter limit=4.00,line
    width=2.423pt] (492.7379,551.1960) .. controls (508.0452,586.0259) and
    (531.5660,624.6879) .. (558.7284,648.1718);
  \path[fill=black,fill opacity=0.136,line width=3.450pt] (690.7066,421.6736) ..
    controls (631.8015,454.8065) and (602.3598,491.5535) .. (603.2412,543.3830) ..
    controls (477.0892,542.7489) and (396.7227,569.1726) .. (331.7659,604.4370) ..
    controls (285.3632,573.9740) and (238.7166,544.4138) .. (83.1953,559.2894) ..
    controls (119.4334,518.1431) and (101.7675,481.7330) .. (47.2015,442.0446) ..
    controls (129.8861,420.9641) and (150.7898,419.3473) .. (243.3171,338.2478) ..
    controls (321.5039,350.5868) and (378.0581,380.1023) .. (530.6691,328.7854) ..
    controls (557.9637,360.8697) and (582.7279,383.6231) .. (690.7066,421.6736) --
    cycle;
  \path[draw=c4d4d4d,line join=round,line cap=round,miter limit=4.00,line
    width=2.961pt] (601.4771,568.6760) -- (603.0728,543.5134) .. controls
    (476.9208,542.8793) and (397.8545,569.8668) .. (332.8977,605.1311) --
    (341.9771,613.2181)(321.2204,614.1172) -- (332.8977,605.1311) .. controls
    (286.4950,574.6681) and (239.4150,544.4867) .. (83.8936,559.3623) --
    (61.8849,582.5202)(51.1581,562.2123) -- (83.8936,559.3623) .. controls
    (120.1318,518.2158) and (98.1318,480.3600) .. (43.5658,440.6717) --
    (24.8897,426.7137)(13.4171,443.9245) -- (43.5658,440.6717) .. controls
    (128.5466,429.9580) and (178.4664,398.0614) .. (242.2818,338.4358) --
    (195.3744,324.7696)(263.5861,315.0317) -- (242.2818,338.4358) .. controls
    (320.4686,350.7748) and (377.7935,380.9571) .. (530.4045,329.6403) --
    (560.1675,319.1294)(516.0959,308.0182) -- (530.4045,329.6403) .. controls
    (557.5442,370.9027) and (639.7484,401.5543) .. (690.9716,420.9084) --
    (724.2238,405.4990)(646.7682,543.9974) -- (603.0728,543.5134) .. controls
    (605.9742,491.4326) and (633.4505,453.9494) .. (690.9716,420.9084) --
    (722.9980,432.8558);
  \path[cm={{1.30509,-0.08668,0.0,0.5532,(503.8673,1153.7369)}},fill=c5b62ff,miter
    limit=4.00,fill opacity=0.444,line width=3.485pt]
    (7.1429,-1271.2092)arc(-0.014:180.014:113.571)arc(-180.014:0.014:113.571) --
    cycle;
  \path[draw=c006a00,dash pattern=on 4.85pt off 2.42pt,line join=miter,line
    cap=butt,miter limit=4.00,line width=2.423pt] (266.4373,472.3315) .. controls
    (270.9636,450.6499) and (273.9300,434.5796) .. (277.2494,413.9569);
  \path[draw=c006a00,fill=c006b00,line join=miter,line cap=butt,miter
    limit=4.00,line width=2.423pt] (277.2494,413.9569) .. controls
    (281.0334,390.4480) and (285.0166,362.6458) .. (287.7312,342.3803);
  \path[draw=cd70000,line join=miter,line cap=butt,miter limit=4.00,line
    width=2.423pt] (286.9728,347.7979) .. controls (294.8322,289.1232) and
    (300.0817,242.7944) .. (307.6547,243.7900) .. controls (315.2276,244.7856) and
    (321.4604,295.1942) .. (322.2974,353.2353);
  \path[draw=c006a00,fill=c006b00,line join=miter,line cap=butt,miter
    limit=4.00,line width=2.423pt] (322.2974,353.1269) .. controls
    (322.6048,374.4363) and (322.6591,394.4388) .. (323.4437,417.0863);
  \path[draw=c006a00,dash pattern=on 4.85pt off 2.42pt,line join=miter,line
    cap=butt,miter limit=4.00,line width=2.423pt] (393.6510,465.5674) .. controls
    (394.3671,444.0049) and (395.9839,420.3467) .. (398.2094,396.7586);
  \path[draw=c006a00,line join=miter,line cap=butt,miter limit=4.00,line
    width=2.423pt] (398.2094,396.7586) .. controls (399.3742,384.4131) and
    (400.5272,368.5153) .. (401.9837,356.5186);
  \path[draw=cd70000,line join=miter,line cap=butt,miter limit=4.00,line
    width=2.423pt] (402.4688,357.6384) .. controls (409.6982,298.0914) and
    (419.6988,249.1192) .. (427.9653,243.7900) .. controls (436.2318,238.4607) and
    (442.3301,286.4149) .. (449.8633,356.8215);
  \path[draw=c006a00,fill=c006b00,line join=miter,line cap=butt,miter
    limit=4.00,line width=2.423pt] (449.2132,350.5371) .. controls
    (451.7171,373.9392) and (455.5651,405.7611) .. (460.1235,431.9574) .. controls
    (460.6234,434.8300) and (461.1396,437.7107) .. (461.6731,440.5972);
  \path[cm={{2.28827,0.0,0.0,2.28827,(126.10941,3872.9184)}},draw=c006a00,fill=c006b00]
    (64.1447,-1486.0249)arc(0.000:180.000:2.652)arc(-180.000:0.000:2.652) --
    cycle;
  \path[cm={{2.28827,0.0,0.0,2.28827,(182.8461,3818.0249)}},draw=c006a00,fill=c006b00]
    (64.1447,-1486.0249)arc(0.000:180.000:2.652)arc(-180.000:0.000:2.652) --
    cycle;
  \path[cm={{2.28827,0.0,0.0,2.28827,(253.14128,3866.3899)}},draw=c006a00,fill=c006b00]
    (64.1447,-1486.0249)arc(0.000:180.000:2.652)arc(-180.000:0.000:2.652) --
    cycle;
  \path[cm={{2.28827,0.0,0.0,2.28827,(320.67857,3840.915)}},draw=c006a00,fill=c006b00]
    (64.1447,-1486.0249)arc(0.000:180.000:2.652)arc(-180.000:0.000:2.652) --
    cycle;
    \path[fill=c006a00] (276.4216,485.5108) node[above right] (text5694) {\Large \textcolor{c006a00}{$\bm{z_1}$}};
  \path[fill=c006a00] (331.71188,430.88452) node[above right] (text5694-9) {\Large \textcolor{c006a00}{$\bm{z_2}$}};
  \path[fill=c006a00] (401.92389,480.15973) node[above right] (text5694-8) {\Large \textcolor{c006a00}{$\bm{z_3}$}};
  \path[fill=c006a00] (471.26907,454.68808) node[above right] (text5694-2) {\Large \textcolor{c006a00}{$\bm{z_4}$}};
  \begin{scope}[cm={{0.58,0.0,0.0,0.58,(146.64374,129.70826)}}]
    \path[color=black,fill=c006b00,nonzero rule,line width=2.367pt]
      (171.4375,449.7188) -- (171.4375,452.6875) -- (220.7500,452.6875) --
      (220.7500,449.7188) -- (171.4375,449.7188) -- cycle;
    \path[draw=c006b00,fill=c006b00,even odd rule,line width=1.184pt]
      (208.9288,451.2065) -- (203.0109,457.1245) -- (223.7236,451.2065) --
      (203.0109,445.2886) -- (208.9288,451.2065) -- cycle;
  \end{scope}
  \begin{scope}[cm={{0.58,0.0,0.0,0.58,(146.64374,129.70826)}}]
    \path[color=black,fill=c006b00,line width=2.286pt] (539.1563,450.2813) --
      (539.1563,453.1250) -- (593.6875,453.1250) -- (593.6875,450.2813) --
      (539.1563,450.2813) -- cycle;
    \path[draw=c006b00,fill=c006b00,even odd rule,line width=1.143pt]
      (550.5785,451.7007) -- (556.2948,445.9844) -- (536.2879,451.7007) --
      (556.2948,457.4169) -- (550.5785,451.7007) -- cycle;
  \end{scope}
  \begin{scope}[cm={{0.58,0.0,0.0,0.58,(146.64374,129.70826)}}]
    \path[color=black,fill=cd70000,nonzero rule,line width=2.367pt]
      (1.0625,855.5313) -- (1.0625,858.5000) -- (75.8125,858.5000) --
      (75.8125,855.5313) -- (1.0625,855.5313) -- cycle;
    \path[draw=cd70000,fill=cd70000,even odd rule,line width=1.184pt]
      (63.9623,857.0068) -- (58.0444,862.9247) -- (78.7571,857.0068) --
      (58.0444,851.0888) -- (63.9623,857.0068) -- cycle;
  \end{scope}
  \begin{scope}[cm={{0.59178,0.0,0.0,0.58,(135.21887,129.70826)}}]
    \path[color=black,fill=c0099ff,nonzero rule,line width=2.349pt]
      (640.4375,565.9688) -- (640.4375,568.9063) -- (970.4688,568.9063) --
      (970.4688,565.9688) -- (640.4375,565.9688) -- cycle;
    \path[draw=c0099ff,fill=c0099ff,even odd rule,line width=1.174pt]
      (652.1742,567.4513) -- (658.0455,561.5800) -- (637.4959,567.4513) --
      (658.0455,573.3226) -- (652.1742,567.4513) -- cycle;
  \end{scope}
  \begin{scope}[cm={{0.58,0.0,0.0,0.58,(146.64374,129.70826)}}]
    \path[color=black,fill=c4d4d4d,nonzero rule,line width=2.456pt]
      (712.6875,420.0938) -- (712.6875,423.1563) -- (971.4688,423.1563) --
      (971.4688,420.0938) -- (712.6875,420.0938) -- cycle;
    \path[draw=c4d4d4d,fill=c4d4d4d,even odd rule,line width=1.228pt]
      (724.9786,421.6169) -- (731.1189,415.4766) -- (709.6277,421.6169) --
      (731.1189,427.7573) -- (724.9786,421.6169) -- cycle;
  \end{scope}
  \path[fill=c006a00] (174.00706,418.19388) node[above right] (text5694-9-5) {\large \textcolor{c006a00}{$\bm{A_o N_{o}.z_1}$}};
  \path[fill=cd70000] (121.87347,633.55927) node[above right] (text5694-9-7) {\large \textcolor{cd70000}{$\bm{S}$}};
  \path[fill=c0099ff,fill opacity=0.784] (716.4162,470.61624) node[above right]
    (text5694-9-52) {\large \textcolor{c0099ff}{$\bm{C_o}$}};
  \path[fill=c4d4d4d] (714.57745,381.03186) node[above right] (text5694-9-6) {\large \textcolor{c4d4d4d}{$\bm{D}$}};
  \path[fill=c006a00] (494.8898,405.4805) node[above right] (text5694-9-2) {\large \textcolor{c006a00}{$\bm{A_o N_{o}.z_4}$}};
\end{tikzpicture}
\caption{Graphic representation of Theorem 1.}
\end{figure}
 
\noindent
For the next steps in this general area, for example computing in concrete terms the trace transform indicated above, we feel that it is important to understand precisely the combinatorial geometry of this situation. This means in particular to describe precisely which Schubert varieties intersect the base cycle $C_0$, their points of intersection and the number of these points. In particular such results will describe the base cycle $C_0$ (or any cycle in the corresponding cycle space, see \cite[p.104]{Fels2006}) in the homology ring of the flag manifold $Z$. 
We have done this for the classical semisimple Lie group $\SLC$ and its real forms $\SLR$, $\SUP$ and $\SLH$ using methods which would seem sufficiently general to handle all classical semisimple Lie groups. 
The description of the Schubert varieties is formulated combinatorially in terms of elements of the Weyl group of $G$. Interesting combinatorial conditions arise and the tight correspondence in between combinatorics and geometry is made explicit. 
\section{Structure}
The work here is organised in three chapters, each of them describing the results for a particular real form of $\SLC$.
\noindent
\bigskip
\newline
 $\bullet$ \textbf{Chapter 1:} Here we are concerned with the real form $\SLR$ where up to orientation we have only one open orbit. In the case of the full flag manifold the Weyl group elements that parametrize the Schubert varieties of interest can be obtained from a simple game that chooses pairs of consecutive numbers from the ordered set $\{1,\dots, n\}$. Moreover, their total number is also a well-known number, the double-factorial. Surprisingly, the number of intersection points with the base cycle does not depend on the Schubert variety and in each case it is $2^{n/2}$. The main results are presented in Theorem $8$ and Theorem $10$. In the case of the partial flag manifold the results depend on whether or not the open orbit is measurable. In the measurable case the main results are found in Theorem 14 and Theorem 15, and in the non-measurable case in Proposition 19.
\noindent
\bigskip
\newline 
$\bullet$ \textbf{Chapter 2:} The real form in this case is $SL(m,\H)$. In terms of methods of study our work here is similar to that for $\SLR$, although as in the case of $\SUP$ the intersection occurs at precisely one point, independent of the chosen Schubert variety. For a statement of results in the $G/B$ case see Theorem $23$ and Theorem $25$. For the partial flag manifold, the main results in the measurable case are found in Theorem $29$ and Theorem $31$ and in the non-measurable case in Proposition $32$.
\noindent
\bigskip
\newline
$\bullet$ \textbf{Chapter 3:} This chapter is concerned with the case of the real form $SU(p,q)$. Due to the fact that this group has a large number of open orbits in, e.g. $G/B$, in a sense this situation is more interesting than that for the other real forms. An explicit algorithm which gives the parametrization of the Schubert varieties in terms of Weyl group elements is provided. Unlike the case of the real form $\SLR$, where the number of intersection points is huge, in the case of $SU(p,q)$ there is only one point of intersection, independent of the open orbit or the Schubert variety of interest. The main results for the full flag manifold are presented in Theorem $39$ and Theorem $41$ and the results for the partial flag manifold are presented in Theorem $45$. In this case flag domains are automatically measurable. As a last remark we run the algorithm for maximal parabolics to obtain explicit formulas for the total number of Schubert varieties. 

\chapter{The case of the real form $\protect\SLR$}
\section{Preliminaries}
Let $G=\SLC$ and $P$ be a parabolic subgroup of $G$ corresponding to a dimension sequence $d=(d_1,\dots, d_s)$ with $d_1+\dots + d_s=n$, i.e. $P$ is given by block upper triangular matrices of sizes $d_1$ up to $d_s$, respectively. Recall that in this case the flag manifold $Z=G/P$ can be identified with the set of all \textit{partial flags of type} $d$, namely $\{V: 0\subset V_1\subset \dots \subset V_s=\C^n\}$, where $\dim(V_i/V_{i-1})=d_i,\, \forall 1\le i \le s,$ $\dim V_0=0$. Equivalently, $Z$ can be defined with the help of the sequence $\delta=(\delta_1,\dots, \delta_s)$, with $\delta_i:=\sum_{k=1}^id_k=\dim V_i$, for all $1\le i \le s$. If $(e_1,\dots, e_n)$ is the standard basis in $\C^n$, the flags consisting of subspaces spanned by elements of this basis are called \textit{coordinate flags}. In the particular case when each $d_i=1$ we have a complete flag and the corresponding full flag variety is identified with the homogeneous space $\hat{Z}=G/B$, with $B$ the Borel subgroup of upper triangular matrices in $G$. In terms of the dimension sequence $d$ we have that $\dim Z=\sum_{1\le i < j \le s}d_id_j$. For each $d$ a fibration $\pi:\hat{Z}\rightarrow Z$ is defined by sending a complete flag to its corresponding partial flag of type $d$.
\noindent
\bigskip
\newline
Let us look at $\C^n$ equipped with the standard real structure $\tau: \C^n\rightarrow \C^n$, $\tau(v)=\overline{v}$ and the standard non-degenerate complex bilinear form $b:\C^n\times\C^n\rightarrow \C$, $b(v,w)=v^t\cdot w$ and view $G$ as the group of complex linear transformations on $\C^n$ of determinant $1$. Moreover, let $h:\C^n\times \C^n\rightarrow \C$ be the standard Hermitian form defined by $h(v,w)=b(\tau(v),w)=\overline{v}^t\cdot w$.
It follows that $G_0:=\{A\in G:\, \tau\circ A=A\circ \tau\}=\SLR$. If $\theta$ denotes both the Cartan involution on $G_0$ and on $G$ defined by $\theta(A)=(A^{-1})^t$, then $K_0:=SO(n,\R)$ and its complexification $K:=SO(n,\C)$ are both obtained as fixed points of the respective $\theta$'s. Fix the Iwasawa decomposition $G_0=\SLR=K_0A_0N_0$, where $A_0N_0$ are the upper triangular matrices with positive diagonal entries in $\SLR$. Thus, in this special case, the Iwasawa Borel subgroup $B_I$ is just the standard Borel subgroup of upper triangular matrices in $\SLC$.
\noindent
\bigskip
\newline
The following definitions give a geometric description in terms of flags of the open $G_0$-orbits in $Z$ and the base cycles associated to this open orbits. These results can be found in \cite{Huckleberry2001} and \cite{Huckleberry2002}.
\begin{defi}
A flag $z=(0 \subset V_{1}\subset \dots \subset V_{s}\subset \mathbb{C}^n)$ in $Z=Z_d$ is said to be $\tau$-\textbf{generic} if $dim(V_i\cap \tau(V_j))=max \{0,\delta_i+\delta_j-n\},\ \forall 1\le i,j\le s$. In other words, these dimensions should be minimal.
\end{defi}
\noindent
Note that in the case of $Z=G/B$ the condition of  $\tau$-genericity is equivalent to $$\tau(V_{j})\oplus V_{n-{j}}=\mathbb{C}^n, \quad \forall 1\le j \le [n/2].$$
\begin{defi}
A flag $z=(0 \subset V_{1}\subset \dots \subset V_{s}\subset \mathbb{C}^n)$ in $Z=Z_d$ is said to be \textbf{isotropic} if either $V_i \subseteq V_j^{\perp}$ or $V_i^\perp \subseteq V_j,\, \forall 1\le i,j \le s$. In other words, $$dim(V_i\cap V_j^\perp)=min\{\delta_i,n-\delta_j\}.$$ 
\end{defi}
\noindent
Note that in the case of $Z=G/B$ and $m=[n/2]$, the isotropic condition on flags is equivalent to $V_i\subset V_i^{\perp}$ for all $1\le i\le m$, $V_m=V_m^\perp$, if $n$ is even and the subspaces $V_{n-i}$ are determined by $V_{n-i}=V_{i}^{\perp},$ for all $1\le i \le m$.
\noindent
\bigskip
\newline
If $n=2m+1$ the unique open $G_0$-orbit is described by the set of $\tau$-generic flags. If $n=2m$ a notion of orientation arises on $\C^n_{\R}$ that is independent on the choice of basis. Since $G_0$ preserves orientation in this case we have two open orbits defined by the set of positively oriented $\tau$-generic flags and by the set of negatively oriented $\tau-$generic flags. One can define a map that reverses orientation and interchanges the two open orbits. It is therefore enough to only consider the open orbit defined by the positively-oriented flags. In each of the open orbits the base cycle $C_0$ is characterised by the set of isotropic flags.   
\noindent
\bigskip
\newline
Finally, recall the definition of Schubert varieties in a general flag manifold $Z=G/P$ and a few interesting properties in order to establish notation. In general, for a fixed Borel subgroup $B$ of $G$, a $B$-orbit $\mathcal{O}$ in $Z$ is called a Schubert cell and the closure of such an orbit is called a Schubert variety. A Schubert cell $\mathcal{O}$ in $Z$ is parametrized by an element $w$ of the Weyl group of $G$ and $Z$ is the disjoint union of finitely many such Schubert cells. Furthermore, the integral homology ring of $Z$, $H_*(Z,\Z)$ is a free $\Z$-module generated by the set of Schubert varieties. 
\noindent
\bigskip
\newline
If $G=\SLC$ and $T$ is the maximal torus of diagonal matrices in $G$, then the Weyl group of $G$ with respect to $T$ can be identified with $\Sigma_n$, the permutation group on $n$ letters. Moreover, the complete coordinate flags in $G/B$ are in $1-1$ correspondence with elements of $\Sigma_n$. Given a complete coordinate flag $$<e_{i_1}>\subset \dots \subset <e_{i_1},e_{i_2},\dots, e_{i_k}>\subset \dots \subset \mathbb{C}^n,$$ one can define a permutation $w$ by $w(k)=i_k$ for all $k$ and viceversa. The complete coordinate flags are also in $1-1$ correspondence with permutation matrices in $GL(n,\C)$. Given an element $w\in \Sigma_n$ one obtains a permutation matrix with column $i^{th}$ equal to $e_{w(i)}$ for each $i$. 
For this reason we use the symbol $w$ for both an element of $\Sigma_n$ in one line notation, i.e. w(1)w(2)\dots w(n), or for the corresponding permutation matrix. It will be clear from the context to which kind of representation we are referring to. 
\noindent
\bigskip
\newline
The fixed points of the maximal torus $T$ in $G/B$ are the coordinate flags $V_w$ for $w\in W$ and $G/B$ is the disjoint union of the Schubert cells $\mathcal{O}_w:=BV_w$, where $w\in W$. Since $\SLC/B$ is isomorphic to $GL(n,C)/B'$, where $B'$ are the upper triangular matrices in $GL(n,\C)$ we have another useful way of visualising Schubert cells via their matrix canonical form. If $\mathcal{O}_w$ is a given Schubert cell parametrized by $w\in W$,  then it can be represented as the matrix $Bw$ in which the lowest nonzero entry in each column is a 1 (on the $i^{th}$ column the $1$ lies on row $w(i)$) and the entries to the right of each leading $1$ are all zero. What is basically done is filling the permutation matrix with $*$'s above each $1$, having in mind the rule that to the right of each $1$ the elements must be zero. This leads us to the observation that the dimension of the Schubert cell $\mathcal{O}_w$ is given by the number of inversions in the permutation $w$, that is the length of $w$, or the number of $*$'s in the matrix representation.  
\noindent
\bigskip
\newline
Schubert cells and varieties can also be defined in $G/P$. Those will be indexed by elements of the coset $\Sigma_n/{\Sigma_{d_1}\times \Sigma_{d_2}\times \dots \times \Sigma_{d_s}}$ and each right coset contains a minimal representative, i.e. a unique permutation $w$ such that $w(1)<\dots < w(d_1), w(d_1+1)<\dots <w(d_1+d_2), \dots , w(d_1+\dots d_{s-1}+1)<\dots<w(d_1+\dots+d_s)=w(n)$. The dimension of the Schubert cell $C_{wP}:=BwP/P$ is the length of the minimal representative $w$ and there is a unique lift to a Schubert cell in $G/B$ of the same dimension, namely $\mathcal{O}_w:=BwB/B$. When working with $\mathcal{O}_{wP}$ we can thus use the same matrix representation as for $\mathcal{O}_w$ and many times we will refer to $\mathcal{O}_{wP}$ just by $\mathcal{O}_w$. 
\section{Dimension-related computations}
\noindent
It is important for our discussion to compute the dimension of the base cycle and of the respective dual Schubert varieties in both the case of $G/B$ and of $G/P$. In the case when $B$ is the standard Borel subgroup in $\SLC$ and $K=SO(n,\C)$, the cycle $C_0$ is a compact complex submanifold of $D$ represented in the form $C_0=K.z_0\cong K/(K\cap B_{z_0})$ for a base point $z_0\in D$, where $K\cap B_{z_0}$ is a Borel subgroup of $K$. Since $C_0$ is a complex manifold $$dim\,C_0=dim\,T_{z_0}C_0=dim\,\mathfrak{k}/\mathfrak{k}\cap\mathfrak{b}_{z_0},$$ where $\mathfrak{k}$ is the Lie algebra associated to $K$ and $\mathfrak{b}$ is the Borel subalgebra associated to $B$.  Thus in the case when $n=2m$, $\dim C_0=m^2-m$ and the Schubert varieties of interest must be of dimension $m^2$. If $n=2m+1$, then $\dim C_0=m^2$ and the Schubert varieties of interest are among those of dimension $m^2+m$.

\section{Introduction to the combinatorics}
The next two sections give a full description of the Schubert varieties of interest that intersect the base cycle $C_0$, the points of intersection and their number, in the case of an open $\SLR$-orbit $D$ in $Z$.
The first case to be considered is the case of $Z=G/B$, where $$\mathcal{S}_{C_0}:=\{S_w \text{ Schubert variety }: dim S_w+dim C_0=dimZ \text{ and } S_w\cap C_0 \ne \emptyset \}.$$ In what follows we describe the conditions that the element $w$ of the Weyl group that parametrizes the Schubert variety $S_w$ must satisfy in order for $S_w$ to be in $S_{C_0}$. One of the main ingredients for this is the fact that $S_w\cap D \subset \mathcal{O}_w $ and the fact that if $S_w\cap D\ne \emptyset$, then $S_w\cap C_0\ne \emptyset$. Moreover, no Schubert variety of dimension less than the codimension of the base cycle intersects the base cycle. These are general results that can be found in \cite[p.101-104]{Fels2006}
\begin{defi}
A permutation $w=k_1\dots k_ml_*l_m\dots l_1$ is said to satisfy the \textbf{spacing condition} if $l_i<k_i,\,\forall 1\le i\le m,$ where $l_*$ is removed
from the representation in the case $n=2m$.
\end{defi}
\noindent
For example, $265431$ satisfies the spacing condition, while $261534$ does not satisfy the spacing condition.
\begin{defi}
A permutation $w=k_1\dots k_ml_*l_m\dots l_1$ is said to satisfy the \textbf{double box contraction} condition if $w$ is constructed by the \textbf{immediate predecessor algorithm}: 
\newline
 $\bullet$ Start by choosing $k_1$ and $l_1:=k_1-1$ from the ordered set $\{1,\dots, n\}$. If we have chosen all the numbers up to $k_i$ and $l_i$ then to go to the step $i+1$ we make a choice of $k_{i+1}$ and $l_{i+1}$ from the ordered set $\{1,\dots, n\}-\{k_1,l_1,\dots,k_i,l_i\}$ in such a way that $l_{i+1}$ sits inside the ordered set at the left of $k_{i+1}$. 
\end{defi}
\noindent
Remark that a permutation that satisfies the double box contraction automatically satisfies the spacing condition as well, but not conversely. For example, $256341$ satisfies both the double box contraction and consequently the spacing condition while $265431$ does not satisfy the double box contraction even though it satisfies the spacing condition.
\noindent
\bigskip
\newline
Figure \ref{SL(n,R)} is an example of how the immediate predecessor algorithm works in the case of building an element $w\in \Sigma_6$ which satisfies the double box contraction condition.
\noindent
\bigskip
\newline
The next results are meant to establish a tight correspondence between the combinatorics of the Weyl group elements that parametrize the Schubert varieties in $\mathcal{S}_{C_0}$ and the geometry of flags that describe the intersection points. Namely, we prove that the spacing condition on Weyl group elements corresponds to the $\tau-$generic condition on flags. Similarly, the double box contraction condition on Weyl group elements corresponds to the isotropic condition on flags.  

\begin{figure}
\includegraphics[scale=0.85]{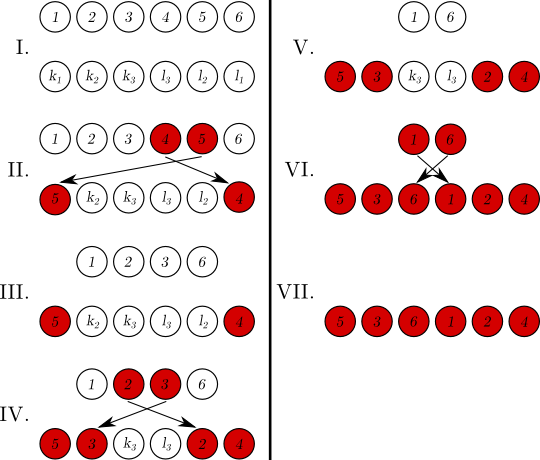}
\caption{Immediate predecessor algorithm example.}
\label{SL(n,R)}
\end{figure}
\newpage
\section{Main results for $Z=G/B$}
The first result of this section describes the Schubert varieties that intersect the base cycle independent of their dimension.
\begin{prop}
A Schubert variety $S_w$ corresponding to a permutation $$w=k_1\dots k_ml_*l_m\dots l_1,$$ where $l_*$ is removed from the representation in the case $n=2m$, has non-empty intersection with $C_0$ if and only if $w$ satisfies the spacing condition. 
\end{prop}

\begin{proof}
We use the fact that if a Schubert variety intersects the open orbit $D$, then it also intersects the cycle $C_0$ and prove that $S_w$ contains a $\tau$-generic point if and only if $w$ satisfies the spacing condition.
\noindent
\bigskip
\newline
First assume that $l_i<k_i$ for $i\le m$, $n=2m$. Under this assumption we need to prove that $$\tau(V_i)\oplus V_{2m-i}=\mathbb{C}^n\quad \forall i\le m,$$ where $V$ is an arbitrary flag in $S_w$. This is equivalent to showing that the matrix formed from the vectors generating $\tau(V_i)$ and $V_{2m-i}$ has maximal rank . Form the following pairs of vectors $(v_i,\bar{v}_i)$ and the matrices $$[v_1\bar{v}_1\dots v_i\bar{v}_iv_{i+1}\dots v_{2m-i}],\quad \forall i\le m.$$ These are the matrices corresponding to $$\tau(V_i)=<\bar{v}_1,\dots , \bar{v}_i>$$ and $$V_{2m-i}=<v_1,\dots, v_i, v_{i+1},\dots v_{2m-i}>.$$ We carry out the following set of operations keeping in mind that the rank of a matrix is not changed by row or column operations. The initial matrix is the canonical matrix representation of the Schubert cell $\mathcal{O}_w$. For the step $j=1$, $l_1<k_1$ and the last column corresponding to $l_1$ is eliminated from the initial matrix. 
\noindent
\bigskip
\newline
Next denote by $c_h$ the $h^{th}$ column of this matrices and obtain the following:
\newline
%On $c_2$ we create a $1$ by normalizing on row $l_1$ (since the last vector $v_{2m}$ from the initial matrix is not there anymore, we have no other column with a $1$ on position $l_1$). We use the new created $1$ to make $0$'s on $c_2$ below row $l_1$ and to make zeros on the left of this $1$ (on row $l_1$ on column $j>2$). In this process we make sure that after each operation (addition and multiplication by scalar of columns) a normalization is done on the resulting row or column in case, on the spot in the matrix where a $1$ existed before.  
On the second column $c_2$ zeros are created on all rows starting with $k_1$ and going down to row $l_1+1$. This is done by subtracting suitable multiplies of $c_2$ from the columns in the matrix having a $1$ on these rows and putting the result on $c_2$. For example to create a zero on the spot corresponding to $k_1$ the second column is subtracted from the first column and the result is left on the second column.   On row $l_1$ a $1$ is created by normalizing. This $1$ is the only $1$ on row $l_1$, because the last column that contained a $1$ on row $l_1$ was removed from the matrix. We now want to create zeros on $c_j$ for $j>2$ on row $l_1$. It is enough to consider those columns which have $1$'s on rows greater than $l_1$, because the columns with $1$'s on rows smaller than $l_1$ already have zeros below them. Finally, subtract from these columns suitable multiplies of $ c_2$.  We thus create a matrix that represents points in the Schubert variety $k_1l_1k_2\dots k_ml_m\dots l_2$, which obviously has maximal rank.
\noindent
\bigskip
\newline
Assume by induction that we have created the maximal rank matrix corresponding to points in the Schubert variety $k_1l_1k_2l_2\dots k_jl_jk_{j+1}\dots k_ml_m\dots l_{j+1}$. To go to the step $j+1$ remove from the matrix the last column corresponding to $l_{j+1}$ and add in between $c_{2j+1}$ and $c_{2j+2}$ the conjugate of $c_{2j+1}$ and reindex the columns. 
\noindent
\bigskip
\newline
On column $c_{2(j+1)}$ zeros are created on all rows starting with $k_{j+1}$ and going down to $l_{j+1}+1$ by subtracting suitable multiplies of $c_{2(j+1)}$ from the columns in the matrix having a $1$ on this rows and putting the result on $c_{2(j+1)}$. Next a $1$ is created on row $l_{2(j+1)}$ by normalisation. Again this is the only spot on row $l_{2(j+1)}$ with value $1$, because the column which had a $1$ on this spot was removed from the matrix. By subtracting suitable multiplies of $ c_{2(j+1)}$ from columns having a $1$ on spots greater than $l_{(j+1)}$, we create zeros on row $l_{j+1}$ at the right of the $1$ on column $c_{2(i+1)}$. We thus obtain points in the maximal rank matrix of the Schubert variety indexed by $k_1l_1k_2l_2\dots k_jl_jk_{j+1}l_{j+1}\dots k_ml_m\dots l_{j+2}$.   
%For the step $j=2$ we delete from the new matrix the column corresponding to $l_2$ and observe that we have columns with $1$'s on all spots ex 
%use the same argument for $c_4$ as we used for $c_2$. We repeat the argument for each step $j$ until we reach $j=i$ to get points in a new Schubert variety with $k_1l_1k_2l_2\dots k_i l_ik_{i+1}\dots k_ml_m\dots l_{i+1}$ as the corresponding permutation. 
\noindent
\bigskip
\newline
At step $j=m$ we obtain points in the maximal rank matrix of the Schubert variety indexed by $k_1l_1k_2l_2\dots k_ml_m$.
 For the odd dimensional case we insert in the middle a column corresponding to $l_*$ and observe that this of course does not change the rank of the matrix.
\noindent
\bigskip 
\newline
Conversely, if the spacing condition is not satisfied let $i$ be the smallest such that $k_i<l_i$ and look at the matrix $\tau(V_i)\oplus V_{2m-i}$. Then use the same reasoning as above to create a $1$ on the spot in the matrix corresponding to row $l_j$ and column $2j$ for all $j<i$ using our chose of $i$. Now there is a $1$ on each row in the matrix except on row $l_i$. Column $c_{2i}$ and $c_{2i-1}$ both have a $1$ on position $k_i$ and zeros bellow. Since $l_i>k_i$ this implies that for each $j<i$ there exist a column in the matrix that has $1$ on row $j$. Using these $1$'s we begin subtracting $c_{2i}$ from suitable multiplies of each such column, starting with $c_{2i-1}$ then going to the one that has 1 on the spot $k_i-1$, then to the one the has $1$ on the spot $k_i-2$ and so on and at each step the result is left on $c_{2i}$. This creates zeros on all the column $c_{2i}$ and proves that the matrix does not have maximal rank.
\end{proof}
\noindent
\textbf{Remark.} For dimension computations, recall how to compute the length $|w|$, of an element $w$ of $\Sigma_n$. Start with the number $1$ and move it from its position toward the left until it arrives at the beginning and associate to this its distance $p_1$ which is the number of other numbers it passes. Then move $2$ to the left until it is adjacent from the right to $1$ and compute $p_2$ in the analogous way. Continuing on compute $p_i$ for each $i$ and then the length of $w$ is just $\sum p_i$. 
\begin{lemma}
If $w=k_1\dots k_ml_*l_m\dots l_1$ satisfies the spacing condition and $|w|=m^2$ for the even dimensional case, or $|w|=m^2+m$ for the odd dimensional case, then $l_1=k_1-1$.
\end{lemma}
\begin{proof}
Suppose that $l_1<k_1-1$. Then there exist $j>1$ such that $p_j=k_1-1$ sits on position $j$ inside $w$. If $p_j$ sits among the $l$'s,  then construct $\tilde{w}$ by making the transposition $(j,n)$ that interchanges $p_j$ and $l_1$. If $p_j$ sits among the $k$'s  then construct $\tilde{w}$ by interchanging $k_1$ with $p_j$. Observe that $\tilde{w}$ still satisfies the spacing condition and $|\tilde{w}|\le m^2-1$ since in the first case all elements smaller then $k_1-1$ (at least one element, namely $l_1$) do not need to cross over $k_1-1$ anymore. In the second case $p_j$ does not need to cross over $k_1$ anymore and since $p_j=k_1-1$, the elements that need to cross $k_1$ on position $j$ remain the same as  the elements that needed to cross $p_j$ in the initial permutation. But this contradicts the fact that no Schubert variety of dimension less than $m^2$ intersects the cycle.
\noindent
\bigskip
\newline
For the odd dimensional case just add $l_*$ to the representation and consider $|w|=m^2+m$. The only case remaining to be considered is that where $l_*=k_1-1$. In this case we interchange $l_*$ with $l_1$ and observe that this still satisfies the spacing condition and it is of dimension strictly smaller then $m^2+m$. As above this implies a contradiction.
\end{proof}

\begin{thm}
A Schubert variety $S_w$ belongs to $\mathcal{S}_{C_0}$ if and only if $w$ satisfies the double box contraction condition. In particular, in this case $w$ satisfies the spacing condition and $|w|=m^2$, for $n=2m$, and $|w|=m^2+m$. 
\end{thm}

\begin{proof}
We prove the theorem using induction on dimension. The notation $w=(k,l_*,l)$ is used to represent the full sequence  $w=k_1\dots k_ml_*l_m\dots l_1$.
\newline
By the lemma above we see that $l_1$ must be defined by $l_1:=k_1-1$. Remove $l_1$ and $k_1$ from the set $\{1,\dots, n\}$ to obtain a set $\Sigma$ with $n-2$ elements. Define a bijective map  $\phi:\Sigma\rightarrow \{1,2,\dots,n-2\}$ by $\phi(x)=x$ for $x<l_1$ and $\phi(x)=x-2$ for $x>k_1$. By induction one constructs all possible $\tilde{w}=(\tilde{k},\tilde{l}_*,\tilde{l}) \in\Sigma_{n-2}$ using the immediate predecessor algorithm. Each such permutation parametrizes a Schubert variety $S_{\tilde{w}}$ in $S_{\tilde{C}_0}$.
\noindent
\bigskip
\newline
Now return to the original situation by defining for each $\tilde{w}$ a corresponding $w\in \Sigma_n$ with $w(1)=k_1$, $w(n)=l_1$ and $w(i+1)=\phi^{-1}(\tilde{w}(i))$ for $1\le i\le n-2$. It is immediate that $w$ satisfies the double box contraction condition. Hence it remains to compute $|(k,l_*,l)|$.   Let $\tilde{p}_j$ be the distances for the permutation $|(\tilde{k},\tilde{l}_*,\tilde{l})|$. First consider those elements $\varepsilon$ of the full sequence $(k,l_*,l)$ which are smaller than $l_1$ in particular which are smaller than $k_1$. In order to move them to their appropriate position one needs the number of steps $\tilde{p}_\varepsilon$ to do the same for their associated point in $(\tilde{k},\tilde{l}_*,\tilde{l})$ plus $1$ for having to pass $k_1$. Thus in order to compute $|(k,l_*,l)|$ from $|(\tilde{k},\tilde{l}_*,\tilde{l})|$ we must first add $k_1-2$ to the former. Having done the above, we now move $l_1$ to its place directly to the left of $k_1$. This requires crossing $2m+1-(k_1-1)$ larger numbers in the odd dimensional case and $2m-(k_1-1)$ numbers in the even dimensional case. So together we have now added $2m$ in the odd dimensional case and $2m-1$ in the even dimensional case to $|(\tilde{k},\tilde{l}_*,\tilde{l})|$ and $|(\tilde{k},\tilde{l})|$, respectively. All other necessary moves are not affected by the transfer to $|(\tilde{k},\tilde{l}_*,\tilde{l})|$ and back. So for those elements we have $\tilde{p}_\varepsilon=p_\varepsilon$ and it follows that $$|(k,l_*,l)|=|(\tilde{k},\tilde{l}_*,\tilde{l})|+2m=(m-1)^2+(m-1)+2m=m^2+m,$$
in the odd dimensional case and $$|(k,l)|=|(\tilde{k},\tilde{l})|+2m-1=(m-1)^2+2m-1=m^2,$$
in the even dimensional case.
\end{proof}

\begin{cor}
The number of Schubert varieties that intersect the cycle is the double factorial $n!!=(n-1)\cdot(n-3)\cdot\dots\cdot 1$.
\end{cor}

\begin{proof}
Observe that $k_1$ can be arbitrary chosen from $\{2,\dots,n\}$ and once it is chosen $l_1$ is fixed. This amounts to $(n-1)$ possibilities for the placement of $k_1$ and $l_1$. Now remove $l_1$ and $k_1$ from $\{1,\dots, n\}$ to obtain the set $\Sigma$ with $n-2$ elements. As before define the bijective map  $\phi:\Sigma\rightarrow \{1,2,\dots,n-2\}$ by $\phi(x)=x$ for $x<l_1$ and $\phi(x)=x-2$ for $x>k_1$. This gives us the sequence $(\tilde{k},\tilde{l}_*,\tilde{l})$ that satisfies our induction assumption and we thus obtain $(n-2-1)\cdot(n-2-3)\dots1$ Schubert varieties. Returning to our original situation one obtains the desired result. 
\end{proof}
\noindent
The next theorem gives a geometric description in terms of flags of the intersection points. The complete flags describing the intersection points are obtained in the following way from the Weyl group element that parametrizes $S_w$. In the case when $n=2m+1$ and $w=k_1\dots k_m l_*l_m\dots l_1$ the points of intersection are given by the following flags 
\begin{equation}
\label{points}
\begin{gathered}
<(\pm i) e_{l_1}+e_{k_1}>\subset\dots \subset <(\pm i) e_{l_1}+e_{k_1},\dots, (\pm i) e_{l_m}+e_{k_m},e_{l_*}>\subset \\
 <(\pm i) e_{l_1}+e_{k_1},\dots, (\pm i) e_{l_m}+e_{k_m},e_{l_*}, e_{l_m}>\subset \dots \subset \C^n. 
\end{gathered}
\end{equation}
In the case when $n=2m$, the complete flags are given by the same expression with the exception that the span of $e_{l_*}$ is removed from the representation. Of course, one must not forget that the positively oriented flags correspond to one open orbit and the negatively oriented flags to the other, but since this are symmetric with respect to the map that reverses orientation we have the same number of intersection points with the base cycle independently of the chosen open orbit.
\begin{thm}
A Schubert variety $S_w$ in $\mathcal{S}_{C_0}$ intersects the base cycle $C_0$ in $2^m$ points in the case when $n=2m+1$ and $2^{m-1}$ points in the case when $n=2m$. The points are given by \eqref{points}.
\end{thm}

\begin{proof}
Let $w=k_1\dots k_m l_*l_m\dots l_1$  and consider the canonical form of $\mathcal{O}_w$ given by a matrix $[v_1\dots v_n]$ with $$v_i=\alpha^i_{1}e_1+\dots +\alpha^i_{w(i)-1}e_{w(i)-1}+e_{w(i)}, \quad \forall 1\le i\le n.$$ The isotropic condition on flags translates to the fact that each matrix $[v_1\dots v_{n-i}]$ has the column vector $v_i$ perpendicular to itself and all the other vectors in the matrix. Such a matrix exists for each $i\le m$. Start with the initial matrix and for the first step $i=1$ disregard the last column of the initial matrix, for the second step $i=2$ disregard the last and the pre last column and go on until the step $i=m$ is reached. Looking at this process closely and imposing the isotropic conditions will give us an explicit description of all the intersection points.
\noindent
\bigskip
\newline
For the step $j=1$ the vector $v_n= \alpha^n_{1}e_1+\dots + e_{l_1}$, where $l_1=k_1-1$, is disregarded from the matrix. If $l_1=1$ then we are done, because $v_1=\alpha^1_1e_1+e_2$. From $v_1\cdot v_1=0$ it follows that $\alpha^1_1=\pm i$. If $l_1 > 1$, disregarding $v_n$  will create a matrix that has among its columns all vectors that contain a $1$ on entry $p$, where $1\le p<l_1$. Denote such column vectors with $f_p$. Using the relations $v_1\cdot f_1=0, \dots, v_1\cdot(f_{l_1-1})=0$, and computing step by step it follows that $\alpha^1_1=0,\dots, \alpha^1_{l_1-1}=0$. Now the only freedom left is on $\alpha_{l_1}$ and using $v_1\cdot v_1=0$ it follows that $(\alpha^1_{l_1})^2+1=0$. Therefore $$v_1=(\pm i)e_{l_1}+e_{k_1}.$$ The condition $v_1\cdot v_p=0$, for all $2\le p \le n-1$, is equivalent to $\alpha^p_{l_1}\cdot (\pm i)=0$, for all $2\le p \le n-1$, which is further equivalent to $\alpha^p_{l_1}=0$, for all $2\le p\le n-1$. The elements $\alpha^n_p$, for $1\le p \le l_1-1$, are all zero, because in the initial canonical representation of $\mathcal{O}_w$ all columns with a $1$ on the spot $p$, for $1\le p\le l_1-1$, sit before $v_n$ in the matrix and at the right of each such entry the row is completed with zeros.  
\noindent
\bigskip
\newline
For $j=2$, $v_n$ and $v_{n-1}$ are removed from the initial matrix. From the \textit{immediate predecessor} algorithm it follows that either $l_2=k_2-1$ or $l_2=l_1-1$ and $\alpha^{2}_{k_1}=0$. From the previous step $\alpha^{2}_{l_1}=0$. Therefore, even though in this step $v_n$ was removed from the matrix a zero was already created on row $l_1$ in $v_2$ in the previous step. Following the same algorithm as before we create zeros step by step starting with $v_2\cdot f_1=0$ and going further to  $v_2\cdot f_{p}=0$, where $f_p$ is the column where a $1$ sits on row $p$ for all $1\le p\le l_2-1$ exempt of $l_1$ and $k_1$ on which spots the values are already zero. The only freedom that remains is on the spot corresponding to $l_2$. Here, using $v_2\cdot v_2=0$, it follows that $$v_2=(\pm i)e_{l_2}+e_{k_2}.$$
\noindent
The condition $v_2\cdot v_p=0$, for all $3\le p \le n-2$, is equivalent with $\alpha^p_{l_2}\cdot (\pm i)=0$, for all $3\le p \le n-2$ which is further equivalent with $\alpha^p_{l_2}=0$, for all $3\le p\le n-2$. The elements $\alpha^{n-1}_p$, for $1\le p \le l_2-1$, are all zero. The reason for this is that in the initial canonical representation of $\mathcal{O}_w$ all columns with a $1$ on the spot $p$, for $1\le p\le l_2-1$, (with the exception of $p=l_1$ where we have created a zero already) sit before $v_{n-1}$ in the matrix. Moreover, at the right of each such entry the row is completed with zeros. 
\noindent
\bigskip
\newline
Assume that we have shown that $v_{s}=\pm i e_{l_s}+e_{k_s}$, for all $1\le s\le j-1$, with $j\le m$, that $\alpha_{l_s}^p=0$, for all $s+1\le p \le n-s$, and that $\alpha_p^{n-s+1}=0$, for all $1\le p \le l_{s}-1$.
To go to step $j$ one usees the fact that $l_j$ belongs to the set $$\{{k_j-1,l_1-1,\dots, l_{j-1}-1}\}$$ and repeats the procedure. In this case the columns $v_s$ for $n-j+1\le s \le n$ are removed from the matrix. As before zeros are created step by step starting with $v_j\cdot f_1$ and going up to $v_j\cdot f_p$, where $f_p$ is the column where a $1$ sits on row $p$ for all $1\le p \le l_j-1$, except of course when $p \in \{k_1,\dots, k_{j-1},l_{1},\dots {l_{j-1}}\}$ in which case even though the columns corresponding to this elements are not in the matrix their spots in column $j$ were already made zero in the previous step.
\noindent
\bigskip
\newline
It thus follows that the points of intersection $<v_1>\subset\dots \subset \C^n$ are given by the following flags $<(\pm i) e_{l_1}+e_{k_1}>\subset\dots \subset <(\pm i) e_{l_1}+e_{k_1},\dots, (\pm i) e_{l_m}+e_{k_m},e_{l_*}>\subset <(\pm i) e_{l_1}+e_{k_1},\dots, (\pm i) e_{l_m}+e_{k_m},e_{l_*}, e_{l_m}>\subset \dots \subset \C^n $
\end {proof}
\noindent
It thus follows that the homology class $[C_0]$ of the base cycle inside the homology ring of $Z$ is given in terms of the Schubert classes of elements in $\mathcal{S}_{C_0}$ by: $$[C_0]=2^m\sum_{S\in\mathcal{S}_{C_0}}\,[S], \text{ if } n=2m+1, \text{ and }[C_0]=2^{m-1}\sum_{S\in\mathcal{S}_{C_0}}\,[S], \text{ if } n=2m.$$

\section{Results in the measurable case}
This section treats the case when $D$ is an open orbit in $Z=G/P$ and $C_0$ is the base cycle in $D$. Recall the notation
$$\mathcal{S}_{C_0}:=\{S_w \text{ Schubert variety }: dim S_w+dim C_0=dimZ \text{ and } S_w\cap C_0 \ne \emptyset \}.$$
The main idea is to lift Schubert varieties $S_w\in \mathcal{S}_{C_0}$ to minimal dimensional Schubert varieties $S_{\hat{w}}$ in $\hat{Z}:=G/B$ that intersect the open orbit $\hat{D}$ and consequently the base cycle $\hat{C}_0$. 
%$$\mathcal{S}_{\hat{C}_0}:=\{S_{\hat{w}} \text{ Schubert variety }: dim S_{\hat{w}}+dim \hat{C}_0=dim\hat{Z} \text{ and } S_{\hat{w}}\cap \hat{C}_0 \ne \emptyset \}.$$
\noindent
\bigskip
\newline
The first step is to consider the situation when $D$ is a measurable open orbit. 
There are many equivalent ways of defining measurability in general. In our context however, this depends only on the dimension sequence $d$ that defines the parabolic subgroup $P$. Namely, an open $SL(n,\R)$ orbit $D$ in $Z=G/P$ is called \textit{measurable} if $P$ is defined by a symmetric dimension sequence as follows:
\begin{itemize}
\item{$d=(d_1,\dots, d_s,d_s,\dots, d_1)$ or $e=(d_1.\dots, d_s,e',d_s \dots d_1 )$, for $n=2m$, }
\item{$e=(d_1.\dots, d_s,e',d_s \dots d_1)$, for $n=2m+1$.}
\end{itemize}  
A general definition of measurability can be found in \cite{Fels2006}. %and the result used here can be found in \cite{Huckleberry2002}.
\noindent
\bigskip
\newline
As discussed in the preliminaries, a Schubert variety $S_w$ in $Z$ is indexed by the minimal representative $w$ of the parametrization coset associated to the dimension sequence defining $P$. Corresponding to the symmetric dimension sequence $d$, each such $w$ can be divided into blocks $B_j$ and $\tilde{B}_j$, both having the same number of elements $d_j$ for $1\le i \le s$. For example, $B_1=(w(1)<w(2)<\dots <w(d_1))$ and $\tilde{B}_1=(w(d_1+\dots +d_s+d_s+\dots+d_2+1)<\dots <w(n))$. The pairs $(B_j,\tilde{B}_j)$ are called \textbf{symmetric block pairs}. In the case of the symmetric dimension symbol $e$, one single block $B_{e'}$ of length $e'$ is introduced in the middle of $w$. In what follows $w$ will always correspond to a symmetric dimension sequence and it will satisfy the conditions of a minimal representative.

\begin{defi}
A permutation $w$ is said to satisfy the \textbf{generalized spacing condition}, if for each symmetric block pair $(B_j,\tilde{B}_j)$ the elements of $\tilde{B}_j$ can be arranged in such a way that if the elements of $B_j$ are denoted by $k_1^j\dots k_{d_j}^j$ and the rearranged elements of $\tilde{B}_j$ by $l_{1}^j\dots l_{d_j}^j$, then $l_i^j<k_i^j$ for all $1\le i \le d_j$. 
\end{defi}  
\noindent
Observe that in the case of the symmetric dimension sequence $e$ we can always rearrange the elements in the single block $B_{e'}$ so that they satisfy the spacing condition inside the block.

\begin{defi}
A permutation $B_1\dots B_sB_{e'} \tilde{B}_s\dots \tilde{B_1}$ is said to satisfy the \textbf{generalized double box contraction} condition if $w$ is constructed by the \textbf{generalized immediate predecessor algorithm}: 
\newline
 $\bullet$ The first symmetric block pair $(B_1=(k_i^1),\tilde{B}_1=(l_i^1))$ is constructed by choosing $d_1$ pairs $(l_i^1,k_i^1)$ of consecutive numbers from the ordered set $\{1,\dots, n\}$  
  \newline
 $\bullet$ If all symmetric block pairs up to $(B_j=(k_i^{j}),\tilde{B}_j=(l_i^{j}))$ have been chosen, then to go to the step $j+1$ choose ${d_{j+1}}$ pairs $(l_{i}^{j+1},k_i^{j+1})$ in such a way that $l_i^{j+1}$ sits at the immediate left of $k_i^{j+1}$ in the ordered set $\{1,\dots, n\}-\{\cup_{i=1}^jB_i\}-\{\cup_{i=1}^j\tilde{B}_i\}$, for all $1\le i \le d_{j+1}$.
 % These elements are placed in the same way as in the first case in order to build ($B_{i+1},\tilde{B}_{i+1}$).
 %Start by choosing $k_1$ and $l_1:=k_1-1$ from the ordered set $\{1,\dots, n\}$. If we have chosen all the numbers up to $k_i$ and $l_i$ then to go to the step $i+1$ we make a choice of $k_{i+1}$ and $l_{i+1}$ from the ordered set $\{1,\dots, n\}-\{k_1,l_1,\dots,k_i,l_i\}$ in such a way that $l_{i+1}$ sits inside the ordered set at the left of $k_{i+1}$. 
\end{defi}
\noindent
Observe that in the case of the symmetric symbol $e$ the elements in the single block $B_{e'}$ can always be rearranged so that they satisfy the double box contraction condition inside the block. The symbol $B_{e'}$ will always be written in a representation of $w$ and disregarded in the case of the symbol $d$.
\noindent
\bigskip
\newline
If $w$ satisfies the generalized spacing condition, $\tilde{w}$ denotes the permutation obtain from $w$ by replacing each block $\tilde{B}_j=l_1^j\dots l_{d_j}^j$ with a choice of rearrangement of its elements $\tilde{l}_1^j\dots \tilde{l}_{d_j}^j$ required so that $w$ satisfies the generalized spacing condition. Further inside each rearranged block $\tilde{B}_j$ in $\tilde{w}$, $\tilde{l}_1^j\dots \tilde{l}_{d_j}^j$  is rewritten as $\tilde{l}_{d_j}^j \tilde{l}_{d_j-1}^j\dots \tilde{l}_1^j.$ In the case of $B_{e'}$ being part of the representation of $w$ the following rearrangement is chosen : if $e'$ is even then $B_{e'}$ is rearranged as $l'_{e'/2+1}$ $l'_{e'/2+2}\dots l'_{e'}l'_1\dots l'_{e'/2-1}l'_{e'/2}$ and if $e'$ is odd then $B_{e'}$ is rearranged as $l'_{(e'+1)/2+1}$ $l'_{(e'+1)/2+2}\dots l'_{e'}l'_1\dots l'_{(e'+1)/2-1}l'_{(e'+1)/2}$. Observe that now $\tilde{w}$ is a permutation that satisfies the spacing condition for the $G/B$ case and it is of course also just another representative of the coset that parametrizes $S_w$.

\begin{prop}
A Schubert variety $S_w$ parametrized by the permutation $$w=B_1B_2\dots B_sB_{e'}\tilde{B}_s\dots \tilde{B}_2\tilde{B}_1$$ has non-empty intersection with $C_0$ if and only if $w$ satisfies the generalized spacing condition, i.e. if and only if there exists a lift of $S_w$ to a Schubert variety $\tilde{S}_{\tilde{w}}$ that intersects the base cycle in $Z=G/B$.
\end{prop}
\begin{proof}
If $w$ satisfies the generalized spacing condition, then by the above observation one can find another representative $\tilde{w}$ of the parametrization coset of $S_w$, that satisfies the spacing condition and thus a Schubert variety $S_{\tilde{w}}$ that intersects $\hat{C}_0$. Because the projection map $\pi$ is equivariant it follows that $\pi(S_{\tilde{w}})=S_w$ intersects $C_0$.
\newline
Conversely, suppose that $S_w\cap C_0 \ne \emptyset$. Then for every point $p \in S_w\cap C_0 $ there exist $\hat{p}\in S_{\tilde{w}} \cap \hat{C}_0$ with $\pi (\hat{p})=p$ and $\pi(S_{\tilde{w}})=S_w$ for some Schubert variety in $G/B$ indexed by $\tilde{w}$. It follows that $\tilde{w}$ satisfies the spacing condition and $w$ is obtained from $\tilde{w}$ by dividing $\tilde{w}$ into blocks $B_1\dots B_sB_{e'}\tilde{B}_s\dots \tilde{B}_1$ and arranging the elements in each such block in increasing order. This shows that $w$ satisfies the generalized spacing condition.  
\end{proof}
\noindent
If $w=B_{d_1}\dots B_{d_s}\tilde{B}_{d_s}\dots \tilde{B_{1}}$, with $\tilde{B}_{d_j}=l^j_1\dots l^j_{d_j}$ for each $1\le j\le s$, then let $\hat{w}:=B_{d_1}\dots B_{d_s}\tilde{C}_{d_s}\dots \tilde{C}_{d_1}$, where $\tilde{C}_{d_j}=l^j_{d_j}l^j_{d_j-1}\dots l^j_2l^j_1$ for each $1\le j\le s$. If $B_{e'}=l'_1\dots l'_{e'}$ is part of the representation of $w$, then let $\tilde{B}_{e'}$ be the single middle block in the representation of $\hat{w}$ defined by: $l'_{e'/2+1}l'_{e'/2+2}\dots l'_{e'}l'_1\dots l'_{e'/2-1}l'_{e'/2}$ if $e'$ is even and $l'_{(e'+1)/2+1}l'_{(e'+1)/2+2}\dots l'_{e'}l'_1\dots l'_{(e'+1)/2-1}l'_{(e'+1)/2}$ if $e'$ is odd. Call such a choice of rearrangement for $w$ a \textbf{canonical rearrangement}. 
\noindent
\bigskip
\newline
Note that if $S_w\in \mathcal{S}_{C_0}$ lifts to $S_{\hat{w}}$, such that $\hat{w}$ satisfies the double box contraction condition, then $\dim S_{\hat{w}}-\dim S_w=(\dim \hat{Z}-\dim \hat{C}_0)-(\dim Z-\dim C_0)=(\dim \hat{Z}-\dim Z)-(\dim \hat{C_0}-\dim C_0).$ Since $\pi$ is a $G_0$ and $K_0$ equivariant map, if $F$ denotes the fiber of $\pi$ over a base point $z_0\in C_0$, then the fiber of $\pi|_{\hat{C}_0}:\hat{C}_0\rightarrow C_0$ over $z_0$ is just $F\cap \hat{C}_0$ and $\dim S_{\hat{w}}-\dim S_w$ must equal $\dim F - \dim (F\cap \hat{C}_0)$. 
\noindent
\bigskip
\newline
As stated in the preliminaries in the case of $Z=G/B$ and $m=[n/2]$, the isotropic condition on flags is equivalent to $V_i\subset V_i^{\perp}$ for all $1\le i\le m$, $V_m=V_m^\perp$, if $n$ is even and the subspaces $V_{n-i}$ are determined by $V_{n-i}=V_{i}^{\perp},$ for all $1\le i \le m$. Thus in the case of the dimension sequence $d=(d_1,\dots, d_s,d_s,\dots,d_1)$, $\dim F- \dim (F\cap C_0)$ is equal to $2\sum_{i=1}^sd_i(d_i-1)/2-\sum_{i=1}^sd_i(d_i-1)/2$ which is equal to $\sum_{i=1}^sd_i(d_i-1)/2$. In the case when the dimension sequence is given by $e=(d_1,\dots, d_s, e', d_s,\dots, d_1 )$ and the base point $z_0$ contains a middle flag of length $e$, it remains to add to the above number the difference in between the dimension of the total fiber over this flag and the dimension of the isotropic flags in this fiber. This is just a special case of the $G/B$ case for a full flag of length $e'$ and the number is $e'(e'-1)/2-(e'/2)^2+(e'/2)$ which equals to $(e'/2)^2$ when $e'$ is even and $e'(e'-1)/2-[(e'-1)/2]^2$ which equals to $(e'-1)(e'+1)/2$, when $e'$ is odd. 
\noindent
\bigskip
\newline
Now if $w$ satisfies the generalized double box contraction condition, then $\hat{w}$ satisfies the double box contraction condition. %and denote $\hat{\mathcal{S}}_{\hat{C}_0}$ the subset of Schubert varieties in $\mathcal{S}_{\hat{C}_0}$ parametrized by such a $\hat{w}$.
Furthermore, by construction, it is immediate that if $w$ satisfies the generalized double box contraction condition, then $|w|=|\hat{w}|-\sum_{i=1}^sd_i(d_i-1)/2$. That is because the block $\tilde{B}_{d_j}$ is formed by arranging the block $\tilde{C}_{d_j}$ in increasing order and thus crossing $l_1^j$ over $d_{j}-1$ numbers, and more generally $l_i^j$ over $d_{j}-i$ numbers for all $1\le i \le d_j$. Similarly, if $e'$ is part of the representation of $w$ and $e'$ is even then $|w|=|\hat{w}|-\sum_{i=1}^sd_i(d_i-1)/2-(e'/2)^2$ and if $e'$ is odd then $|w|=|\hat{w}|-\sum_{i=1}^sd_i(d_i-1)/2-(e'-1)(e'+1)/4$. 
\begin{thm}
A permutation $w=B_1B_2\dots B_sB_{e'}\tilde{B}_s\dots \tilde{B}_2\tilde{B}_1$ satisfies the generalized double box contraction condition if and only if $w$ parametrizes an Iwasawa-Schubert variety $S_w\in \mathcal{S}_{C_0}$ and the lifting map $f:\mathcal{S}_{C_0}\rightarrow \mathcal{S}_{\hat{C}_0}$ defined by $S_w\mapsto S_{\hat{w}}$, with $\hat{w}$ the canonical rearrangement of $w$, is injective.
\end{thm}
%\begin{thm}
%A Schubert variety $S_w$ parametrized by the permutation $$w=B_1B_2\dots B_sB_{e'}\tilde{B}_s\dots \tilde{B}_2\tilde{B}_1$$ belongs to $\mathcal{S}_{C_0}$  if and only if $w$ satisfies the generalized double box contraction condition, i.e. if and only if $S_w$ lifts to $S_{\hat{w}}$ in $\mathcal{S}_{\hat{C}_0}$, where $\hat{w}$ is the canonical rearrangement of $w$,  and a one-to-one correspondence $S_{C_0} \leftrightarrow \hat{S}_{\hat{C}_0} exists.$\end{thm}
\begin{proof}
If $w$ satisfies the generalized double box contraction condition, then by the above observation the canonical rearrangement $\hat{w}$ of $w$ is just another representative of the parametrization coset of $S_w$. Moreover, $\hat{w}$ satisfies the double box contraction condition and thus parametrizes a Schubert variety $S_{\hat{w}}$ that intersects $\hat{C}_0$. Because the projection map $\pi$ is equivariant it follows that $\pi(S_{\hat{w}})=S_w$ intersects $C_0$. From the remarks before the statement of the theorem it follows that if $w$ satisfies the generalized double box contraction condition, then the difference in dimensions in between $S_{\hat{w}}$ and $S_w$ is achieved, and this happens when one needs to write the elements of each block $\tilde{B_i}$ in strictly decreasing order to form the block $\tilde{C}_j$.
\noindent
\bigskip
\newline
Conversely, suppose that $S_w\in \mathcal{S}_{C_0}$ but $w$ does not satisfy the generalize double box contraction condition. It then follows that there exists a first block pair $(B_j,\tilde{B}_j)$ and a first pair $(k_i^j,l_i^j)$, for some $i$ in between $1$ and $d_j$ such that $l_{i-1}^{j}$ sits at the immediate left of $l_{i}^j$ and $k_{i}^j$ sits at the immediate right of $k_{i-1}^j$ in the ordered set $\{1,\dots, n\}-\{\cup_{s=1}^{j-1}B_s\}-\{\cup_{s=1}^{j-1}\tilde{B}_s\}$. This means that when $w$ is lifted to $\hat{w}$, the place of $l_{i-1}^j$ and $l_i^j$ remain the same because otherwise $w$ will not satisfy the double box contraction condition. But this implies that the difference in between the length of $\hat{w}$ and the length of $w$ is strictly smaller than what the difference $\dim S_{\hat{w}}-\dim S_{w}$ should be.

%Then the only way $\pi(S_{\hat{w}})$ will belong to $\mathcal{S}_{C_0}$ is if the difference in between the dimension of $S_{\hat{w}}$ and the dimension of $S_w$ is exactly the dimension of the fiber $F$ of $\pi|_{C_0}$ over a base point $z_0\in C_0$.  
\end{proof}
%
%\begin{cor}
%The cardinality of $\mathcal{S}_{C_0}$ depends only on the dimension sequence that defines the parabolic subgroup $P$ and is given by $$m_d:=(n-2d_1+1)\cdot (n-2d_1-2d_2+1)\dots (n-2d_1-\dots-2d_s+1).$$ 
%\end{cor}
%\begin{proof}
%Observe that there are $(n-2d_1+1)$ ways two choose $2d_1$ consecutive numbers from the ordered set $\{1,\dots, n\}$ to form the first symmetric block pair $(B_1,\tilde{B}_1)$. Then the result follows by induction.  
%\end{proof}
%\noindent
%It thus follows that the homology class $[C_0]$ of the base cycle inside the homology ring of $Z$ is given in terms of the Schubert classes of elements in $\mathcal{S}_{C_0}$ by: $$[C_0]=m_d\sum_{S\in\mathcal{S}_{C_0}}\,[S].$$
\begin{thm}
A Schubert variety $S_w$ in $S_{C_0}$ intersects the base cycle $C_0$ in $2^{d_1+\dots + d_s}$ points in the case where $w$ is given by a symmetric symbol $e=(d_1,\dots,d_s,e,d_s,$ $\dots,d_1)$ and $n=2m+1$, while in the even dimensional case we have $2^{m-1}$  points in the case $d=(d_1,\dots, d_s,d_s,\dots,d_1)$ and $2^{d_1+\dots d_s-1}$ in the case $e=(d_1,\dots,d_s,e,d_s,$ $\dots,d_1)$. 
\end{thm}

\begin{proof}
The result follows from the lifting principle, because the intersection points of $S_w$ can be identified in a one-to-one manner with a subset of the intersection points of $S_{\hat{w}}$. Thus the cardinality of this subset can be directly computed.
\end{proof}
\noindent
It thus follows that the homology class $[C_0]$ of the base cycle inside the homology ring of $Z$ is given in terms of the Schubert classes of elements in $\mathcal{S}_{C_0}$ by: $$[C_0]=2^{m-1}\sum_{S\in\mathcal{S}_{C_0}}\,[S], \text{ if } n=2m, \text{ and } Z=Z_d,$$ $$[C_0]=2^{d_1+\dots +d_s-1}\sum_{S\in\mathcal{S}_{C_0}}\,[S], \text{ if } n=2m, \text{ and } Z=Z_e,$$  $$[C_0]=2^{d_1+\dots +d_s}\sum_{S\in\mathcal{S}_{C_0}}\,[S], \text{ if } n=2m+1, \text{ and } Z=Z_e.$$

\section{Results in the non-measurable case}
The last case to be considered is the case of $Z=G/P$ and $D\subset Z$ a non-measurable open orbit, i.e. the dimension sequence $f=(f_1,\dots, f_u)$ that defines $P$ is not symmetric. Associated to the flag domain $D$ there exists its measurable model, a canonically defined measurable flag domain $\hat{D}$ in $\hat{Z}=G/\hat{P}$ together with the projection map $\pi: \hat{Z} \rightarrow {Z}$. If $\hat{z}_0$ is a base point in $\hat{C}_0$, $z_0=\pi(\hat{z}_0)$ and we denote by $\hat{F}$ the fiber over $z_0$, by $H_0$ and $\tilde{H}_0$ the isotropy of $G_0$ at $z_0$ and $\hat{z}_0$ respectively, then $\hat{F}\cap\hat{D}=H_0/\hat{H}_0$, is holomorphically isomorphic with $\mathbb{C}^k$, where $k$ is root theoretically computable. 
\noindent
\bigskip
\newline
Moreover, if one considers the extension of the complex conjugation $\tau$ from $\C^n$ to $\SLC$, defined by $\tau(s)(v)=\tau(s(\tau(v)))$, for all $s\in \SLC$, $v\in \C^n$, then one obtains an explicit construction of $\hat{P}$ as follows. Consider the Levi decomposition of $P$ as the semidirect product $P=L\rtimes U$, where $L$ denotes its Levi part and $U$ the unipotent radical. If $P^{-}=L\rtimes U^{-}$ denotes the opposite parabolic subgroup to P, namely the block lower triangular matrix with blocks of size $f_1,\dots f_u$, then $\hat{P}= P\cap \tau(P^{-})$. Furthermore, the parabolic subgroup $\tau(P^{-})$ has dimension sequence $(f_u,\dots, f_1)$. These results are proved in complete generality in \cite{Wolf1995}.
\noindent
\bigskip
\newline
Because $\pi$ is a $K_0$ equivariant map, the restriction of $\pi$ to $\hat{C}_0$ maps $\hat{C}_0$ onto $C_0$ with fiber $\hat{C}_0\cap(\hat{F}\cap\hat{D})$. In this case this fiber is a compact analytic subset of $\mathbb{C}^k$ and consequently it is finite, i.e. the projection $\pi|_{\hat{C}_0}:\hat{C}_0\rightarrow C_0$ is a finite covering map. Because $C_0$ is simply-connected it follows that:
\begin{prop}
The restriction map $$\pi|_{\hat{C}_0}:\hat{C}_0\rightarrow C_0$$
is biholomorphic. In particular, if $q$ and $\hat{q}$ denote the respective codimensions of the cycles, it follows that $\hat{q}=dim(\hat{F})+q$
\end{prop}   
\noindent
If we denote with $\mathcal{S}_{C_0}$ the set of Schubert varieties $S_w$ in $Z$ that intersect the base cycle $C_0$ and $\dim S_w+\dim C_0=\dim Z$, where $w$ is a minimal representative of the parametrization coset of $S_w$ and by $\mathcal{S}_{\hat{C}_0}$ the analogous set in $\hat{Z}$, then the above discussion implies the following:
\begin{prop}
The map $\Phi: S_{C_0}\rightarrow \pi^{-1}(S_{C_0})\subset S_{\hat{C}_0}$ is bijective.
\end{prop}
\noindent
If $d=(d_1,\dots,d_s,d_s,\dots,d_1)$ or $e=(d_1,\dots,d_s,e',d_s,\dots,d_1)$ is a symmetric dimension sequence, then one can construct another dimension sequence out of it, not necessarily symmetric, by the following method. Consider an arbitrary sequence $t=(t_1,\dots,t_p)$ such that each $t_i\ge 1$ for all $1\le i\le p$, at least one $t_i$ is strictly bigger than $1$ and $t_1+\dots +t_p=2s$ or $2s+1$ depending on wether one considers $d$ or $e$, respectively. Associated to $t$ the sequence $\delta=(\delta_1,\dots , \delta_p)$ is defined by $\delta_{j}:=\sum_{i=1}^{j}\,t_i$. With the use of $\delta$ the new dimension sequence $f_{\delta}=(f_{\delta_{1}},\dots,f_{\delta_{p}})$ is defined by $f_{\delta_{1}}:=\sum_{i=1}^{\delta_{1}}d_i$,   $$f_{\delta_{j}}:=\sum_{i=\delta_{j-1}+1}^{\delta_{j}}\,d_i, \text{ for all } 2\le j \le p.$$ 
\noindent
\bigskip
\newline
Because $\hat{P}$ is obtained as the intersection of two parabolic subgroups $P$ and $\tau(P^{-})$, it follows that the dimension sequence $f$ of $P$ is obtained as above, from the dimension sequence of $\hat{P}$, as $f_{\delta}$ for a unique choice of $t$. For ease of computation we do not break up anymore the dimension sequence of $\hat{P}$ into its symmetric parts and we simply write it as $d=(d_1,\dots,d_s)$, where $s$ can be both even or odd. 
Using the usual method of computing the dimension of $Z$ it then follows that $$\dim Z=\dim {\hat{Z}}-\sum_{t_j>1}\sum_{\delta_{j-1}+1\le h<g\le \delta_j}\,d_hd_g.$$ For example if $P$ corresponds to the dimension sequence $(2,4,3)$, then an easy computation with matrices shows that $\hat{P}$ corresponds to the dimension sequence $(2,1,3,1,2)$, $t=(1,2,2)$ and $\delta=(1,3,5)$. Moreover, $\dim Z=\dim{\hat{Z}}-1\cdot3-1\cdot2$.
\noindent
\bigskip
\newline
Given the sequence $f=f_{\delta}$, we are now interested in describing the set $\mathcal{S}_{C_0}$. Let $S_{\hat{w}}$ in  $\mathcal{S}_{\hat{C}_0}$ be the unique Schubert variety associated to a given $S_w \in \mathcal{S}_{C_0}$ such that  $\pi (S_{\hat{w}})=S_{w}$.  If $\hat{w}$ is given in block form by $B_1\dots B_s$, where here again the notation used does not take into consideration the symmetric structure of $\hat{w}$, then $w$ is given in block form by $C_1\dots C_p$ corresponding to the dimension sequence $f_\delta$. The blocks $C_j$ are given by $C_1=\bigcup_{i=1}^{\delta_{1}}\, B_{d_i}$ and $$C_j=\bigcup_{i=\delta_{j-1}+1}^{\delta_j}\, B_{d_i}, \text{ for all } 2\le j\le p,$$ arranged in increasing order. Moreover,
\begin{align*}
\dim S_w&=\dim Z-\dim C_0=\dim Z-\dim{\hat{C}_0} \\
&=\dim{\hat{Z}}-\sum_{t_j>1}\sum_{\delta_{j-1}+1\le h<g\le \delta_j}\,d_hd_g -\dim{\hat{C}_0} \\
&=\dim{S_{\hat{w}}}-\sum_{t_j>1}\sum_{\delta_{j-1}+1\le h<g\le \delta_j}\,d_hd_g.
\end{align*}
Finally, understanding what conditions $\hat{w}$ satisfies in order for the above equality to hold amounts to understanding the difference in length that the permutation $\hat{w}$ looses when it is transformed into $w$. If $C_{j}$ contains only one $B$-block from $\hat{w}$, i.e. $t_j=1$, then this is already ordered in increasing order and it does not contribute to the decrease in dimension. If $C_{j}$ contains more $B$-blocks, say $$C_j=\bigcup_{i=\delta_{j-1}+1}^{\delta_j}\, B_{d_i},$$ we start with the first block $B_{\delta_{j-1}+1}$ which is already in increasing order and bring the elements from $B_{\delta_{j-1}+2}$ to their correct spots inside the first block. As usual, to each number we can associate a distance, i.e. the number of elements it needs to cross  in order to be brought to the correct spot, and we denote by $\alpha_{\delta_{j-1}+2}$ the sum of these distances. The maximum value that $\alpha_{\delta_{j-1}+2}$ can attain is that when all the elements in the second block are smaller than each element in the first block, i.e. the last element in the second block is smaller than the first element in the first block. In this case $\alpha_{\delta_{j-1}+2}=d_{\delta_{j-1}+1}d_{\delta_{j-1}+2}$, the product of the number of elements in the first block with the number of elements in the second block. Next we bring the elements in the $3^{rd}$ block among the already ordered elements from the first and second block. Observe that the maximal value that $\alpha_{\delta_{j-1}+3}$ can attain is $d_{\delta_{j-1}+1}d_{\delta_{j-1}+3}+d_{\delta_{j-1}+2}d_{\delta_{j-1}+3}$ when the last element in the $3^{rd}$ block is smaller than all elements in the first two blocks. In general we say that the group of blocks used to form  $C_j$ is in \textbf{strictly decreasing order} if it satisfies the following: the last element of block $B_{i+1}$ is smaller than the first element of block $B_i$ for all $\delta_{j-1}+1\le i \le \delta_j-1$. Consequently, if one wants to order this sequence of blocks into increasing order one needs to cross over $$\sum_{\delta_{j-1}+1\le h<g\le \delta_j}\,d_hd_g$$ elements. Thus if all the blocks $C_j$ with $t_j>1$ among $w$ come from groups of blocks arranged in strictly decreasing order in $\hat{w}$, then $$|w|=|\hat{w}|-\sum_{t_j>1}\sum_{\delta_{j-1}+1\le h<g\le \delta_j}\,d_hd_g.$$
As a consequence we have the following observation: 
\begin{prop}
A Schubert variety $S_{\hat{w}}\in \mathcal{S}_{\hat{C}_0}$ is mapped under the projection map to a Schubert variety $S_{w}\in \mathcal{S}_{C_0}$ if and only if all the blocks $C_j$ with $t_j>1$ among $w$ come from groups of blocks arranged in strictly decreasing order in $\hat{w}$.
\end{prop}
\noindent
As an example consider the complex projective space $Z=\mathbb{P}_5$. The dimension sequence of the measurable model in this case is given by $d=(1,4,1)$ and the Schubert varieties in $\mathcal{S}_{\hat{C}_0}$ are parametrized by the following permutations: $(2)(3456)(1)$, $(3)(1456)(2)$, $(4)(1256)(3)$, $(5)(1236)(4)$, $(6)(1234)(5)$. The only permutation that satisfies the strictly decreasing order among the last two blocks is the permutation $(2)(3456)(1)$ and this gives the only Schubert variety in $\mathcal{S}_{C_0}$ parametrized by $(2)(13456)$. More generally, for $Z=\mathbb{P}_n$ we have only one element in $\mathcal{S}_{C_0}$ parametrized by $(2)(13\dots n+1)$. 
\noindent
\bigskip
\newline
For a slightly more complicated example let $P$ be given by the dimension sequence $(3,3,2)$. Then the dimension sequence for the measurable model is given by $(2,1,2,1,2)$. Table \ref{table2} gives a list of the $36$ Weyl group elements that parametrize the Schubert varieties in $\mathcal{S}_{\hat{C}_0}$. Only $3$ elements satisfy the strictly decreasing order for the first two blocks and the third and forth block, namely $(57)(2)(38)(1)(46)$, $(58)(2)(36)(1)(47)$, $(68)(2)(34)(1)(57)$. The $3$ Schubert varieties in $\mathcal{S}_{C_0}$ are thus parametrized by $(257)(138)(46)$, $(258)(136)(47)$ and $(268)(134)(57)$.
\begin{table}
\caption{The elements $w\in \Sigma_8$ which satisfy the generalised double box contraction condition for the dimension sequence (2,1,2,1,2).}
\vspace*{0.5cm}
\begin{minipage}{0.5\textwidth}
		\begin{center}
		\begin{tabular}{cccccccc}
			\hline (24) & (6) & (78) & (5) & (13) \\
			\hline (24) & (7) & (58) & (6) & (13) \\ 
			\hline (24) & (8) & (56) & (7) & (13) \\ 
			\hline (25) & (6) & (78) & (3) & (14) \\ 
			\hline (25) & (7) & (38) & (6) & (14) \\ 
			\hline (25) & (8) & (36) & (7) & (14) \\ 
			\hline (26) & (4) & (78) & (3) & (15) \\
			\hline (26) & (7) & (38) & (4) & (15) \\
			\hline (26) & (8) & (34) & (7) & (15) \\
			\hline (27) & (4) & (58) & (3) & (16) \\
			\hline (27) & (5) & (38) & (4) & (16) \\
			\hline (27) & (8) & (34) & (5) & (16) \\
			\hline (28) & (4) & (56) & (3) & (17) \\
			\hline (28) & (5) & (36) & (4) & (17) \\
			\hline (28) & (6) & (34) & (5) & (17) \\
			\hline (35) & (6) & (78) & (1) & (24) \\
			\hline (35) & (7) & (18) & (6) & (24) \\
			\hline (35) & (8) & (21) & (7) & (24) \\ 
			\hline
		\end{tabular} 
		\end{center}
		
	\end{minipage}
	\begin{minipage}{0.5\textwidth}
	\begin{center}
			\begin{tabular}{cccccccc}
			\hline (36) & (4) & (78) & (1) & (25) \\
			\hline (36) & (7) & (18) & (4) & (25) \\ 
			\hline (36) & (8) & (14) & (7) & (25) \\ 
			\hline (37) & (4) & (58) & (1) & (26) \\ 
			\hline (37) & (5) & (18) & (4) & (26) \\ 
			\hline (37) & (8) & (14) & (5) & (26) \\ 
			\hline (38) & (4) & (56) & (1) & (27) \\
			\hline (38) & (5) & (16) & (4) & (27) \\
			\hline (38) & (6) & (14) & (5) & (27) \\
			\hline (57) & (2) & (38) & (1) & (46) \\
			\hline (57) & (3) & (18) & (2) & (46) \\
			\hline (57) & (8) & (12) & (3) & (46) \\
			\hline (58) & (2) & (36) & (1) & (47) \\
			\hline (58) & (3) & (16) & (2) & (47) \\
			\hline (58) & (6) & (12) & (3) & (47) \\
			\hline (68) & (2) & (34) & (1) & (57) \\
			\hline (68) & (3) & (14) & (2) & (57) \\
			\hline (68) & (4) & (12) & (3) & (57) \\ 
			\hline
		\end{tabular}
	\end{center}
	
	\end{minipage}	
\label{table2}	
\end{table}
\chapter{The case of the real form $\protect\SLH$}
\section{Preliminaries}
Let us consider $\C^n$ equipped with the quaternion structure $\mathbf{j}: v\mapsto J\bar{v}$, $J= \begin{pmatrix}
0 & I_m \\
-I_m & 0
\end{pmatrix}
.$ Then $G_0=SL(m,\H):=\{X\in SL(2m,\C): \,\mathbf{j} X=X \mathbf{j}\}$, n=2m, is the group of complex linear transformations of determinant $1$ that commute with $\mathbf{j}$. Equivalently, $G_0$ is the real form of $\SLC$ associated with the real structure $\tau: \SLC\rightarrow \SLC$ given by $\tau(X)=\mathbf{j} X \mathbf{ j}^{-1}$. If $\theta: \SLC\rightarrow \SLC$ is the Cartan involution defined by $\theta(X)=(\overline{X}^{t})^{-1}$ with fixed point set $SU(2m)$, then $\theta$ commutes with $\mathbf{j}$ and so it restricts to a Cartan involution on $SL(2m,H)$. Consequently, $K_0=SU(2m)\cap SL(2m,\H)=Sp(m)$ is the unitary symplectic group and its complexification $K$ is the complex symplectic group $Sp(m,\C)$.  
\noindent
\bigskip
\newline
In this case $G_0$ has only one open orbit in $Z=G/P$, where $P$ corresponds as usual to a dimension sequence $d=(d_1,\dots, d_s)$. The unique open orbit is described in terms of flags as the set of $\mathbf{j}-$\textit{generic} flags $(0\subset V_1 \subset \dots \subset V_s=\C^n)$, i.e. those flags that satisfy $\dim (V_i\cap \mathbf{j} (V_j))$ is minimal, for all $1\le i,j\le s$. When $P=B$, the condition of $\mathbf{j}-$genericity is equivalent to $\mathbf{j}(V_i)\oplus V_{2m-i}=\C^n$, for all $1\le i \le m$.
\noindent
\bigskip
\newline
In the case of $Z=G/B$ the base cycle $C_0$ is the compact complex submanifold of $D$ represented in the form $$C_0=K.z_0\cong K/(K\cap B_{z_0}),$$ for a base point $z_0$ in D. From basic computations with roots it follows that $dim\,C_0=m^2$. Thus the dimension of the corresponding Schubert varieties of interest is $m^2-m$.
In general, consider $\omega$ to be the standard symplectic structure on $\C^n$ defined by $\omega(v,w)=v^tJw.$ The base cycle $C_0$ is the set of $\omega-$ isotropic flags, namely $V: 0\subset V_1\subset\dots\subset V_{s}= \mathbb{C}^n$ such that $dim\,(V_i\cap V_j^\perp)$ is maximal, for all $1\le i,j\le s$, where $\perp$ is considered with respect to $\omega$.
\noindent
\bigskip
\newline
The above results can be found in \cite{Huckleberry2002}. The same paper contains the description of the Iwasawa Borel subgroup $B_{I}$ as the isotropy subgroup that fixes the flag: $$<e_1>\subset<e_1,j(e_1)>\subset<e_1,j(e_1),e_2>\subset <e_1,j(e_1),e_2,j(e_2)>\subset\dots \subset\mathbb{C}^n,$$ where $j(e_i)=-e_{m+i},\,\forall\, 1\le i\le m.$ 
\noindent
\bigskip
\newline
Note that the subgroup of diagonal matrices in $G$ is a torus for both the standard basis and for the basis used to describe the Iwasawa-Borel subgroup. Moreover, recall that once an ordered basis is fixed, then one obtains an isomorphism of the Weyl group of $G$ with respect to $T$ and the symmetric group $\Sigma_n$. 
%It is convenient to differentiate in terms of notation between the Weyl group of $G$ with respect to $T$, given by the isomorphism associated to the standard basis in $\C^n$, denoted by $\Sigma_n^{St}$ and the Weyl group of $G$ with respect to $T$ given by the isomorphism associated to the basis that defines the Iwasawa-Borel subgroup, denoted by $\Sigma_n^{I}$. Of course, this two groups are isomorphic and a particular useful isomorphism in this situation is $\gamma_{I}:\Sigma_n^I\rightarrow \Sigma_n^{St}$, $i\mapsto (i+1)/2$, if $i$ is odd and $i\mapsto m+i/2$, if $i$ is even.
\noindent
\bigskip
\newline
%Throughout this chapter we mostly work with $\Sigma_n^{I}$, with only one exception when working with $\Sigma_n^{St}$ turns out to be easier. 
We use the basis associated to the Iwasawa-Borel subgroup $B_I$ to define this isomorphism and consider the canonical form of an Iwasawa-Schubert cell in this basis as defined in the preliminaries of the $SL(n,\R)$-case.

\section{Introduction to the combinatorics}
The next two sections give a full description of the $B_{I}$-Schubert varieties $S$ of complementary dimension to $C_0$ with $\mathcal{S}_{C_0}\ne \emptyset$, the points of intersection and their number, in the case of the open $SL(m,H)$-orbit $D$ in $Z$.
The first case to be considered is the case of $Z=G/B$, where $$\mathcal{S}_{C_0}:=\{S_w \text{ Schubert variety }: dim S_w+dim C_0=dimZ \text{ and } S_w\cap C_0 \ne \emptyset \}.$$ In what follows we describe the conditions that the element $w$ of the Weyl group that parametrizes the Schubert variety $S_w$ must satisfy in order for $S_w$ to be in $S_{C_0}$. For convenience of the reader we recall the background results as stated in the $\SLR$-case. One of the main ingredients for this is the fact that $S_w\cap D \subset \mathcal{O}_w $ and the fact that if $S_w\cap D\ne \emptyset$, then $S_w\cap C_0\ne \emptyset$. Moreover, no Schubert variety of dimension less than the codimension of the base cycle intersects the base cycle. These are general results that can be found in \cite[p.101-104]{Fels2006}

\begin{defi}
A permutation $w=k_1\dots k_ml_m\dots l_1 \in \Sigma_n$ is said to satisfy the spacing condition if $l_i<k_i$ or $k_i$ is an odd number and $l_i=k_i+1$, for all $1\le i \le m.$
\end{defi}
\noindent
Note that the spacing condition in this case is actually the spacing condition for the $\SLR$-case together with one degree of freedom, namely that it is allowed that $l_i=k_i+1$ for $k_i$ odd.
\begin{defi}
A permutation $w=k_1\dots k_ml_m\dots l_1 \in \Sigma_n$ is said to satisfy the strictly pairing condition if $k_i$ is an odd number, $l_i$ is even and $l_i=k_i+1$, for all $1\le i\le m$.
\end{defi}
\noindent
Here as well, the strictly pairing condition implies the spacing condition but not conversely. For example, $315624$ satisfies both conditions while $325614$ does not satisfy the strictly pairing condition. Observe that here, in contrast with the $SL(n,\R)$-case, the condition is much stronger than the double box contraction condition. It is not enough that the pairs $(k_i,l_i)$ are in consecutive order in the set $\{1,\dots, n\}-\{k_1,l_1,\dots, k_{i-1},l_{i-1}\}$ but here they are actually consecutive numbers.

\section{Main results for $Z=G/B$}
The first result gives the condition for a Schubert variety to intersect the base cycle, independent of its dimension.
\begin{prop}
A Schubert variety $S_w$ corresponding to a permutation $$w=k_1\dots k_ml_m\dots l_1$$ has nonempty intersection with $C_0$ if and only if $w$ satisfies the spacing condition.
\end{prop}

\begin{proof}
Similar to the $SL(n,\R)$ case, the key fact in this proof is that if a Schubert variety intersects the open orbit $D$, then it also intersects the base cycle $C_0$. Moreover, $S_w\cap D\subset \mathcal{O}_w$. The arguments bellow provide a proof of the fact that $\mathcal{O}_w$ (the canonical matrix form in the basis that defined the Iwaswa-Borel subgroup) contains a $\mathbf{j}$-generic point if and only if $w$ satisfies the spacing condition.
\noindent
\bigskip
\newline
First assume that $w$ satisfies the spacing condition. Under this assumption we need to prove that $$\mathbf{j}(V_i)\oplus V_{2m-i}=\mathbb{C}^n\quad \forall i\le m,$$ where $V$ is an arbitrary flag in $\mathcal{O}_w$. This is equivalent to showing that the matrix formed from the vectors generating $\mathbf{j}(V_i)$ and $V_{2m-i}$ has maximal rank . Form the following pairs of vectors $(v_i,\mathbf{j}({v}_i))$ and the matrices $$[v_1\mathbf{j}({v}_1)\dots v_i\mathbf{j}({v}_i)v_{i+1}\dots v_{2m-i}],\quad \forall i\le m.$$ These are the matrices corresponding to $$\mathbf{j}(V_i)=<\mathbf{j}({v}_1),\dots , \mathbf{j}({v}_i)>$$ and $$V_{2m-i}=<v_1,\dots, v_i, v_{i+1},\dots v_{2m-i}>.$$ We carry out the following set of operations keeping in mind that the rank of a matrix is not changed by row or column operations. The initial matrix is the canonical matrix representation of the Schubert cell $\mathcal{O}_w$. For the step $j=1$, $l_1=k_1+1$ with $k_1$ odd or $l_1=k_1-1$ with $k_1$ even and the last column corresponding to $l_1$ is eliminated from the initial matrix. 
\noindent
\bigskip
\newline
Next denote by $c_h$ the $h^{th}$ column of this matrices and obtain the following:
\newline
%On $c_2$ we create a $1$ by normalizing on row $l_1$ (since the last vector $v_{2m}$ from the initial matrix is not there anymore, we have no other column with a $1$ on position $l_1$). We use the new created $1$ to make $0$'s on $c_2$ below row $l_1$ and to make zeros on the left of this $1$ (on row $l_1$ on column $j>2$). In this process we make sure that after each operation (addition and multiplication by scalar of columns) a normalization is done on the resulting row or column in case, on the spot in the matrix where a $1$ existed before.  
If $k_1$ is odd, then $v_1$ is a vector of the form $$v_1=\alpha_1^1e_1+\alpha_2^1\mathbf{j}(e_1)+\dots +\alpha_{k_1-1}^1\mathbf{j}(e_{(k_1-1)/2})+e_{(k_1+1)/2},$$
and $$\mathbf{j}(v_1)=(\bar{\alpha}_2^1)e_1+(\bar{\alpha}_1^1)\mathbf{j}(e_1)+\dots + (\bar{\alpha}_{k_1-2}^1)\mathbf{j}(e_{(k_1-1)/2})+\mathbf{j}(e_{(k_1+1)/2}).$$ 
\noindent
If $l_1=k_1+1$, then it follows that $c_2$ has a leading coefficient $1$ on row $l_1$. Because the last column was removed from the initial matrix representation of $\mathcal{O}_w$ this is the only column in the new matrix that has its leading coefficient on row $l_1$. It remains to create a zero on $c_2$ on row $k_1$ and zeros on $c_j$, for each $j>2$, on row $l_1$. For the first point subtract from $c_2$ a suitable multiple of $c_1$ and leave the result on $c_2$. For the second point, it is enough to consider those columns which have $1$'s on rows greater than $l_1$, because the columns with $1$'s on rows smaller than $l_1$ already have zeros bellow them. Finally subtract from these columns suitable multiplies of $c_2$. We thus create a matrix that represents points in the Schubert variety $k_1l_1k_2\dots k_ml_m\dots l_2$. This clearly has maximal rank.
\noindent
\bigskip
\newline
If $l_1<k_1$, then in the second column $c_2$ zeros are created in all rows starting with $k_1$ and going down to row $l_1+1$. This is done by subtracting suitable multiplies of $c_2$ from the columns in the matrix having a $1$ in these rows and putting the result in $c_2$. For example, to create a zero in the spot corresponding to $k_1$ the second column is subtracted from the first column and the result is left in the second column. in row $l_1$ a $1$ is created by normalizing. This $1$ is the only $1$ in row $l_1$, because the last column that contained a $1$ in row $l_1$ was removed from the matrix. We now want to create zeros in $c_j$ for $j>2$ in row $l_1$. This is done as in the case when $l_1=k_1+1$. We thus create a matrix that represents points in the Schubert variety $k_1l_1k_2\dots k_ml_m\dots l_2$. This clearly has maximal rank.
\noindent
\bigskip
\newline
If $k_1$ is even, then $v_1$ is a vector of the form $$v_1=\alpha_1^1e_1+\alpha_2^1\mathbf{j}(e_1)+\dots +\alpha_{k_1-1}^1e_{k_1/2}+\mathbf{j}(e_{k_1/2}),$$
and $$\mathbf{j}(v_1)=(\bar{\alpha}_2^1)e_1+(\bar{\alpha}_1^1)\mathbf{j}(e_1)+\dots + e_{k_1/2}+(\bar{\alpha}_{k_1-1}^1)\mathbf{j}(e_{k_1/2}).$$ 
\noindent
First create a zero in $c_2$ in row $k_1$ by subtracting a suitable multiple of $c_1$ from $c_2$ and leaving the result in $c_2$. Then follow the same kind of elementary transformations as in the case when $k_1$ is odd and $l_1<k_1$ to create a matrix that represents points in the Schubert variety $k_1l_1k_2\dots k_ml_m\dots l_2$.
\noindent
\bigskip
\newline
Assume by induction that we have created a maximal rank matrix corresponding to points in the Schubert variety $k_1l_1k_2l_2\dots k_jl_jk_{j+1}\dots k_ml_m\dots l_{j+1}$. To go to the step $j+1$, remove from the matrix the last column corresponding to $l_{j+1}$ and in between add $c_{2j+1}$ and $c_{2j+2}$ the conjugate of $c_{2j+1}$ and reindex the columns. 
\noindent
\bigskip
\newline
If $k_{j+1}$ is odd, then $v_{j+1}$, the vector in column $c_{2j+1}$ is a vector of the form $$v_1=\alpha_1^{j+1}e_1+\alpha_2^{j+1}\mathbf{j}(e_1)+\dots +\alpha_{k_{j+1}-1}^{j+1}\mathbf{j}(e_{(k_{j+1}-1)/2})+e_{(k_{j+1}+1)/2}$$
and $$\mathbf{j}(v_1)=(\bar{\alpha}_2^{j+1})e_1+(\bar{\alpha}_1^{j+1})\mathbf{j}(e_1)+\dots + (\bar{\alpha}_{k_{j+1}-2}^{j+1})\mathbf{j}(e_{(k_{j+1}-1)/2})+\mathbf{j}(e_{(k_{j+1}+1)/2}).$$ 
\noindent
Thus, if $l_{j+1}=k_{j+1}$, the column $c_{2j+2}$ has the leading coefficient $1$ in row $l_{j+1}$. Because the column corresponding to $l_{j+1}$ was removed from the matrix associated to $$k_1l_1k_2l_2\dots k_jl_jk_{j+1}\dots k_ml_m\dots l_{j+1},$$ $c_{2j+1}$ is the only column with leading coefficient in row $l_{j+1}$. To create a zero in column $c_{2j+2}$ in row $k_{j+1}$, subtract a suitable multiple of column $c_{2j+1}$ from column $c_{2j+2}$ and leave the result in $c_{2j+2}$. Perform the same kind of elementary transformations as in the case of $j=1$ to create zeros in each column $c_p$, for all $p>2j+2$, in row $l_{j+1}$. We thus create a matrix that represents points in the Schubert variety $k_1l_1k_2l_2\dots k_jl_jk_{j+1}l_{j+1}k_{j+2}\dots k_ml_m\dots l_{j+2}$.
\noindent
\bigskip
\newline
If $k_{j+1}$ is even, then $v_{j+1}$, the vector in column $c_{2j+1}$ is a vector of the form $$v_{j+1}=\alpha_1^{j+1}e_1+\alpha_2^{j+1}\mathbf{j}(e_1)+\dots +\alpha_{k_{j+1}-1}^{j+1}e_{k_{j+1}/2}+\mathbf{j}(e_{k_{j+1}/2})$$
and $$\mathbf{j}(v_{j+1})=(\bar{\alpha}_2^{j+1})e_1+(\bar{\alpha}_1^{j+1})\mathbf{j}(e_1)+\dots +e_{k_{j+1}/2}+ (\bar{\alpha}_{k_{j+1}-1}^{j+1})\mathbf{j}(e_{(k_{j+1})/2}).$$ 
\noindent
If $l_{j+1}<k_{j+1}$, independent of the parity of $k_{j+1}$, the same kind of elementary transformations as in the corresponding proof for the $\SLR$ case are performed. in column $c_{2(j+1)}$ zeros are created in all rows starting with $k_{j+1}$ and going down to $l_{j+1}+1$ by subtracting suitable multiplies of $c_{2(j+1)}$ from the columns in the matrix having a $1$ in this rows and putting the result in $c_{2(j+1)}$. Next, a $1$ is created in row $l_{2(j+1)}$ by normalisation. Again this is the only spot in row $l_{2(j+1)}$ with value $1$, because the column which had a $1$ in this spot was removed from the matrix. By subtracting suitable multiplies of $ c_{2(j+1)}$ from columns having a $1$ in spots greater than $l_{(j+1)}$, we create zeros in row $l_{j+1}$ at the right of the $1$ in column $c_{2(i+1)}$. We thus obtain points in the maximal rank matrix of the Schubert variety indexed by $k_1l_1k_2l_2\dots k_jl_jk_{j+1}l_{j+1}\dots k_ml_m\dots l_{j+2}$.   \noindent
\bigskip
\newline
At step $j=m$ we obtain points in the maximal rank matrix of the Schubert variety indexed by $k_1l_1k_2l_2\dots k_ml_m$.
\noindent
\bigskip 
\newline
Conversely, if the spacing condition is not satisfied, let $i$ be the smallest such that $k_i<l_i$ and $k_i$ is not odd with $l_i=k_i+1$. Then look at the matrix $\mathbf{j}(V_i)\oplus V_{2m-i}$. Then use the same reasoning as above to create a $1$ in the spot in the matrix corresponding to row $l_j$ and column $2j$ for all $j<i$ using our chose of $i$. Now there is a $1$ in each row in the matrix except in row $l_i$. Column $c_{2i}$ and $c_{2i-1}$ both have a $1$ in position $k_i$ or $k_i+1$, depending on the parity of $k_i$, and zeros bellow. Since $l_i>k_i$ this implies that for each $j<i$ there exists a column in the matrix that has $1$ in row $j$. Using these $1$'s we begin subtracting $c_{2i}$ from suitable multiplies of each such column, starting with $c_{2i-1}$ for $k_i$ even and the column corresponding to $k_i+1$ for $k_i$ odd, then going to the one that has 1 in the spot $k_i-1$, then to the one the has $1$ in the spot $k_i-2$ and so on. At each step the result is inserted in $c_{2i}$. This creates zeros in all the column $c_{2i}$ and proves that the matrix does not have maximal rank. 
\end{proof}
\noindent
Next we impose the additional condition on the Schubert varieties which intersect the base cycle $C_0$ that there dimension is equal to the codimension of the base cycle.
\begin{lemma}
If $w=k_1\dots k_ml_*l_m\dots l_1$ satisfies the spacing condition and $|w|=m^2-m$, then $k_1$ is odd and $l_1=k_1+1$.
\end{lemma}
\begin{proof}
If $k_1$ is odd and $l_1<k_1$, construct a permutation $\tilde{w}$ out of $w$ as follows. The number $k_1+1$ sits in the permutation $w$ on one of the positions from $2$ to $n-1$. If $k_1+1$ sits on a position greater than $m$, perform the transposition that exchanges $l_1$ and $k_1+1$. In this case $|\tilde{w}|<|w|$, because of the following two reasons. Firstly, the number $l_1$ does not need to cross over the number $k_1+1$, as in $w$. Secondly, because $l_1<k_1+1$, the number $k_1+1$ needs to cross over at most the number of numbers that $l_1$ was crossing over, to be brought in its initial position in $w$. Moreover, $\tilde{w}$ satisfies the spacing condition which is contrary to the fact that no Schubert variety of dimension less than the codimension of the base cycle intersects the base cycle. 
\noindent
\bigskip
\newline
If $k_1$ is odd and the number $k_1+1$ sits on a position $p$ smaller than $m$, then do the transposition that exchanges $k_1$ and $k_1+1$. If $k_1$ is even and $l_1=k_1-1$, then do the transposition that exchanges $k_1$ and $l_1$. If $l_1$ is smaller than $k_1-1$, then do the transposition that exchanges $k_1$ and $k_1-1$.
\noindent
\bigskip
\newline
By the same reason used in the first case, in all the other cases $|\tilde{w}|<|w|$ and $\tilde{w}$ satisfies the spacing condition. This implies a contradiction.
\end{proof}

\begin{figure}
\centerline {\includegraphics[scale=0.7]{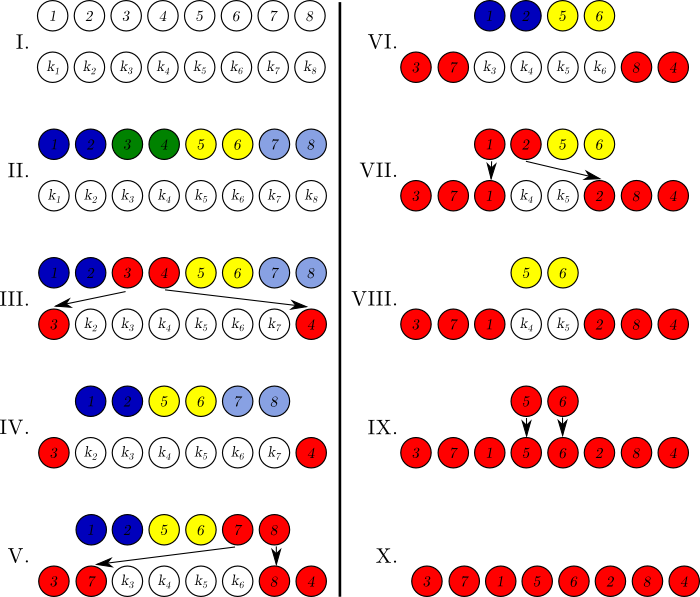}}
\caption{Strictly pairing condition example.}
\label{figure3}
\end{figure}

\begin{thm}
A Schubert variety $S_w$ belongs to $\mathcal{S}_{C_0}$ if and only if $w$ satisfies the strictly pairing condition. In particular, in this case $w$ satisfies the spacing condition and $|w|=m^2-m$.
\end{thm}
\begin{proof}
As in the $\SLR$ case the theorem is proved by induction on dimension. From the previous lemma $k_1$ must be odd and $l_1$ must equal $k_1+1$.
Remove from the set $\{1,2,\dots,n\}$ the elements $k_1$ and $k_1+1$ to form a set $\Sigma$ with $n-2$ elements. Define a bijective map $\Phi:\Sigma\rightarrow \{1,2,\dots,n-2\}$ by $\Phi(x)=x$ if $x < k_1$ and $\Phi(x)=x-2$ if $x>k_1+1$. By induction all possible $\tilde{w}\in \Sigma_{n-2}$ that satisfy the strictly pairing condition and parametrize a Schubert variety in $\mathcal{S}_{\tilde{C}_0}$ are obtained.
\noindent
\bigskip
\newline
Now return to the original situation by defining for each $\tilde{w}$ a corresponding $w \in \Sigma_n$ with $w(1)=k_1$, $w(i+1)=\Phi^{-1}(\tilde{w}(i)), \,\forall 1\le i \le n-2$, $w(n)=k_1+1$. 
\noindent
\bigskip
\newline
In order to compute the dimension of $w$ from the dimension of $\tilde{w}$, let $\tilde{p}_j$ be the distances for the elements of $\tilde{w}$. First consider those elements $\epsilon$ of the full sequence $w$ which are smaller than $k_1$. In order to move them to their appropriate position one needs the number of steps $\tilde{p}_\epsilon$ to do the same for their associated point in $\tilde{w}$, plus $1$ for having to pass $k_1$. Thus in order to compute the length of $w$ from the length of $\tilde{w}$, $k_1-1$ is first added to the former. Having done the above $k_1+1$ is moved to its place, directly at the right of $k_1$. This requires crossing $2m-k_1-1$ numbers. So together $2m$ was added to $(m-1)^2-(m-1)$, to obtain a total of $m^2+m$. Thus if $w$ satisfies the strictly pairing condition it has the correct minimal length to parametrize a Schubert variety in $\mathcal{S}_{C_0}$. 
\end{proof}
%For ease of calculations, it is convenient to work here with the canonical matrix form of an Iwasawa-Schubert cell in the standard basis, as defined in section $3.2$.
\begin{table}
\caption{The elements $w\in \Sigma_8$ which satisfy the strictly pairing condition}
\vspace*{0.5cm}
\begin{minipage}{0.5\textwidth}
		\begin{center}
		\begin{tabular}{cccccccc}
			\hline 1 & 3 & 5 & 7 & 8 & 6 & 4 & 2 \\ 
			\hline 1 & 3 & 7 & 5 & 6 & 8 & 4 & 2 \\ 
			\hline 1 & 5 & 3 & 7 & 8 & 4 & 6 & 2 \\ 
			\hline 1 & 5 & 7 & 3 & 4 & 8 & 6 & 2 \\ 
			\hline 1 & 7 & 3 & 5 & 6 & 4 & 8 & 2 \\ 
			\hline 1 & 7 & 5 & 3 & 4 & 6 & 8 & 2 \\ 
			\hline 
		\end{tabular} 
		\end{center}
		\begin{center}
			\textbf{I.}
		\end{center}
	\end{minipage}
	\begin{minipage}{0.5\textwidth}
	\begin{center}
			\begin{tabular}{cccccccc}
			\hline 3 & 1 & 5 & 7 & 8 & 6 & 2 & 4 \\ 
			\hline 3 & 1 & 7 & 5 & 6 & 8 & 2 & 4 \\ 
			\hline 3 & 5 & 1 & 7 & 8 & 2 & 6 & 4 \\ 
			\hline 3 & 5 & 7 & 1 & 2 & 8 & 6 & 4 \\ 
			\hline 3 & 7 & 1 & 5 & 6 & 2 & 8 & 4 \\ 
			\hline 3 & 7 & 5 & 1 & 2 & 6 & 8 & 4 \\ 
			\hline 
		\end{tabular}
	\end{center}
		\begin{center}
			\textbf{II.}
		\end{center} 
	\end{minipage}
	
	\vspace*{1cm}
	
	\begin{minipage}{0.5\textwidth}
		\begin{center}
			\begin{tabular}{cccccccc}
			\hline 5 & 1 & 3 & 7 & 8 & 4 & 2 & 6 \\ 
			\hline 5 & 1 & 7 & 3 & 4 & 8 & 2 & 6 \\ 
			\hline 5 & 3 & 1 & 7 & 8 & 2 & 4 & 6 \\ 
			\hline 5 & 3 & 7 & 1 & 2 & 8 & 4 & 6 \\ 
			\hline 5 & 7 & 1 & 3 & 4 & 2 & 8 & 6 \\ 
			\hline 5 & 7 & 3 & 1 & 2 & 4 & 8 & 6 \\ 
			\hline 
		\end{tabular} 
		\end{center}
		\begin{center}
			\textbf{III.}
		\end{center}
	\end{minipage}
	\begin{minipage}{0.5\textwidth}
		\begin{center}
			\begin{tabular}{cccccccc}
			\hline 7 & 1 & 3 & 5 & 6 & 4 & 2 & 8 \\ 
			\hline 7 & 1 & 5 & 3 & 4 & 6 & 2 & 8 \\ 
			\hline 7 & 3 & 1 & 5 & 6 & 2 & 4 & 8 \\ 
			\hline 7 & 3 & 5 & 1 & 2 & 6 & 4 & 8 \\ 
			\hline 7 & 5 & 1 & 3 & 4 & 2 & 6 & 8 \\ 
			\hline 7 & 5 & 3 & 1 & 2 & 4 & 6 & 8 \\ 
			\hline 
		\end{tabular} 
		\end{center}
		\begin{center}
			\textbf{IV.}
		\end{center}
	\end{minipage}
\label{table1}	
\end{table}
\noindent
Figure \ref{figure3} is an example of how the algorithm for building an element $w\in \Sigma_8$ which satisfies the strictly pairing condition works. Table \ref{table1} gives all the Weyl group elements in $\Sigma_8$ which satisfy the strictly pairing condition and thus all the Iwaswa-Schubert varieties in this case.
\noindent
\bigskip
\newline
The next corollary is an immediate consequence of the previous theorem.
\begin{cor}
There exists a bijection in between $\Sigma_m$ and $\mathcal{S}_{C_0}$ that sends each element $s=s_1\dots s_m\in \Sigma_m$ to an element $w=k_1\dots k_ml_m\dots l_1\in \Sigma_{2m}$ that parametrizes a Schubert variety in $\mathcal{S}_{C_0}$. The element $w$ is obtained from $s$ by $k_i=2s_i-1$ for all $1\le i \le m$ and $l_i=k_i+1$ for all $1\le i \le m$. Conversely, each $w=k_1\dots k_ml_m\dots l_1$ that parametrizes an element of $\mathcal{S}_{C_0}$ gives the unique element of $\Sigma_m$ defined by $(k_i+1)/2$, for all $1\le i \le m$. Thus, the cardinality of $\mathcal{S}_{C_0}$ is $m!.$
\end{cor}

\noindent
The next theorem gives a geometric description in terms of flags of the intersection point. 
\begin{thm}
A Schubert variety $S_{w}$ in $\mathcal{S}_{C_0}$ intersects the base cycle $C_0$ in precisely one point. If $s=s_1\dots s_m\in \Sigma_m$ corresponds to $w$, then the intersection point is the flag associated to the basis $$(e_{s_1},e_{s_2},\dots e_{s_m},\mathbf{j}(e_{s_m}),\dots, \mathbf{j}(e_{s_1}))$$ of $\mathbb{C}^n$. 
\end{thm}
\begin{proof}
For $w=k_1\dots k_m l_m\dots l_1$, let $s=s_1\dots s_m$ the associated element in $\Sigma_m$. Conside $V:=(0\subset <v_1>\subset \dots \subset <v_1,\dots, v_{n-1}>\subset <v_1,\dots,v_{n-1},v_n>)$ an arbitrary point in $\mathcal{O}_w$, where $$v_i=\alpha_1^{i}e_1+\dots + \alpha_{s_i-1}^{i}e_{s_i-1}-\beta_1^{i}e_{m+1}-\dots -\beta_{s_i-1}^{i}e_{m+s_i-1}+e_{s_i}, \,\forall\,1\le i\le m.$$  
\noindent
It is convenient to look at this point as a matrix $[v_1,\dots, v_n]$ with each of its columns having a leading coefficient $1$ or $-1$. 
Observe that multiplying $v_i$ with the matrix $J$ from the right gives $$v_jJ=\beta_1^{i}e_{1}+\dots +\beta_{s_i-1}^{i}e_{s_i-1}+\alpha_1^{i}e_{m+1}+\dots + \alpha_{s_i-1}^{i}e_{m+s_i-1}+e_{m+s_i},$$
and $\omega(v_i,v_k)$ is given by the standard inner product of the vectors $v_iJ$ and $v_k$. 
\noindent
The $\omega-$isotropic condition on flags means that the matrix $[v_1\dots v_{n-i}]$ has the column vector $v_iJ$ perpendicular to itself and all the other vectors in the matrix, for each $1\le i \le m$. Observe that the vectors that are excluded to obtain this matrices are only vectors with leading coefficient "$-1$". Start with the initial matrix and for the first step $i=1$ disregard the last column of the initial matrix, for the second step $i=2$ disregard the last and pre last column and so on until step $i=m$ is reached. Looking at this process closely and imposing the $\omega$-isotropic conditions gives the explicit description of the intersection point. 
\noindent
\bigskip
\newline
For the step $j=1$ disregard from the matrix the vector $v_n$ which contains a "$-1$" on row $m+s_1$. If $s_1=1$, then there is nothing to prove. Otherwise, disregarding $v_n$ creates a matrix that has among its columns all vectors that contain a leading "$1$" or a leading "$-1$" respectively on the spot $p$, where $1\le p \le s_1-1$ and $m+1\le p\le m+s_1-1$, respectively. Denote such column vectors with $f_p$ and observe that if $m+1\le p \le m+s_1-1$, then the first $m$ rows of $f_p$ are all zero. Using the relations $$v_1J\cdot f_1=0,\dots , v_1J\cdot f_{s_i-1}=0, v_1J\cdot f_{m+1}=0, \dots, v_1J\cdot f_{m+s_i-1}=0$$ and computing step by step it follows that $\beta_i^1=0$, and then $\alpha_i^1=0$ for all $1\le i \le s_i-1$.
\noindent
\bigskip
\newline
The condition $v_1J\cdot v_p=0, \,\forall \, 2\le i \le n-1$ is equivalent with $\beta^p_{s_1}=0, \,\forall \, 2\le p \le n-1$, because $v_1J$ is now a vector consisting of a $1$ on position $s_1$ and zeros elsewhere. This shows in particular that $v_{n-1}=j(e_{s_2}).$
\noindent
\bigskip
\newline
For $j=2$, $v_n$ and $v_{n-1}$ are removed from the initial matrix. This creates a matrix that has among its columns all vectors that contain a leading $1$ or a leading $-1$ respectively on spots $p$, with $1\le p \le s_2-1$ and $m\le p \le m+s_2-1$, respectively, with the possible exception of $p=s_1+m$, when $s_1<s_2$. This exception is however harmless, because $\alpha^2_{s_1}=0$ by definition, since it sits at the left of the leading $1$ on column $v_1$. 
\noindent
\bigskip
\newline
Following the same algorithm as before, zeros are created step by step starting with  $v_2J\cdot f_1=0$ and going further to $v_2J\cdot f_p=0$, where $f_p$ is a column vector where a $1$ sits on row $p$ for all $1\le i \le s_2-1$. This gives us $\beta_i^2=0,\,\forall\,1\le i\le s_2-1$. To obtain $\alpha_i^2=0,\,\forall\,1\le i \le s_2-1$ the equations $v_2J\cdot f_p=0, \,\forall \,m+1\le p\le m+s_2-1$ are used, with a possible jump at $p=m+s_1$, where $\alpha_{s_1}^2=0$ by default. The condition $v_2J\cdot v_p=0, \,\forall \, 3\le i \le n-1$ is equivalent to $\beta^p_{s_2}=0, \,\forall \, 3\le p \le n-2$, because $v_2J$ is now a vector consisting of a $1$ on position $s_2$ and zeros elsewhere.  
\noindent
\bigskip
\newline
Assume that we have shown that $v_{i}=e_{s_i}\,\forall\,1\le i \le j-1<m$ and $$\beta_{s_i}^p=0, \,\forall\,1\le i \le j-1, \, \forall i+1\le p \le n-i.$$ 
\noindent
To go to step $j$ use the fact that $l_j=k_j+1$. In this case the columns $v_p$ for $n-j+1\le p \le n$ are removed from the initial matrix. This creates a matrix that has among its columns all vectors that have a leading $1$ on spots $p$ for all $1\le p \le s_j-1$ and all vectors that have a leading $-1$ on spots $p$ for all $m+1\le p \le m+s_j-1$, with the possible exception of the spots $m+s_i$ for all $1\le i \le {j-1}$. As before these exceptions are harmless because $\alpha_{s_i}^j=0$ for all $1\le i \le {j-1}$. Zeros are created step by step starting with $v_jJ\cdot f_1$ and going up to $v_jJ\cdot f_{s_j-1}$ to obtain $\beta_i^j=0, \, \forall \, 1\le i \le s_j-1$. To obtain $\alpha^j_{i}=0$, for all $1\le i \le s_j-1$, the equations $v_jJ\cdot f_p=0$, for all $m+1\le p \le s_j-1$, are used together with the possible exceptions mentioned above. 
\end{proof}
\noindent
It thus follows that the homology class $[C_0]$ of the base cycle inside the homology ring of $Z$ is given in terms of the Schubert classes of elements in $\mathcal{S}_{C_0}$ by: $$[C_0]=\sum_{S\in\mathcal{S}_{C_0}}\,[S].$$

\section{Results in the measurable case}
As in the case of an open $SL(n,\R)$-orbit, the open $SL(m,\H)$-orbit $D$ in $Z=G/P$ is \textit{measurable} if $P$ is defined by a symmetric dimension sequence as follows:
\begin{itemize}
\item{$d=(d_1,\dots, d_s,d_s,\dots, d_1)$ or $e=(d_1.\dots, d_s,e',d_s \dots d_1 )$, for $n=2m$, }
\item{$e=(d_1.\dots, d_s,e',d_s \dots d_1)$, for $n=2m+1$.}
\end{itemize}  
\noindent
Recall the notation
$$\mathcal{S}_{C_0}:=\{S_w \text{ Schubert variety }: dim S_w+dim C_0=dimZ \text{ and } S_w\cap C_0 \ne \emptyset \}.$$
The main idea in the measurable case is to lift Schubert varieties $S_w\in \mathcal{S}_{C_0}$ to minimal dimensional Schubert varieties $S_{\hat{w}}$ in $\hat{Z}:=G/B$ that intersect the open orbit $\hat{D}$ and consequently the base cycle $\hat{C}_0$. As in the $SL(n,\R)$ case the Weyl group element $w$ that parametrizes a Schubert variety $S_w$ is divided into \textbf{symmetric block pairs} $(B_i,\tilde{B_j})$ corresponding to the symmetric dimension sequence defining the parabolic $P$.

\begin{defi}
A permutation $w$ is said to satisfy the \textbf{generalized spacing condition}, if for each symmetric block pair $(B_j,\tilde{B}_j)$ the elements of $\tilde{B}_j$ can be arranged in such a way that if the elements of $B_j$ are denoted by $k_1^j\dots k_{d_j}^j$ and the rearranged elements of $\tilde{B}_j$ by $l_{1}^j\dots l_{d_j}^j$, then either $l_i^j<k_i^j$ or $k_i^j$ is odd and $l_i^j=k_i^j+1$, for all $1\le i \le d_j$. 
\end{defi}  
\noindent
Observe that in the case of the symmetric dimension sequence $e$ we can always rearrange the elements in the single block $B_{e'}$ so that they satisfy the spacing condition inside the block.
\begin{defi}
A permutation $w$ is said to satisfy the \textbf{generalized strictly pairing condition} if for each symmetric block pair $(B_j,\tilde{B}_j)$, if we denote the elements of $B_j$ by $k_1^j\dots k_{d_j}^j$ and the elements of $\tilde{B}_j$ by $l_1^j\dots l_{d_j}^j$, read from left to write, then $k_i^j$ is odd, $l_i^j$ is even and $l_i^j=k_i^j+1,\,\forall \, 1\le i \le d_j$.
\end{defi}
\noindent
If $w$ satisfies the generalized spacing condition, $\tilde{w}$ denotes the permutation obtain from $w$ by replacing each block $\tilde{B}_j=l_1^j\dots l_{d_j}^j$ with a choice of rearrangement of its elements $\tilde{l}_1^j\dots \tilde{l}_{d_j}^j$ required so that $w$ satisfies the generalized spacing condition. Further inside each rearranged block $\tilde{B}_j$ in $\tilde{w}$, $\tilde{l}_1^j\dots \tilde{l}_{d_j}^j$  is rewritten as $\tilde{l}_{d_j}^j \tilde{l}_{d_j-1}^j\dots \tilde{l}_1^j.$ In the case of $B_{e'}$ being part of the representation of $w$ the following rearrangement is chosen : if $e'$ is even then $B_{e'}$ is rearranged as $l'_{e'/2+1}$ $l'_{e'/2+2}\dots l'_{e'}l'_1\dots l'_{e'/2-1}l'_{e'/2}$ and if $e'$ is odd then $B_{e'}$ is rearranged as $l'_{(e'+1)/2+1}$ $l'_{(e'+1)/2+2}\dots l'_{e'}l'_1\dots l'_{(e'+1)/2-1}l'_{(e'+1)/2}$. Observe that now $\tilde{w}$ is a permutation that satisfies the spacing condition for the $G/B$ case and it is of course also just another representative of the coset that parametrizes $S_w$.
\begin{prop}
A Schubert variety $S_w$ parametrized by the permutation $$w=B_1B_2\dots B_sB_{e}\tilde{B}_s\dots \tilde{B}_1$$ has nonempty intersection with the base cycle $C_0$ if and only if $w$ satisfies the generalized spacing condition, i.e. if and only if there exist a lift of $S_w$ to a Schubert variety $S_{\hat{w}}$ that intersects the base cycle in $G/B$. 
\end{prop}
\begin{proof}
If $w$ satisfies the generalized spacing condition, then by the above observation one can find another representative of the parametrization coset of $S_w$, $\hat{w}$ that satisfies the spacing condition and thus a Schubert variety $S_{\hat{w}}$ that intersects $\hat{C}_0$. Because the projection map $\pi:G/B\rightarrow G/P$ is equivariant it follows that $\pi(S_{\hat{w}})=S_w$ intersects $C_0$.
\noindent
\bigskip
\newline
Conversely, suppose that $S_w\cap C_0 \ne \emptyset$. Then for every point $p \in S_w\cap C_0 $ there exist $\hat{p}\in S_{\hat{w}} \cap \hat{C}_0$ with $\pi (\hat{p})=p$ and $\pi(S_{\hat{w}})=S_w$, for some Schubert variety in $G/B$ indexed by $\hat{w}$. It follows that $\hat{w}$ satisfies the spacing condition and $w$ is obtained from $\hat{w}$ by dividing $\hat{w}$ into blocks $B_1\dots B_sB_{e'}\tilde{B}_s\dots \tilde{B}_1$ and arranging the elements in each such block in increasing order. As a consequence this just shows that $w$ satisfies the generalized spacing condition.  
\end{proof}
\noindent
If $w$ satisfies the generalized strictly pairing condition, one can construct the permutation $\hat{w}$ by rearranging the elements inside each block $\tilde{B}_j$ for all $1\le i \le s$, such that $l_1^j\dots l_{d_j}^j$ becomes $l_{d_j}^j\dots l_1^j$. In the case of $B_{e'}$ being part of the representation the following rearrangement is chosen: $k_1'\dots k_{e'/2-1}'k_{e'/2}'\dots k'_{e'}$ becomes $k_1'k_3'\dots k_{e'-1}'k_{e'}'\dots k_4'k_2'.$ Note that $\hat{w}$ is a permutation that satisfies the strictly pairing condition for the $G/B$ case and it is of course just another representative of the parametrization coset of $S_w$. In parallel to the $SL(n,\R)$ case such a choice of rearrangement for $w$ is called a \textbf{canonical rearrangement}. Denote by $\hat{\mathcal{S}}_{\hat{C}_0}$ the subset of Schubert varieties in $\mathcal{S}_{\hat{C}_0}$ parametrized by such a $\hat{w}$.
\noindent
\bigskip
\newline
Note that if $S_w\in \mathcal{S}_{C_0}$ lifts to $S_{\hat{w}}$ such that $\hat{w}$ satisfies the strictly pairing condition, then $\dim S_{\hat{w}}-\dim S_w=(\dim \hat{Z}-\dim \hat{C}_0)-(\dim Z-\dim C_0)=(\dim \hat{Z}-\dim Z)-(\dim \hat{C_0}-\dim C_0).$ Since $\pi$ is a $G_0$ and $K_0$ equivariant map, if $F$ denotes the fiber of $\pi$ over a base point $z_0\in C_0$, then the fiber of $\pi|_{\hat{C}_0}:\hat{C}_0\rightarrow C_0$ over $z_0$ is just $F\cap \hat{C}_0$ and $\dim S_{\hat{w}}-\dim S_w$ must equal $\dim F - \dim (F\cap \hat{C}_0)$. 
\noindent
\bigskip
\newline
Similar to the $SL(n,\R)$ case, in the case of $Z=G/B$ the $\omega$-isotropic condition on flags is equivalent to $V_i\subset V_i^{\perp}$ for all $1\le i\le m$, $V_m=V_m^\perp$ and the flags $V_{n-i}$ are determined by $V_{n-i}=V_{i}^{\perp},$ for all $1\le i \le m$. Thus in the case of the dimension sequence $d=(d_1,\dots, d_s,d_s,\dots,d_1)$, $\dim F- \dim (F\cap C_0)$ is equal to $2\sum_{i=1}^sd_i(d_i-1)/2-\sum_{i=1}^sd_i(d_i-1)/2$ which is equal to $\sum_{i=1}^sd_i(d_i-1)/2$. In the case when the dimension sequence is given by $e=(d_1,\dots, d_s, e', d_s,\dots, d_1 )$ and the base point $z_0$ contains a middle flag of length $e$, it remains to add to the above number the difference in between the dimension of the total fiber over this flag and the dimension of the $\omega$-isotropic flags in this fiber. This is just a special case of the $G/B$ case for a full flag of length $e'$ and the number is $e'(e'-1)/2-(e'/2)^2$ which equals to $(e'/2)^2-(e'/2)$.
\noindent
\bigskip
\newline
By construction, if $w$ satisfies the generalized strictly pairing condition, then $|w|=|\hat{w}|-\sum_{i=1}^sd_i(d_i-1)/2$ in the case of the symmetric dimension sequence $d=(d_1,\dots , d_s,d_s,\dots, d_1)$ and $|w|=|\hat{w}|-\sum_{i=1}^sd_i(d_i-1)/2-(e'/2)^2+(e'/2)$ in the case of the dimension sequence $e=(d_1,\dots, d_s, e',d_s,\dots, d_1)$.
\noindent
\bigskip
\newline
With the above preparations we are now in a position to prove the main results on lifting to $\hat{Z}=G/B$ in the measurable case.
\begin{thm}
A permutation $w=B_1B_2\dots B_sB_{e'}\tilde{B}_s\dots \tilde{B}_2\tilde{B}_1$ satisfies the generalized strictly pairing condition if and only if $w$ parametrizes an Iwasawa-Schubert variety $S_w\in \mathcal{S}_{C_0}$ and the lifting map $f:\mathcal{S}_{C_0}\rightarrow \mathcal{S}_{\hat{C}_0}$ defined by $S_w\mapsto S_{\hat{w}}$, with $\hat{w}$ the canonical rearrangement of $w$, is injective.
\end{thm}
\begin{proof}
If $w$ satisfies the generalized strictly pairing condition, then by the above observation the canonical rearrangement $\hat{w}$ of $w$ is just another representative of the parametrization coset of $S_w$. Moreover, $\hat{w}$ satisfies the spacing condition and thus parametrizes a Schubert variety $S_{\hat{w}}$ that intersects $\hat{C}_0$. Because the projection map $\pi$ is equivariant it follows that $\pi(S_{\tilde{w}})=S_w$ intersects $C_0$. From the remarks before the statement of the theorem it follows that if $w$ satisfies the generalized strictly pairing condition, then the difference in dimensions in between $S_{\hat{w}}$ and $S_w$ is such that the codimension of $S_w$ is the same as the dimension of the base cycle.
\noindent
\bigskip
\newline
Conversely, suppose that $S_w$ has a lift to $S_{\hat{w}}\in \mathcal{S}_{\hat{C}_0}$. It follows that $\hat{w}$ satisfies the strictly paring condition and $w$ is obtained from $\hat{w}$ by dividing $\hat{w}$ into blocks $B_1\dots B_sB_{e'}\tilde{B}_s\dots \tilde{B}_1$ and arranging the elements in each such block in increasing order. And this just shows that $w$ satisfies the generalized strictly pairing condition.

\end{proof}

\begin{cor}
There exist a bijection in between the subset $A$ of $\Sigma_m$ defined by elements that have inside each block of length $d_i$ and $e'/2$ respectively  consecutive increasing numbers, for all $1\le i \le s$  and $S_{C_0}$ that sends each element $s=s_1\dots s_m\in \Sigma_m$ to an element $w=k_1\dots k_ml_m\dots l_1\in \Sigma_{2m}$ that parametrizes a Schubert variety in $S_{C_0}$. The element $w$ is obtained from $s$ by $k_i=2s_i-1$ for all $1\le i \le m$ and $l_i=k_i+1$ for all $1\le i \le m$. Conversely, each $w=k_1\dots k_ml_m\dots l_1$ that parametrizes an element of $S_{C_0}$ gives the unique element of $A$ defined by $(k_i+1)/2$, for all $1\le i \le m$. 
\end{cor}

\begin{thm}
A Schubert variety $S_{w}$ in $S_{C_0}$ intersects the base cycle $C_0$ in precisely one point. If $s=s_1\dots s_m$ is the corresponding element to $w$ in $\Sigma_m$ then the intersection point is the partial flag associated to the basis $$(e_{s_1},e_{s_2},\dots e_{s_m},j(e_{s_m}),\dots, j_(e_{s_1}))$$ of $\mathbb{C}^n$. 
\end{thm}
\noindent
It thus follows that the homology class $[C_0]$ of the base cycle inside the homology ring of $Z$ is given in terms of the Schubert classes of elements in $\mathcal{S}_{C_0}$ by: $$[C_0]=\sum_{S\in\mathcal{S}_{C_0}}\,[S].$$

\section{Results in the non-measurable case}
The case of $Z=G/P$ and $D\subset Z$ a non-measurable open orbit is completely analogous with the $\SLR$-case and the same result holds. Recall that associated to the flag domain $D$ there exists its measurable model $\hat{D}$ in $\hat{Z}=G/\hat{P}$ together with the projection map $\pi:\hat{Z}\rightarrow Z$. If $P^{-}=L\rtimes U^{-}$ denotes the opposite parabolic subgroup to P, namely the block lower triangular matrix with blocks of size $f_1,\dots f_u$, then $\hat{P}= P\cap \tau(P^{-})$. Furthermore, recall the following two results: 
\noindent
\newline
$\bullet$
The restriction map $$\pi|_{\hat{C}_0}:\hat{C}_0\rightarrow C_0$$
is biholomorphic. In particular, if $q$ and $\hat{q}$ denote the respective codimensions of the cycles, it follows that $\hat{q}=dim(\hat{F})+q$
\noindent
\bigskip
\newline
$\bullet$
The map $\Phi: S_{C_0}\rightarrow \pi^{-1}(S_{C_0})\subset S_{\hat{C}_0}$ is bijective.
\noindent
\bigskip
\newline
For convenience, the argument in the $SL(n,\R)$ case is reproduced here.
If $d=(d_1,\dots,d_s,d_s,\dots,d_1)$ or $e=(d_1,\dots,d_s,e',d_s,\dots,d_1)$ is a symmetric dimension sequence, then one can construct another dimension sequence out of it, not necessarily symmetric, by the following method. Consider an arbitrary sequence $t=(t_1,\dots,t_p)$ such that each $t_i\ge 1$ for all $1\le i\le p$, at least one $t_i$ is strictly bigger than $1$ and $t_1+\dots +t_p=2s$ or $2s+1$ depending on wether one considers $d$ or $e$, respectively. Associated to $t$ the sequence $\delta=(\delta_1,\dots , \delta_p)$ is defined by $\delta_{j}:=\sum_{i=1}^{j}\,t_i$. With the use of $\delta$ the new dimension sequence $f_{\delta}=(f_{\delta_{1}},\dots,f_{\delta_{p}})$ is defined by $f_{\delta_{1}}:=\sum_{i=1}^{\delta_{1}}d_i$,   $$f_{\delta_{j}}:=\sum_{i=\delta_{j-1}+1}^{\delta_{j}}\,d_i, \text{ for all } 2\le j \le p.$$ 
\noindent
\bigskip
\newline
Because $\hat{P}$ is obtained as the intersection of two parabolic subgroups $P$ and $\sigma(P^{-})$, it follows that the dimension sequence $f$ of $P$ is obtained as above, from the dimension sequence of $\hat{P}$, as $f_{\delta}$ for a unique choice of $t$. For ease of computation we do not break up anymore the dimension sequence of $\hat{P}$ into its symmetric parts and we simply write it as $d=(d_1,\dots,d_s)$, where $s$ can be both even or odd. 
Using the usual method of computing the dimension of $Z$ it then follows that $$\dim Z=\dim {\hat{Z}}-\sum_{t_j>1}\sum_{\delta_{j-1}+1\le h<g\le \delta_j}\,d_hd_g.$$ For example if $P$ corresponds to the dimension sequence $(2,4,3)$, then an easy computation with matrices shows that $\hat{P}$ corresponds to the dimension sequence $(2,1,3,1,2)$, $t=(1,2,2)$ and $\delta=(1,3,5)$. Moreover, $\dim Z=\dim{\hat{Z}}-1\cdot3-1\cdot2$.
\noindent
\bigskip
\newline
Given the sequence $f=f_{\delta}$, we are now interested in describing the set $\mathcal{S}_{C_0}$. Let $S_{\hat{w}}$ in  $\mathcal{S}_{\hat{C}_0}$ be the unique Schubert variety associated to a given $S_w \in \mathcal{S}_{C_0}$ such that  $\pi (S_{\hat{w}})=S_{w}$.  If $\hat{w}$ is given in block form by $B_1\dots B_s$, where here again the notation used does not take into consideration the symmetric structure of $\hat{w}$, then $w$ is given in block form by $C_1\dots C_p$ corresponding to the dimension sequence $f_\delta$. The blocks $C_j$ are given by $C_1=\bigcup_{i=1}^{\delta_{1}}\, B_{d_i}$ and $$C_j=\bigcup_{i=\delta_{j-1}+1}^{\delta_j}\, B_{d_i}, \text{ for all } 2\le j\le p,$$ arranged in increasing order. Moreover,
\begin{align*}
\dim S_w&=\dim Z-\dim C_0=\dim Z-\dim{\hat{C}_0} \\
&=\dim{\hat{Z}}-\sum_{t_j>1}\sum_{\delta_{j-1}+1\le h<g\le \delta_j}\,d_hd_g -\dim{\hat{C}_0} \\
&=\dim{S_{\hat{w}}}-\sum_{t_j>1}\sum_{\delta_{j-1}+1\le h<g\le \delta_j}\,d_hd_g.
\end{align*}
Finally, understanding what conditions $\hat{w}$ satisfies in order for the above equality to hold amounts to understanding the difference in length that the permutation $\hat{w}$ looses when it is transformed into $w$. If $C_{j}$ contains only one $B$-block from $\hat{w}$, i.e. $t_j=1$, then this is already ordered in increasing order and it does not contribute to the decrease in dimension. If $C_{j}$ contains more $B$-blocks, say $$C_j=\bigcup_{i=\delta_{j-1}+1}^{\delta_j}\, B_{d_i},$$ we start with the first block $B_{\delta_{j-1}+1}$ which is already in increasing order and bring the elements from $B_{\delta_{j-1}+2}$ to their correct spots inside the first block. As usual, to each number we can associate a distance, i.e. the number of elements it needs to cross  in order to be brought to the correct spot, and we denote by $\alpha_{\delta_{j-1}+2}$ the sum of these distances. The maximum value that $\alpha_{\delta_{j-1}+2}$ can attain is that when all the elements in the second block are smaller than each element in the first block, i.e. the last element in the second block is smaller than the first element in the first block. In this case $\alpha_{\delta_{j-1}+2}=d_{\delta_{j-1}+1}d_{\delta_{j-1}+2}$, the product of the number of elements in the first block with the number of elements in the second block. Next we bring the elements in the $3^{rd}$ block among the already ordered elements from the first and second block. Observe that the maximal value that $\alpha_{\delta_{j-1}+3}$ can attain is $d_{\delta_{j-1}+1}d_{\delta_{j-1}+3}+d_{\delta_{j-1}+2}d_{\delta_{j-1}+3}$ when the last element in the $3^{rd}$ block is smaller than all elements in the first two blocks. In general we say that the group of blocks used to form  $C_j$ is in \textbf{strictly decreasing order} if it satisfies the following: the last element of block $B_{i+1}$ is smaller than the first element of block $B_i$ for all $\delta_{j-1}+1\le i \le \delta_j-1$. Consequently, if one wants to order this sequence of blocks into increasing order one needs to cross over $$\sum_{\delta_{j-1}+1\le h<g\le \delta_j}\,d_hd_g$$ elements. Thus if all the blocks $C_j$ with $t_j>1$ among $w$ come from groups of blocks arranged in strictly decreasing order in $\hat{w}$, then $$|w|=|\hat{w}|-\sum_{t_j>1}\sum_{\delta_{j-1}+1\le h<g\le \delta_j}\,d_hd_g.$$
As a consequence we have the following observation:

\begin{prop}
A Schubert variety $S_{\hat{w}}\in \mathcal{S}_{\hat{C}_0}$ is mapped under the projection map to a Schubert variety $S_{w}\in \mathcal{S}_{C_0}$ if and only if all the blocks $C_j$ with $t_j>1$ among $w$ come from groups of blocks arranged in strictly decreasing order in $\hat{w}$.
\end{prop}

\chapter{The case of the real form $\protect \SUP$}
\section{Preliminaries} 
Here we are concerned with the action of the real form $G_0:=SU(p,q)$, for all $p,q$ such that $p+q=n$, on an arbitrary flag manifold $Z=G/P$. Without loss of generality we assume throughout this chapter that $p\ge q$. One way of viewing $G_0$ is as the real form associated to the real structure $\tau:\SLC\rightarrow \SLC$, defined by $$\tau(A)=I_{p,q}\overline{A}^{-t}I_{p,q},  \text{ where } I_{p,q}=\begin{pmatrix} -I_q & 0 \\ 0& I_p
\end{pmatrix}. 
$$  
\noindent
If $h:\C^n\times\C^n\rightarrow \C$ is the standard Hermitian form of signature $(p,q)$ defined by $$h(v,w)=-\sum_{i=1}^qv_i\bar{w}_i+\sum_{i=1}^pv_{q+i}\bar{w}_{q+1}, \, \forall v,w\in \C^n,$$ then $G_0$ is also the group of isometries associated to $h$. The notions of orthogonality and isotropy of vectors are defined with respect to $h$.
A Cartan involution $\theta:G_0\rightarrow G_0$ is given by $\theta(A)=I_{p,q}AI_{p,q}$ and the maximal compact subgroup associated to it is $K_0:=S(U(p)\times U(q))$.
\noindent
\bigskip
\newline
In contrast to the case of the real forms $\SLR$ and $SL(m,\H)$, where the number of open orbits is either one or two, here the number of open orbits is huge. This fact has important consequences for the combinatorics associated to this situation and much more complicated considerations arise.  
\noindent
\bigskip
\newline
The open $G_0$-orbits in $Z=Z_\delta=G/P$, with $\delta=(\delta_1,\dots, \delta_s)$ and $\delta_i:=\dim V_i$, are parametrized by the sequences $a: 0\le a_1\le \dots \le a_s\le q$ and $b: 0\le b_1\le \dots \le b_s\le p$, with $a_i+b_i=\delta_i$, for all $1\le i\le s$. For a fixed pair $(a,b)$ the open orbit $D_{a,b}$ is described by the set of all flags $(0\subset V_1\subset \dots \subset V_s=\C^n)$ such that $V_i$ has signature $(a_i,b_i)$ with respect to $h$ for all $1\le i\le s$. In particular, no flags in an open orbit contain degenerate subspaces. For more on the orbit structure of $G_0$ see \cite{Wolf1972}.
\noindent
\bigskip
\newline
An equivalent parametrization of an open orbit $D_{a,b}$, in a way which is better suited for combinatorial considerations, is given by a binary sequence $\alpha$ constructed from $a$ and $b$ as follows. For $Z=G/B$ and $\alpha=\alpha_1\dots\alpha_n$ set $\alpha_1:=0$ if $a_1=1$ or otherwise $\alpha_1=1$. Construct the sequence inductively by setting $\alpha_i:=0$ if $a_i>a_{i-1}$ or otherwise $\alpha_i:=1$ if $b_i>b_{i-1}$, for all $2\le i\le n$. For a more geometrical meaning the set of values of $\alpha_i$ can be replaced by $\{-,+\}$. Namely, $0$ corresponds to $-$ and thus to negative signature and $1$ corresponds to $+$ and thus to positive signature. For example, the pair $a: 0\le 1\le 1\le 2\le 2$ and $b: 1\le 1\le 2\le 2\le 3$ which parametrizes an open orbit of $SU(3,2)$ is equivalent to $\alpha=10101=+-+-+$.
\noindent
\bigskip
\newline
For ease of dimension computations, the parabolic $P$ will be parametrized by a dimension sequence $d=(d_1,\dots, d_s)$, where $d_i=\dim(V_i/V_{i-1})$, for all $1 \le i \le  s$ with $\dim V_0=0$. Equivalently, the open orbits are parametrized by the sequences $a: 0\le a_1\le \dots \le a_s\le q$ and $b: 0\le b_1\le \dots \le b_s\le p$, with $a_i+b_i=d_i$, for all $1\le i\le s$. For a fixed pair $(a,b)$ the open orbit $D_{a,b}$ is described by the set of all flags $(0\subset V_1\subset \dots \subset V_s=\C^n)$ such that $V_i$ has signature $(\sum_{j=1}^ia_j,\sum_{j=1}^ib_j)$ with respect to $h$, for all $1\le i\le s$. From the parametrization by binary sequences it immediately follows that the number of open orbits in $Z=G/B$ is $n \choose p$, i.e., the number of binary sequences of length $n$ with $q$ $0$'s and $p$ $1$'s is the number of ways to choose $p$ slots from $n$ slots.
\noindent
\bigskip
\newline
A suitable parametrization of an open orbit in $G/P$, for reason that will become clearer later, is described as follows. Given an open orbit $D_{a,b}$, let $f_i:=min(a_i,b_i)$ with $a_i+b_i=d_i$ for all $1\le i \le s$. The sequence $\alpha$ will be defined by prescribing $s$ adjacent blocks. The block $i$ is defined for each $1\le i \le s$ by letting the first $f_i$ members be equal to $-$. If $a_i>b_i$, then the next $a_i-b_i$ members are set to $-$ and if $b_i>a_i$ then the next $b_i-a_i$ members are set to $+$. Finally, the last $f_i$ members are set to $+$. Call this parametrization the \textbf{canonical parametrization} associated to the open orbit $D_{a,b}$.
For example, if $G_0=SU(6,3)$, $d=(2,4,3)$, $a:1\le 1\le 1$, and $b:1\le 3\le 2$, then $\alpha=(-+)(-+++)(-++)$.

\section{Dimension-related computations}
Denote $E^{-}:=<e_1,\dots, e_q>$ and $E^{+}:=<e_{q+1},\dots , e_n>$, where $(e_1,\dots, e_n)$ is the standard basis in $\C^n$, and consider the decomposition of $\C^n$ as $\C^n=E^{-}\oplus E^{+}$. For a fixed open orbit $D_{a,b}$, the base cycle $C_0$ is the set $$C_0=\{\mathcal{V}\in Z: \,\dim V_i\cap E^{-}=\sum_{j=1}^ia_j\text{ and } \dim V_i\cap E^{+}=\sum_{j=1}^ib_j,  1\le i\le s\}.$$ 
If $\mathcal{V}^-$ and $\mathcal{V}^+$ are complete flags in $E^-$ and $E^+$ respectively, then one can put them together in the obvious way to obtain complete flags $\mathcal{V}$ that belong to the open orbit $D_{a,b}$ in $G/B$. By definition, the group $K_0$ stabilises the set of flags that are constructed in this way. Furthermore, since $K_0=K_0^-\times K_0^+$ and each factor of $K_0$ acts transitively on the corresponding factor $\mathcal{V}^-$ and $\mathcal{V}^+$, it follows that $K_0$ acts transitively on the set of complete flags obtained in this way. This gives $\mathcal{V}$ the structure of a complex manifold. Because $K_0$ has a unique orbit in $D_{a,b}$ which is a complex manifold it follows that $\mathcal{V}$ is the base cycle $C_0$. 
\noindent
\bigskip
\newline
Note that the above construction defines a bijective holomorphic mapping $\mathcal{V}^-\times \mathcal{V}^+\rightarrow C_0$ and hence $C_0$ is the product of this flag manifolds. In order to construct a complete flag in $C_0$, one needs both a complete flag in $E^-$ and a complete flag in $E^+$. It thus follows that $\mathcal{V}^-$ and $\mathcal{V}^+$ are the full flag manifolds in $E^-$ and $E^+$, respectively. As a consequence if $Z=G/B$, $$\dim C_0=p(p-1)/2+q(q-1)/2,$$
and the Iwasawa-Schubert varieties must be of dimension $pq$.
\noindent
\bigskip
\newline
More generally, in the case of $Z=G/P$, for a fixed open orbit $D_{a,b}$, let $d_i=a_i+b_i$. Furthermore, let $d^1$ be the sequence obtained from considering the sequence $a$ with the $0$'s removed, let $d^2$ the sequence obtained from considering the sequence $b$ with the $0$'s removed and let $s_1$ and $s_2$ the number of elements of $d^1$ and $d^2$ respectively. Then $d^1$ defines a partial flag variety in $E^-$, $d^2$ defines a partial flag variety in $E^+$ and the base cycle in $D_{a,b}$ is obtained by putting together these flags. As a consequence, if $Z=G/P$, then $$\dim C_0=\sum_{1\le i<j \le s_1} d^1_id^1_j+\sum_{1\le i <j\le s_2}d^2_id^2_j.$$
\noindent
Of course, the zero elements in both the sequence $a$ and the sequence $b$ do not change the result of computing the dimension of the partial flag manifolds in both $E^-$ and $E^+$. Thus, one can rewrite $$\dim C_0=\sum_{1\le i<j \le s} a_ia_j+\sum_{1\le i <j\le s}b_ib_j,$$ and the Iwasawa-Schubert varieties must be of dimension $$\sum_{1\le i <j \le s}(a_ib_j+b_ia_j).$$
\noindent
\textbf{Remark.} Certain base points in $C_0$ are particularly easy to describe, namely the coordinate flags associated to the standard basis in $\C^n$. Recall that the coordinate flags are also the fixed points of the torus $T$ of diagonal matrices in $SL(n,\C)$ and denote by $\fix(T)$ the full set of fixed points of $T$ in Z. For future comparison it is important to note that the cardinality of $\fix(T)\cap{C_0}$ is $q!\cdot p!$ and this number is independent of the open orbit $D_{a,b}$. For example if $D_\alpha$, $\alpha=+-++-$, is an $SU(3,2)$-open orbit, then the $T$-fixed points inside $C_0$ are the flags associated to $<e_3,e_1,e_4,e_5,e_2>$, $<e_3,e_1,e_5,e_4,e_2>$,
$<e_4,e_1,e_3,e_5,e_2>$, $<e_4,e_1,e_5,e_3,e_2>$, $<e_5,e_1,e_3,e_4,e_2>$, $<e_5,e_1,e_4,e_3,e_2>$, $<e_3,e_2,e_4,e_5,e_1>$, $\dots$, $<e_5,e_2,e_4,e_3,e_1>$.

\section{Introduction to the combinatorics}
%In the case of $SU(p,q)$ with $p\ge q$ the fixed points of $T_{I}=AT$ are $e_i\pm e_{2q-i+1}$ for $1\le i \le q$ together with $e_{2q+i}$ for $1\le i\le p-q$, for the case when $p>q$. 
We choose $B_I$ to be the stabiliser of the complete flag obtained from the ordered basis $$(e_1+e_{2q},\dots ,e_q+e_{q+1},e_{2q+1},\dots,e_n,e_q-e_{q+1},\dots, e_1-e_{2q}).$$ 
\noindent
To see that this is an Iwasawa-Borel subgroup recall that $G_0$ has a unique closed orbit in $Z=G/B$ and this closed orbit parametrises the set of Iwasawa-Borel subgroups of $G$. Namely, each Iwasawa-Borel subgroup can be obtained as the subgroup of $G$ which fixes a given point inside the closed orbit. This result can be found in \cite{Fels2006}. 
%\begin{defi}
%A flag $\mathcal{V}$ in $Z=G/B$ is called \textbf{maximally isotropic} if any orthonormal basis of $V_i$ contains $i$ isotropic vectors, for each $1\le i\le q$ and any orthonormal basis of $V_{q+i}$ contains $q$ isotropic vectors and $i$ positive vectors, for each $1\le i \le q-i$.  
%\end{defi}
\begin{defi}
A flag $\mathcal{V}$ in $Z=G/B$ is called \textbf{maximally isotropic} if $V_i\subset V_i^{\perp}$, for all $1\le i \le q$, $V_{p+i}=V_{q-i}^{\perp}$, for all $1\le i \le q$.
\end{defi}
\begin{prop}
The closed $G_0$-orbit in $Z=G/B$ is the set of all maximally isotropic flags.
\end{prop}
\begin{proof}
Let $(a_i,b_i,c_i)$ be the triple that denotes the signature of $V_i$ with respect to $h$, where $c_i$ denotes the dimension of the degeneracy of the restricted form. Then the maximally isotropic flags are precisely the flags of signature $(0,0,i)$, for all $V_i$, $1\le i \le q$, $(j,0,q)$, for all $V_{q+j}$, $1\le j\le p-q$ and $(p-q+j,j,q-j)$, for all $V_{p+j}$, $1\le j\le q$. By Witt's Theorem $G_0$ acts transitively on the set of such flags. 
\end{proof}
\noindent
Let $\mathfrak{g}_0:=\mathfrak{su}(p,q)$ be the Lie algebra of $SU(p,q)$. The next paragraphs contain a sketch of the basic ideas of Cartan subalgebras and restricted roots of a real semisimple Lie algebra, in the context of $\mathfrak{su}(p,q)$. A more detailed description together with the general results on the Iwasawa decomposition of a real semisimple Lie algebra can be found in \cite{Knapp2002}.
\noindent
\bigskip
\newline
The Lie algebra $g_0$ is the fixed point set of the real structure $\tau:\mathfrak{sl}(n,\C)\rightarrow \mathfrak{sl}(n,\C)$, $\tau(X)=-I_{p,q}\overline{X}^{t}I_{p,q}$. If $q=0$ and $p=n$, then $g_0$ is denoted by $\mathfrak{su}(n)$. Then, 
$$\mathfrak{g}_0=\{ \begin{pmatrix} A & B \\ \overline{B}^t& D
\end{pmatrix}: A\in \mathfrak{u}(q), D\in\mathfrak{u}(p), \tr A+\tr D=0.
   \}$$
\noindent
The Cartan involution $\theta$ at the Lie algebra level is given by the map $X\mapsto -\overline{X}^t$. This defines the Cartan decomposition $\mathfrak{g}_0=\mathfrak{k}\oplus\mathfrak{p}$ of $\mathfrak{g}_0$ as a direct sum of the $+1$ and $-1$ eigenspaces of $\theta$. In terms of matrices, $$\mathfrak{k}=\{ \begin{pmatrix} A & 0 \\ 0& D
\end{pmatrix}: A\in \mathfrak{u}(q), D\in\mathfrak{u}(p), \tr A+\tr D=0,
   \}$$
$$\mathfrak{p}=\{ \begin{pmatrix} 0 & B \\ \overline{B}^t& 0
\end{pmatrix}: B \text{ is a } p\times q \text{ complex matrix}.
   \}$$
For fixed $Y\in \mathfrak{g}_0$, let $\ad Y:\mathfrak{g}_0\rightarrow \mathfrak{g}_0$ be the adjoint map defined by $X\mapsto XY-YX$. If $\mathfrak{a}$ is the maximal abelian subspace in $\mathfrak{p}$ given by
$$B=\{ \begin{pmatrix} 0 & 0 &\dots &a_q &0&\dots&0\\ 0&0& \dots & 0 &0&\dots&0 \\ \dots&\dots&\dots&\dots&\dots&\dots&\dots \\ 0&a_2&\dots&0&0&\dots&0 \\a_1&0&\dots&0&0&\dots&0
\end{pmatrix}: a_i\in \R, \, \forall 1\le i \le q
   \},$$
then $\{\ad H:H\in \mathfrak{a}\}$ is a commuting family of selfadjoint transformations of $\mathfrak{g}_0$. This gives a decomposition of $\mathfrak{g}_0$ as an orthogonal direct sum of simultaneous eigenspaces. The corresponding eigenvalues, which can be seen as members $\lambda\in \mathfrak{a}^*$ of the dual space of $\mathfrak{a}$, are called restricted roots if they are nonzero and their corresponding eigenspace $\mathfrak{g}_\lambda$ is nonzero. If $\Psi$ denotes the set of restricted roots one can choose a set of positive roots in $\Psi$, $\Psi^+$. If $\mathfrak{m}:=Z_{\mathfrak{k}}(\mathfrak{a})$ denotes the centraliser of $\mathfrak{a}$ in $\mathfrak{k}$, $\mathfrak{n}^+:=\sum_{\lambda\in \Psi^+}\mathfrak{g}_\lambda$, and $\mathfrak{n}^-:=\sum_{\lambda\in \Psi^+}\mathfrak{g}_{-\lambda}$, then the Iwasawa-decomposition of $\mathfrak{g}_0$ is $\mathfrak{n}^-\oplus \mathfrak{m}\oplus\mathfrak{a}\oplus\mathfrak{n}^+.$
\noindent
\bigskip
\newline
By definition, a Cartan subalgebra of $\mathfrak{g}_0$ is a subalgebra $\mathfrak{h}_0$ of $\mathfrak{g}_0$ such that the complexification $\mathfrak{h}$ of $\mathfrak{h}_0$ is a Cartan subalgebra of $\mathfrak{g}$.In general, if $\mathfrak{t}$ is a maximal abelian subspace in $m$, then $\mathfrak{t}\oplus a$ is a Cartan subalgebra of $\mathfrak{g}_0$. In the example of $\mathfrak{su}(p,q)$ one can choose $\mathfrak{h}_0$ to be

$$\mathfrak{h}_0=\{\begin{pmatrix} \alpha_q&0 & \dots &0&a_q &0&\dots&0\\ 0&\alpha_{q-1}& \dots & a_{q-1}&0&0&\dots&0 \\ \dots&\dots&\dots&\dots&\dots&\dots&\dots&\dots \\ 0&a_{q-1}&\dots&\alpha_{q-1}&0&0&\dots&0 \\a_q&0&\dots&0&\alpha_q&0&\dots&0\\0&0&\dots &0&0&\alpha_{q+1}&\dots&0\\\dots&\dots&\dots&\dots&\dots&\dots&\dots&\dots \\0&0&\dots&0&0&0&\dots&\alpha_p
\end{pmatrix}: a_i\in \R, \,  \alpha_i\in i \R
   \}.$$
If $T_I:=Z_{G}(\mathfrak{h})$ is the maximal torus associated to $\mathfrak{h}$, then it is immediate that the fixed points of $T_I$ in $Z=G/B$ are the complete flags associated to the basis $$(e_1+e_{2q},\dots ,e_q+e_{q+1},e_q-e_{q+1},\dots, e_1-e_{2q}, e_{2q+1},\dots,e_n).$$
It is also immediate that the fixed point associated to the ordered basis  $$(e_1+e_{2q},\dots ,e_q+e_{q+1},e_{2q+1},\dots,e_n,e_q-e_{q+1},\dots, e_1-e_{2q}),$$ 
is a maximally isotropic flag. We choose the full basis above as a basis of $\C^n$ and define the Iwasawa-Borel subgroup $B_I$ to be the isotropy subgroup of $G$ at this point.
\noindent
\bigskip
\newline
It is convenient to differentiate in terms of notation between the Weyl group of $G$ with respect to $T$, given by the isomorphism associated to the standard basis in $\C^n$, denoted by
 $\Sigma_n^T$ and the Weyl group of $G$ with respect to $T_I$, given by the isomorphism associated to the basis that defines the Iwasawa-Borel subgroup, denoted by $\Sigma_n^I$. Of course, this two groups are isomorphic and a particularly useful isomorphism in this situation is $\gamma_I:\Sigma_n^I\rightarrow \Sigma_n^T$, $\gamma(i)=i-p+q$, if $i>p$, $\gamma(i)=i+q$, if $q< i \le p$ and $\gamma(i)=i$, if $i\le q$, for $p>q$. The isomorphism $\gamma_I$ is defined to be the \textbf{canonical isomorphism} in between the two Weyl groups. If $p=q$ this is just the identity map.

\begin{defi}
A permutation $w\in \Sigma_n^I$ is said to satisfy the \textbf{pairing condition} if the number $n-i+1$ stays at the left of the number $i$ for all $1\le i \le q$ in the one line notation of the permutation. (When $p>q$, the order of the numbers $q+i$ is arbitrary for all $1\le i \le p-q$.)  
\end{defi}
\noindent
Geometrically this condition says that the first subspace in a flag of $S_w$ that contains a point of the form $\beta_1(e_1+e_{2q})+\dots +\beta_i({e_i}+e_{n-i+1})+\dots +(e_{i}-e_{n-i+1})$ sits to the left of the first subspace of the flag which contains a point of the form $\beta_1(e_1+e_{2q})+\dots +({e_i}+e_{n-i+1})$ and this gives the possibility of splitting the degeneracy.

\begin{defi}
A permutation $w\in \Sigma_n^I$ is said to satisfy the \textbf{strictly pairing condition} if it is constructed by the following algorithm: consider the pairs $({n-i+1},i)$ in order from $i=1$ until $i=q$ and place them step by step in a sequence of $n$ empty boxes, in such a way that the components of each pair sit as close as possible to each other. This means that once a pair is placed its position can be ignored so that the places at the immediate left and right of this pair become adjacent. After all the pairs are placed, in the case when $p>q$, the numbers $q+i$ for $1\le i \le p-q$ are placed in the remaining spots in strictly increasing order. 
\end{defi} 
\begin{figure}
\centerline{\includegraphics[scale=0.85]{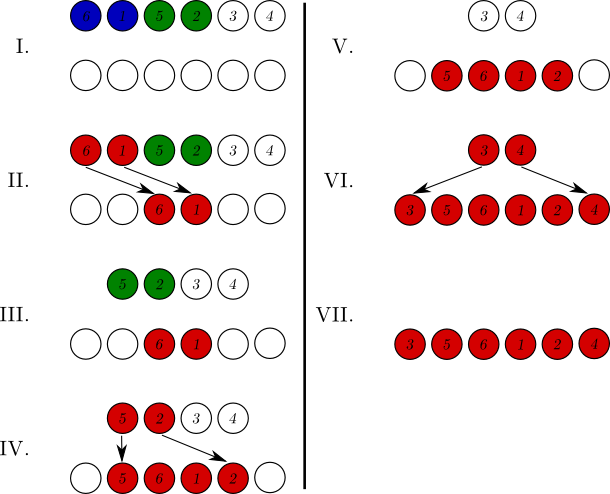}}
\caption{Strictly pairing condition example.}
\end{figure}
\noindent
\textbf{Example.} In $SU(4,2)$ we have two pairs $(6,1)$ and $(5,2)$ and $15$ permutations that satisfy the strictly pairing condition. These are $(61)(52)34$, $(61)3(52)4$, $3(61)(52)$ $4$, $(61)34(52)$, $3(61)4(52)$, $34(61)(52)$, $(52)(61)34$, $(52)3(61)4$, $3(52)(61)4$, $(52)3$ $4(61)$, $3(52)4(61)$, $34(52)(61)$, $(5(61)2)34$, $3(5(61)2)4$, $34(5(61)2)$. The permutation "4(5(61)2)3" does not satisfy the strictly pairing condition because $3$ and $4$ are not in increasing order in the permutation.

\section{Main results for $Z=G/B$}
Let $d:=p(p-1)/2+q(q-1)/2$ denote the dimension of any of the base cycles $C_0^\alpha$. Recall that a cycle in the Barlet space $C^{d}(Z)$ is a formal linear combination $C=n_1C_1+\dots + n_kC_k$ with positive integral coefficients where each $C_j$ is an irreducible $d-$dimensional compact subvariety of $Z$. Let $\mathcal{C}:=\sum C_0^\alpha$ be the cycle in $Z$ obtained by summing over all the base cycles $C_0^\alpha$ associated to each one of the $SU(p,q)$-open orbits $D_\alpha$. The first part of this section is concerned with the description of $\mathcal{C}$ in the homology ring of $Z$. That is we are interested in finding all Iwasawa-Schubert varieties which have non-empty intersection with at least one of the summands involved in the definition of $\mathcal{C}$ and which are of dimension $pq$. Furthermore, for each such Iwasawa-Schubert variety we describe the base cycles that it intersects their total number and their respective points of intersection. A unified formula for $[\mathcal{C}]$ can be found in \ref{eqcycle} below. Finally, the second part of this section provides a description of the homology class of each base cycle $C_0^\alpha$, as stated in \ref{eqbasecycle} below.
\begin{prop}
A Schubert variety $S_w$ intersects at least one open orbit if and only if $w$ satisfies the pairing condition.
\end{prop}

\begin{proof}
Assume first that $w$ satisfies the pairing condition and prove that $S_w$ intersects at least one open orbit. Let $\tilde{w}$ be the image of $w$ under the canonical isomorphism $\gamma_{I}$ defined above and let ${Q}_{\tilde{w}}$ be the coordinate flag associated to $\tilde{w}$:
 $$<e_{\tilde{w}(1)}>\subset<e_{\tilde{w}(1)},e_{\tilde{w}(2)}>\subset\dots \subset<e_{\tilde{w}(1)},\dots,e_{\tilde{w}(n)}>=\mathbb{C}^n.$$
This is clearly a flag that lies in an open orbit (and also in one base cycle). The claim is that this flag belongs to the orbit $B_I.Q_w$, where $Q_w$ is the coordinate flag associated to $w$. This is proved by induction. 
\noindent
\bigskip
\newline
At step one observe that by assumption $w(1)>q$. If $p>q$ and $w(1)\le p$, then the orbit $B_I.<e_{\tilde{w}(1)}>$ clearly contains the point $e_{\tilde{w}(1)}$. Otherwise, the orbit of interest is $B_I.<e_{n-w(1)+1}-e_{\tilde{w}(1)}>$. Because $B_I$ fixes $Q$ it follows that $<e_{n-w(1)+1}+e_{\tilde{w}(1)}>$ sits in $Q$ before $<e_{n-w(1)+1}-e_{\tilde{w}(1)}>$. This implies that the $B_I$-orbit contains the point $$<\beta(e_{n-w(1)+1}+e_{\tilde{w}(1)})+(e_{n-w(1)+1}-e_{\tilde{w}(1)})>,$$ and for $\beta:=-1$ this is $<e_{\tilde{w}(1)}>$. 
\noindent
\bigskip
\newline
By induction one constructs the $j^{th}$ dimensional subspace of the complete flag given by $$<e_{\tilde{w}(1)}>\subset \dots \subset <e_{\tilde{w}(1)},\dots,e_{\tilde{w}(j)}>.$$ 
\noindent
\bigskip
\newline
At the step $j+1$ two possibilities exist: either $w(j+1)>q$ or $w(j+1)\le q$. If $p>q$ and $q<w(j+1)\le p$, then the orbit $B_I.<e_{\tilde{w}(j+1)}>$ clearly contains the point $e_{\tilde{w}(j+1)}$.
The case when $w(j+1)>p$ is equivalent to the fact that $<e_{n-w(j+1)+1}+e_{\tilde{w}(j+1)}>$ sits in $Q_{\tilde{w}}$ before $<e_{n-w(j+1)+1}-e_{\tilde{w}(j+1)}>$. This in turn is equivalent to saying that $B_I.<e_{n-w(j+1)+1}-e_{\tilde{w}(j+1)}>$ contains the point $$<\beta(e_{n-w(j+1)+1}+e_{\tilde{w}(j+1)})+(e_{n-w(j+1)+1}-e_{\tilde{w}(j+1)})>,$$ which for $\beta=-1$ is $<e_{\tilde{w}(j+1)}>$. Therefore the flag $$<e_{w(1)}>\subset \dots \subset <e_{w(1)},\dots,e_{w(j)},e_{w(j+1)}>,$$ is constructed.
\noindent
\bigskip
\newline
Finally, if $w(j+1)\le q$, then, $w(j+1)=\tilde{w}(j+1)$ and because $w$ satisfies the pairing condition, the number $n-w(j+1)+1$ sits before the number $w(j+1)$ in $w$. By induction it follows that $e_{\tilde{w}(n-w(j+1)+1)}$ appears among the $e_i$'s included in the span of $<e_{\tilde{w}(1)},\dots,e_{\tilde{w}(j)}>$. Since the point $<e_{w(j+1)}+e_{\tilde{w}({n-w(j+1)+1})}>$ is obviously included in the $B_I$-orbit of itself at this stage one constructs the subspace $$<e_{w(1)},\dots, e_{\tilde{w}(n-w(j+1)+1)},\dots, e_{w(j)},e_{w(j+1)}+e_{\tilde{w}(n-w(j)+1)}>,$$ which equals $$<e_{\tilde{w}(1)},\dots,e_{\tilde{w}(j)},e_{\tilde{w}(j+1)}>.$$
\noindent
\bigskip
\newline
To prove the converse implication assume that $w$ does not satisfy the pairing condition. The claim is that $S_w$ does not intersect any open orbit. Assume that $w(j)$ is the first number smaller than $q$ that sits in the permutation $w$ before its corresponding pair $n-w(j)+1$. This means that until this point at least one flag $(0\subset V_1\subset\dots\subset V_{j-1})$ of non-degenerate signature can be constructed, namely $<e_{\tilde{w}(1)}>\subset \dots \subset <e_{\tilde{w}(1)},\dots,e_{\tilde{w}(j-1)}>$. But there can be others, of course. 
\noindent
\bigskip
\newline
Next one tries to add to any of these flags a point from the $B_I$- orbit of $<e_{\tilde{w}(n-w(j)+1)}+e_{w(j)}>$. Because this subspace sits in $Q$ before any subspace that contains the point $<e_{w(i)}-e_{\tilde{w}(n-w(j)+1)}>$, the $B_I$-orbit contains the following points: $$<\beta_1(e_1+e_{2q})+\beta_2(e_2+e_{2q-1})+\dots+(e_{w(j)}+e_{\tilde{w}(n-w(j)+1})>,$$
and this subspace is isotropic for all values of $\beta_i$. The only option to kill this degeneracy would be to use $<e_{w(j)}-e_{\tilde{w}(n-w(j)+1}>$ but by our assumption none of the $e_{w(j)}$ or $e_{\tilde{w}(n-w(j)+1)}$ sit in any of the flags $(0\subset V_1\subset\dots\subset V_{j-1})$. It follows that $S_w$ does not intersect any open orbit and this completes the proof.
\end{proof}
\noindent
If $w$ satisfies the strictly pairing condition, let $Perm_w:=\{\gamma_I(w'):\, w' \text{ is obtained}$ $\text{from } w \text{ by making none, some or all of the transpositions } (n-i+1,i), \text{ for } i\le q .\}$ Define by $T_w$ the subset of $\fix(T)$ that consists of all complete flags associated to elements of the Weyl group in the set $Perm_w \subset \Sigma_n^T$. It is immediate that the cardinality of both $Perm_w$ and $T_w$ is $2^q$. For example, consider $w=615234$, a permutation that satisfies the pairing condition for $p=4$ and $q=2$. Then $Perm_w=\{413256,143256,412356,142356\}$.
\noindent
\bigskip
\newline
Denote by $\mathcal{I}_{p,q}$ the set of all Iwasawa-Schubert varieties $S_w$, which intersect at least one open orbit and are parametrized by a Weyl group element of length $pq$. In general, for a set $C$ denote its power set by $\mathcal{P}(C)$.

\begin{lemma}
If $w\in \Sigma_n^I$ satisfies the pairing condition and $|w|=pq$, then the pair $(n,1)$ sits inside consecutive boxes in $w$. 
\end{lemma}

\begin{proof}
Assume that the pair $(n,1)$ does not sit inside consecutive boxes. Because $w$ satisfies the pairing condition it follows that $n$ sits at the left of $1$ in the one line notation of $w$ and there exists at least one other number between $n$ and $1$. Let $f$ be the minimum of the numbers that sit in between $n$ and $1$ and construct $w'$ out of $w$ by interchanging $f$ and $1$. Then $w'$ still satisfies the pairing condition, but $|w'|<|w|$. This is of course a contradiction to the fact that no Schubert variety of dimension less than the codimension of the base cycle has nonempty intersection with it.  
\end{proof}

\begin{thm}
A Schubert variety $S_w$ that intersects at least one open orbit is parametrized by an element $w\in \Sigma_n^I$ of length $pq$ if and only if $w$ satisfies the strictly pairing condition. Under this condition $S_w$ intersects exactly $2^q$ open orbits and it intersects the base cycles in each one of them in precisely one point. Moreover, the map $f:I_{p,q}\rightarrow \mathcal{P}({\fix(T)})$, defined by $S_w\mapsto T_w$, is injective and $f(S_w)$ consists of the points of intersection of $S_w$ with the base cycles inside the $2^q$ open orbits.
\label{theorem36}
\end{thm}

\begin{proof}
The first goal is to prove that $S_w$ intersects at least one open orbit with $w\in \Sigma_n^I$ of length $pq$ if and only if $w$ satisfies the strictly pairing condition. This is done by induction on the dimension of the flag manifold in the following way. Assume without loss of generality that $p>1$, and $q>1$. The case when either one of them is $1$ follows immediately from the above lemma.
\noindent
\bigskip
\newline
Remove the numbers $1$ and $n$ from the set $\{1,2,\dots . n\}$ to obtain a set $\Sigma$ with $n-2$ elements and define the bijective map $\Phi:\Sigma\rightarrow \{1,2,\dots,n-2\}$, by $\Phi(x)=x-1$. By induction one constructs all $w'\in \Sigma_{n-2}^I$ that satisfy the strictly pairing condition. In particular, any such $w'$ has length $(p-1)(q-1)$, satisfies the pairing condition and parametrizes an element of $\mathcal{I}_{p-1,q-1}$.  
 \noindent
 \bigskip
 \newline
Now return to the original situation by defining for each $w'$ a corresponding $w\in \Sigma_n^I$. This is done by first taking the preimage of $w'$ under $\Phi$ and then, using the above lemma, adjoining to this in two arbitrary consecutive places the pair $(n,1)$. It is clear that the newly formed permutation $w$ satisfies the strictly pairing condition. It remains to check that $w$ has minimal length, namely $pq$. For this, let $f$ denote the number of position at the left of $1$ in $w$. Then $n-f-1$ is the number of positions at the right of $1$ (and implicitly $n$ as well) in $w$. In order to compute the length of $w$ from the length of $w'$, one must add to $(p-1)(q-1)$, $f$ for the number of positions that $1$ needs to cross to be placed in the first box, together with $n-f-1$, for each element at the right of $n$. This amounts to $pq-p-q+1+f+n-f-1=pq.$
\noindent
\bigskip
\newline
The second goal is to show that if $w$ satisfies the strictly pairing condition, then the points of intersection are precisely the elements of $T_w$ and no other point in $S_w$ has suitable intersection with $E^+$ and $E^-$ (it does not respect the intersection dimension condition imposed to flags in $C_0$). The final goal is to show that all this points lie in different orbits. As usually, a main ingredient is that the intersection of $S_w$ with any open orbit is inside the Schubert cell which defines $S_w$. Without loss of generality assume $p>q$. The case when $p=q$ is analogous and corresponds to $\gamma_I$ being equal to the identity map.
\noindent
\bigskip
\newline
At the first step, because $w$ satisfies the strictly pairing condition, $w(1)>q$. 
Consider first the case when $w(1)>p$,  and look at the orbit of $<e_{(2q-\tilde{w}(1)+1)}-e_{\tilde{w}(1)}>$ under the $B_I$ action: $$\beta_1(e_1+e_{2q})+\dots+\beta_q(e_q+e_{q+1})+\dots+(e_{2q-\tilde{w}(1)+1}-e_{\tilde{w}(1)}).$$ The idea is to find all possible one-dimensional subspaces spanned only by vectors in $E^+$ or only by vectors in $E^-$. Two obvious ones are the spaces spanned by $<e_{\tilde{w}(1)}>$ and $<e_{2q-\tilde{w}(1)+1}>$. The claim is that these are the only possible ones that could complete to a full flag so that it has suitable intersection with $E^+$ and $E^-$. 
\noindent
\bigskip
\newline
Suppose that one can start with a flag of the type $<e_{\tilde{w}(1)}+\beta\,e_j>$, where $j$ is any number smaller than $\tilde{w}(1)$ but bigger than $q$ or with a flag of the type $<e_{2q-\tilde{w}(1)+1}+e_j>$, where $j$ is any number bigger than $2q-\tilde{w}(1)+1$ but smaller than $q$.  Now assume that it is possible to complete this to a $k-1$-dimensional subspace $F$ that has suitable intersection with $E^-$ and $E^+$, where $w(k)=n-w(1)+1$, i.e. the signature of $F$ is $(a,b)$ and $dimF\cap E^+=a$, $dimF\cap E^-=b$, $a+b=k-1$. 
\noindent
\bigskip
\newline
In order to obtain the next $k$-dimensional subspace one needs to add a point from the $B_I$-orbit of $<e_{2q-\tilde{w}(1)+1}+e_{\tilde{w}(1)}>$. Because none of the subspaces considered to construct $F$ contain neither of $e_{2q-\tilde{w}(1)+1}$ or $e_{\tilde{w}(1)}$ as free vectors, one needs to add their sum to obtain the $k$-dimensional subspace. 
This means that the new one-dimensional space just added will have intersection with both $E^+$, since $e_{\tilde{w}(1)}$ is in there, and, $E^-$ since $e_{2q-\tilde{w}(1)+1}$ is in there. This is of course a contradiction. 
\noindent
\bigskip
\newline
In the case when $q<w(1)\le p$, the only possible value for $w(1)$ is $q+1$ ( or equivalently, $\tilde{w}(1)=2q+1$). That is because $w$ satisfies the strictly pairing condition and thus the numbers $q< i \le p$ are arranged in increasing order inside $w$. Consequently, the $B_I$ orbit of $<e_{2q+1}>$ contains points of the form $<\beta_1(e_1+e_{2q})+\dots+\beta_q(e_q+e_{q+1})+e_{2q+1}>$. It is immediate that the only point inside $B_I$ which suitably intersects either only $E^-$ or only $E^+$ is $<e_{2q+1}>$.  
\noindent
\bigskip
\newline
By induction assume that $<e_{\tilde{w}(1)}>$ and $<e_{2q-\tilde{w}(1)+1}>$ can be completed to a $k-$dimensional flag by consistently choosing either $e_{\tilde{w}(j)}$ or $e_{2q-\tilde{w}(j)+1}$, if $w(j)>p$ or $w(j)<q$, and $e_{\tilde{w}(j)}$, if $q<w(j)\le p$, at each step $j\le k$, to complete the flag and this is the only possibility to do this. 
\noindent
\bigskip
\newline
First assume that $w(k+1)<q$. This means that $n-w(k+1)+1$ was encountered before in the permutation and either $e_{\tilde{w}(k+1)}$ or $e_{2q-\tilde{w}(k+1)+1}$ were considered when building the $k-$dimensional subspaces. The $B_I$ orbit of $<e_{\tilde{w}(k+1)}+e_{2q-\tilde{w}(k+1)+1)}>$ contains points of the form $<\beta_1(e_1+e_{2q})+\dots+(e_{\tilde{w}(k+1)} +e_{2q-\tilde{w}(k+1)+1})>$. It is immediate that the only way to obtain a $k+1$-dimensional subspace that has suitable intersection with $E^+$ and $E^-$ is to add $<e_{\tilde{w}(k+1)}+e_{2q-\tilde{w}(k+1)+1)}>$ to the $k-$dimensional subspace considered before which will be equivalent to only adding $e_{\tilde{w}(k+1)}$ or $e_{2q-\tilde{w}(k+1)+1)}$ depending on which one of them had previously appeared in the flag before. 
\noindent
\bigskip
\newline
If $w(k+1)>p$, then as in the case of $w(1)$, there exists the possibility of adding to the $k$-dimensional subspace built by induction either $e_{\tilde{w}(k+1)}$ or $e_{2q-\tilde{w}(k+1)+1}$ or a space of the form $e_{\tilde{w}(k+1)}+\beta\,e_j$, where $j$ is any number smaller than $\tilde{w}(k+1)$ but bigger than $q$, or $e_{2q-\tilde{w}(k+1)+1}+\beta\,e_j$, where $j$ is any number smaller than $q$ but bigger then $2q-\tilde{w}(k+1)+1$. The same reasoning as in the first step of the induction shows that none of the options except the first two are possible.
\noindent
\bigskip
\newline
Finally, assume that $q< w(k+1)\le p.$ By induction and the strictly pairing condition it follows that the $k$-dimensional subspace $F$ contains among its spanning vectors, $e_{q+i},$ for all $q<i<w(k+1).$ Thus the only relevant points in the $B_I$ orbit of $<e_{\tilde{w}(k+1)}>$ are the points of the form $<\beta_1(e_1+e_{2q})+\dots+\beta_q(e_q+e_{q+1})+e_{\tilde{w}(k+1)}>$. It is immediate that the only point inside $B_I$ which suitably intersects either only $E^-$ or only $E^+$ is $<e_{\tilde{w}(k+1)}>$.  
\noindent
\bigskip
\newline
This ends the proof of the fact that the set of points of intersection of $S_w$ with different base cycles is $T_w$. To see that all this points lie in different open orbits let $D_{\alpha(w)}$ be the open orbit that contains the point associated to $\tilde{w}$. 
If $\gamma_I(w') \in Perm_w$, let $\alpha(w')$ be the sequence of $+$'s and $-$'s obtained from $\alpha(w)$ by performing the same transpositions that are performed to obtain $w'$ form $w$ to $\alpha(w)$. Then from the definition of $T_w$ it follows that the point in $T_w$ associated with $\gamma_I(w')$ belongs to the open orbit $D_{\alpha(w')}$ and if $\gamma_I(w')\ne \gamma_I(w'')$ are two elements of $Perm_w$, then $\alpha(w')\ne\alpha(w'')$. 
\end{proof}
\noindent
Figure \ref{figure5} on the next page is a graphic representation of Theorem \ref{theorem36} where the Iwasawa-Schubert variety parametrized by $w=45910126783$ intersects $8$ of the open orbits of $SU(7,3)$. The points of intersection are depicted in red. At the right of each representation of an open orbit the parametrisation $\alpha$ together with the point of intersection are given explicitly. 
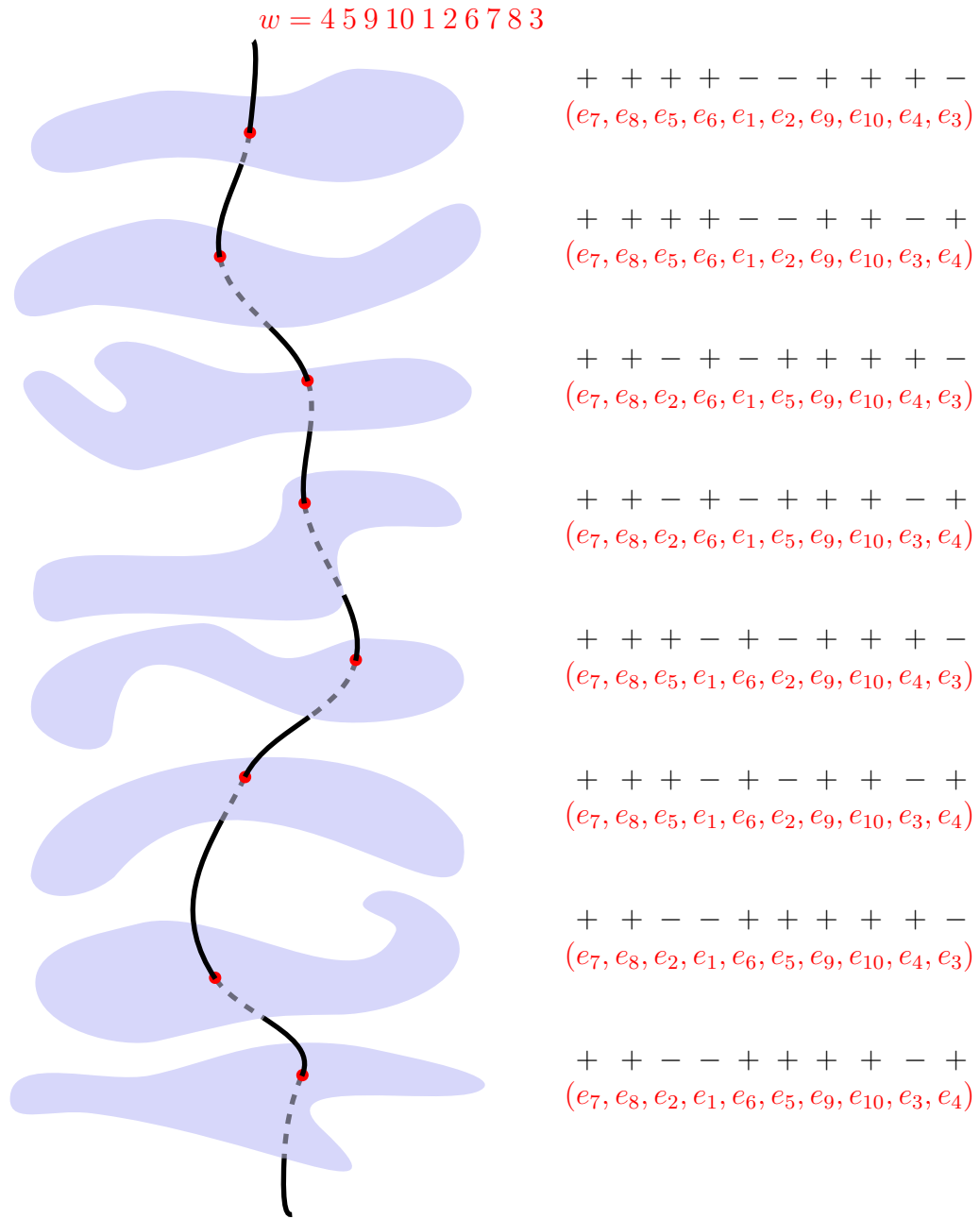
\begin{figure}
\setlength{\tabcolsep}{1pt}
\begin{tikzpicture}[y=0.60pt,x=0.60pt,yscale=-1, inner sep=0pt, outer sep=0pt]
\path[fill=c0000e1,fill opacity=0.157] (308.2060,154.5451) .. controls
  (275.9551,159.9556) and (241.7694,146.3887) .. (210.1014,154.5451) .. controls
  (184.6741,161.0943) and (136.8410,166.6935) .. (141.1308,192.5978) .. controls
  (144.0907,210.4713) and (168.7576,208.2264) .. (194.6425,202.1110) .. controls
  (269.7717,184.3615) and (296.3927,213.5094) .. (347.4479,212.8134) .. controls
  (382.4421,212.3362) and (432.4107,191.9600) .. (426.5261,165.2475) .. controls
  (421.2291,141.2020) and (382.7627,137.9864) .. (358.1502,137.3025) .. controls
  (340.5447,136.8134) and (325.5755,151.6312) .. (308.2060,154.5451) -- cycle;
\path[fill=c0000e1,fill opacity=0.157] (278.9950,339.7605) .. controls
  (247.1205,336.9883) and (201.2434,302.3551) .. (183.6098,329.0519) .. controls
  (176.5061,339.8066) and (210.6781,352.6105) .. (202.6588,362.7009) .. controls
  (188.4909,380.5283) and (155.6375,321.4822) .. (137.1454,343.3382) .. controls
  (125.4435,357.1689) and (189.5773,410.2589) .. (215.4623,404.1435) .. controls
  (290.5915,386.3940) and (279.3741,379.3852) .. (330.4293,378.6891) .. controls
  (365.4235,378.2120) and (437.6747,375.9141) .. (431.7902,349.2016) .. controls
  (406.4713,309.4951) and (329.8320,344.1820) .. (278.9950,339.7605) -- cycle;
\path[fill=c0000e1,fill opacity=0.157] (211.0550,239.0746) .. controls
  (179.2388,246.1053) and (120.9589,263.7020) .. (130.7329,294.7852) .. controls
  (136.1256,311.9348) and (166.6976,294.0263) .. (184.6650,294.6286) .. controls
  (235.6962,296.3393) and (288.3637,319.3184) .. (337.4704,305.3309) .. controls
  (376.2071,294.2972) and (453.7747,270.5182) .. (435.4679,234.6416) .. controls
  (421.7417,207.7420) and (379.5513,260.8114) .. (349.4340,263.0338) .. controls
  (302.7483,266.4788) and (256.7648,228.9736) .. (211.0550,239.0746) -- cycle;
\path[fill=c0000e1,fill opacity=0.157] (307.0199,422.1947) .. controls
  (308.6467,491.8899) and (167.5326,441.6419) .. (144.5694,472.4398) .. controls
  (140.7912,484.4792) and (138.5619,510.7714) .. (164.4468,504.6560) .. controls
  (239.5761,486.9065) and (362.4971,528.8729) .. (346.2617,480.4629) .. controls
  (324.1454,414.5176) and (431.2245,459.6095) .. (425.3400,432.8970) .. controls
  (420.0430,408.8515) and (381.5765,405.6359) .. (356.9640,404.9521) .. controls
  (339.3585,404.4629) and (306.6088,404.5872) .. (307.0199,422.1947) -- cycle;
\path[fill=c0000e1,fill opacity=0.157] (140.7030,674.8651) .. controls
  (143.6629,692.7386) and (175.3021,692.3809) .. (195.8963,675.5493) .. controls
  (286.8633,565.1667) and (441.4392,740.3170) .. (426.0983,647.5147) .. controls
  (373.5051,559.3327) and (145.9483,593.9033) .. (140.7030,674.8651) -- cycle;
\path[fill=c0000e1,fill opacity=0.157] (302.4399,786.9687) .. controls
  (274.2121,789.7060) and (247.6497,801.7596) .. (219.8910,807.5696) .. controls
  (189.1363,814.0067) and (127.0795,791.9188) .. (126.9562,823.3397) .. controls
  (126.8851,841.4564) and (154.0811,829.5104) .. (180.4678,832.8529) .. controls
  (287.4211,846.4008) and (399.4067,906.1018) .. (333.2732,843.5552) .. controls
  (293.7656,806.1903) and (411.1713,829.1398) .. (436.3158,818.2719) .. controls
  (458.9171,808.5034) and (391.9706,795.6907) .. (367.9399,790.3270) .. controls
  (346.6029,785.5645) and (324.1998,784.8586) .. (302.4399,786.9687) -- cycle;
\path[fill=c0000e1,fill opacity=0.157] (381.8698,711.8598) .. controls
  (360.0212,763.8722) and (267.3779,692.4514) .. (212.7130,706.3942) .. controls
  (186.8669,712.9865) and (142.4417,720.3278) .. (143.7424,746.9695) .. controls
  (145.1984,776.7911) and (199.1173,790.7666) .. (225.0023,784.6512) .. controls
  (300.1315,766.9017) and (293.5388,767.8811) .. (344.5939,767.1850) .. controls
  (379.5882,766.7079) and (429.5568,746.3316) .. (423.6722,719.6191) .. controls
  (418.3752,695.5736) and (373.5384,672.6144) .. (360.7618,691.6742) .. controls
  (355.3410,699.7608) and (385.6401,702.8841) .. (381.8698,711.8598) -- cycle;
\path[fill=c0000e1,fill opacity=0.157] (308.3318,529.7016) .. controls
  (288.4430,528.0695) and (273.0395,505.5417) .. (253.1107,506.5781) .. controls
  (210.6712,508.7853) and (134.3909,525.8157) .. (141.2566,567.7543) .. controls
  (144.1836,585.6332) and (192.0422,603.7250) .. (194.7683,577.2675) .. controls
  (203.9438,488.2166) and (276.0610,563.3340) .. (335.3813,571.5732) .. controls
  (370.0460,576.3879) and (432.5365,567.1165) .. (426.6520,540.4039) .. controls
  (421.3549,516.3585) and (385.8213,518.0809) .. (363.3211,516.6633) .. controls
  (344.5204,515.4787) and (327.1067,531.2423) .. (308.3318,529.7016) -- cycle;
\path[draw=black,dash pattern=on 4.00pt off 4.00pt,opacity=0.506,line
  join=miter,line cap=butt,miter limit=4.00,line width=2.002pt]
  (285.2463,180.1119) .. controls (284.3808,187.0925) and (282.4036,193.9895) ..
  (280.0025,200.8399);
\path[draw=black,dash pattern=on 4.00pt off 4.00pt,opacity=0.506,line
  join=miter,line cap=butt,miter limit=4.00,line width=2.002pt]
  (265.6254,262.7576) .. controls (269.2727,281.6714) and (284.6064,295.4174) ..
  (298.8924,309.7487);
\path[draw=black,dash pattern=on 4.00pt off 4.00pt,opacity=0.506,line
  join=miter,line cap=butt,miter limit=4.00,line width=2.002pt]
  (323.8936,345.4033) .. controls (327.1305,356.1540) and (326.6285,367.3395) ..
  (325.1969,378.6626);
\path[draw=black,dash pattern=on 4.00pt off 4.00pt,opacity=0.506,line
  join=miter,line cap=butt,miter limit=4.00,line width=2.002pt]
  (321.5153,426.8599) .. controls (324.5346,447.9787) and (338.2502,467.8412) ..
  (347.6006,487.8044);
\path[cm={{1.61064,0.0,0.0,1.61064,(2281.4315,-908.0222)}},fill=cff0000]
  (-1214.2034,828.6134)arc(-0.285:180.285:2.525)arc(-180.285:0.285:2.525) --
  cycle;
\path[draw=black,dash pattern=on 4.00pt off 4.00pt,opacity=0.506,line
  join=miter,line cap=butt,miter limit=4.00,line width=2.002pt]
  (355.4060,531.5048) .. controls (351.7476,547.6921) and (338.8884,558.6647) ..
  (324.6549,569.0283);
\path[draw=black,dash pattern=on 4.00pt off 4.00pt,opacity=0.506,line
  join=miter,line cap=butt,miter limit=4.00,line width=2.002pt]
  (282.2734,608.7994) .. controls (277.6240,618.3967) and (272.1485,628.0300) ..
  (266.9480,637.8327);
\path[draw=black,dash pattern=on 4.00pt off 4.00pt,opacity=0.506,line
  join=miter,line cap=butt,miter limit=4.00,line width=2.002pt]
  (262.7020,743.2750) .. controls (269.3151,753.2043) and (282.3536,761.3243) ..
  (294.4928,769.3652);
\path[draw=black,dash pattern=on 4.00pt off 4.00pt,opacity=0.506,line
  join=miter,line cap=butt,miter limit=4.00,line width=2.002pt]
  (319.8233,808.0289) .. controls (312.2908,824.5861) and (309.0254,844.0500) ..
  (308.1771,862.6841);
\path[cm={{1.61064,0.0,0.0,1.61064,(2245.2748,-1154.8128)}},fill=cff0000]
  (-1214.2034,828.6134)arc(-0.285:180.285:2.525)arc(-180.285:0.285:2.525) --
  cycle;
\path[draw=black,line join=miter,line cap=butt,miter limit=4.00,line
  width=2.002pt] (287.9223,118.9001) .. controls (287.9223,118.9001) and
  (292.4161,122.2810) .. (285.2463,180.1119);
\path[cm={{1.61064,0.0,0.0,1.61064,(2225.5148,-1072.304)}},fill=cff0000]
  (-1214.2034,828.6134)arc(-0.285:180.285:2.525)arc(-180.285:0.285:2.525) --
  cycle;
\path[draw=black,line join=miter,line cap=butt,miter limit=4.00,line
  width=2.002pt] (280.0025,200.8399) .. controls (272.7386,221.5645) and
  (261.5959,241.8623) .. (265.6254,262.7576);
\path[cm={{1.61064,0.0,0.0,1.61064,(2283.3234,-989.58495)}},fill=cff0000]
  (-1214.2034,828.6134)arc(-0.285:180.285:2.525)arc(-180.285:0.285:2.525) --
  cycle;
\path[draw=black,line join=miter,line cap=butt,miter limit=4.00,line
  width=2.002pt] (298.8924,309.7487) .. controls (309.6053,320.4957) and
  (319.7291,331.5719) .. (323.8936,345.4033);
\path[draw=black,line join=miter,line cap=butt,miter limit=4.00,line
  width=2.002pt] (325.1969,378.6626) .. controls (323.1649,394.7348) and
  (319.2598,411.0839) .. (321.5153,426.8599);
\path[cm={{1.61064,0.0,0.0,1.61064,(2315.2759,-803.6512)}},fill=cff0000]
  (-1214.2034,828.6134)arc(-0.285:180.285:2.525)arc(-180.285:0.285:2.525) --
  cycle;
\path[draw=black,line join=miter,line cap=butt,miter limit=4.00,line
  width=2.002pt] (347.6006,487.8044) .. controls (354.3204,502.1513) and
  (358.7857,516.5502) .. (355.4060,531.5048);
\path[cm={{1.61064,0.0,0.0,1.61064,(2242.0755,-725.63109)}},fill=cff0000]
  (-1214.2034,828.6134)arc(-0.285:180.285:2.525)arc(-180.285:0.285:2.525) --
  cycle;
\path[draw=black,line join=miter,line cap=butt,miter limit=4.00,line
  width=2.002pt] (324.6549,569.0283) .. controls (308.4673,580.8148) and
  (290.5022,591.8136) .. (282.2734,608.7994);
\path[cm={{1.61064,0.0,0.0,1.61064,(2222.2565,-591.96615)}},fill=cff0000]
  (-1214.2034,828.6134)arc(-0.285:180.285:2.525)arc(-180.285:0.285:2.525) --
  cycle;
\path[draw=black,line join=miter,line cap=butt,miter limit=4.00,line
  width=2.002pt] (266.9480,637.8327) .. controls (250.0591,669.6679) and
  (236.0715,703.2904) .. (262.7020,743.2750);
\path[cm={{1.61064,0.0,0.0,1.61064,(2279.96,-527.1153)}},fill=cff0000]
  (-1214.2034,828.6134)arc(-0.285:180.285:2.525)arc(-180.285:0.285:2.525) --
  cycle;
\path[draw=black,line join=miter,line cap=butt,miter limit=4.00,line
  width=2.002pt] (294.4928,769.3652) .. controls (311.7059,780.7672) and
  (327.1107,792.0105) .. (319.8233,808.0289);
\path[draw=black,line join=miter,line cap=butt,miter limit=4.00,line
  width=2.002pt] (308.1771,862.6841) .. controls (306.4555,900.5013) and
  (310.6488,900.5555) .. (313.3745,900.3030);
\path[fill=cff0000] (491.73584,180.35548) node[above right] (text15064-9)
  {\begin{tabu}{ccccccccccc}
  	$+$ & $+$ & $+$ & $+$ & $-$ & $-$ & $+$ & $+$ & $+$ & $-$ \\ 
  	\rowfont{\color{cff0000}}
  	$(e_7,$ & $e_8,$ & $e_5,$ & $e_6,$ & $e_1,$ & $e_2,$ & $e_9,$ & $e_{10},$ & $e_4,$ & $e_3)$ &\\ 
  	\end{tabu}};
\path[fill=cff0000] (491.73584,273.73514) node[above right] (text15064-9-5)
  {\begin{tabu}{ccccccccccc}
  	$+$ & $+$ & $+$ & $+$ & $-$ & $-$ & $+$ & $+$ & $-$ & $+$ \\ 
  	\rowfont{\color{cff0000}}
  	$(e_7,$ & $e_8,$ & $e_5,$ & $e_6,$ & $e_1,$ & $e_2,$ & $e_9,$ & $e_{10},$ & $e_3,$ & $e_4)$ &\\ 
  	\end{tabu}};
\path[fill=cff0000] (491.73584,367.11478) node[above right] (text15064-9-7)
  {\begin{tabu}{ccccccccccc}
  	$+$ & $+$ & $-$ & $+$ & $-$ & $+$ & $+$ & $+$ & $+$ & $-$ \\ 
  	\rowfont{\color{cff0000}}
  	$(e_7,$ & $e_8,$ & $e_2,$ & $e_6,$ & $e_1,$ & $e_5,$ & $e_9,$ & $e_{10},$ & $e_4,$ & $e_3)$ &\\ 
  	\end{tabu}};
\path[fill=cff0000] (491.73584,460.49445) node[above right] (text15064-9-8)
  {\begin{tabu}{ccccccccccc}
  	$+$ & $+$ & $-$ & $+$ & $-$ & $+$ & $+$ & $+$ & $-$ & $+$ \\ 
  	\rowfont{\color{cff0000}}
  	$(e_7,$ & $e_8,$ & $e_2,$ & $e_6,$ & $e_1,$ & $e_5,$ & $e_9,$ & $e_{10},$ & $e_3,$ & $e_4)$ &\\ 
  	\end{tabu}};
\path[fill=cff0000] (491.73584,553.87408) node[above right] (text15064-9-0)
  {\begin{tabu}{ccccccccccc}
  	$+$ & $+$ & $+$ & $-$ & $+$ & $-$ & $+$ & $+$ & $+$ & $-$ \\ 
  	\rowfont{\color{cff0000}}
  	$(e_7,$ & $e_8,$ & $e_5,$ & $e_1,$ & $e_6,$ & $e_2,$ & $e_9,$ & $e_{10},$ & $e_4,$ & $e_3)$ &\\ 
  	\end{tabu}};
\path[fill=cff0000] (491.73584,647.25378) node[above right] (text15064-9-6)
  {\begin{tabu}{ccccccccccc}
  	$+$ & $+$ & $+$ & $-$ & $+$ & $-$ & $+$ & $+$ & $-$ & $+$ \\ 
  	\rowfont{\color{cff0000}}
  	$(e_7,$ & $e_8,$ & $e_5,$ & $e_1,$ & $e_6,$ & $e_2,$ & $e_9,$ & $e_{10},$ & $e_3,$ & $e_4)$ &\\ 
  	\end{tabu}};
\path[fill=cff0000] (491.73584,740.63342) node[above right] (text15064-9-2)
  {\begin{tabu}{ccccccccccc}
  	$+$ & $+$ & $-$ & $-$ & $+$ & $+$ & $+$ & $+$ & $+$ & $-$ \\ 
  	\rowfont{\color{cff0000}}
  	$(e_7,$ & $e_8,$ & $e_2,$ & $e_1,$ & $e_6,$ & $e_5,$ & $e_9,$ & $e_{10},$ & $e_4,$ & $e_3)$ &\\ 
  	\end{tabu}};
\path[fill=cff0000] (491.73584,834.01306) node[above right] (text15064-9-05)
  {\begin{tabu}{ccccccccccc}
  	$+$ & $+$ & $-$ & $-$ & $+$ & $+$ & $+$ & $+$ & $-$ & $+$ \\ 
  	\rowfont{\color{cff0000}}
  	$(e_7,$ & $e_8,$ & $e_2,$ & $e_1,$ & $e_6,$ & $e_5,$ & $e_9,$ & $e_{10},$ & $e_3,$ & $e_4)$ &\\ 
  	\end{tabu}};
\path[fill=cff0000] (291.41092,110.13641) node[above right] (text15064-9-08)
  {\textcolor{cff0000}{$w =4\: 5\: 9\: 10\: 1\: 2\: 6\: 7\: 8\: 3$}};
\end{tikzpicture}
\setlength{\tabcolsep}{6pt}
\caption{Graphic representation of Theorem \ref{theorem36}}
\label{figure5}
\end{figure}
\noindent
\bigskip
\newline
\textbf{Example.} If $p=3$ and $q=2$ then the members of $\mathcal{I}_{p,q}$ are $(51)(42)3$, $(51)3(42)$, $3(51)(42)$, $(42)(51)3$, $(42)3(51)$, $3(42)(51)$, $(4(51)2)3$, $3(4(51)2)$.
\begin{cor}
The cardinality of $\mathcal{I}_{p,q}$ is equal to $$m_q:=(n-1)\cdot(n-3)\dots(n-2q+1).$$
\end{cor}

\begin{proof}
Observe that if $p>q$ in the strictly pairing condition algorithm, once the pairs $(n-i+1,i)$ for all $1\le i \le q$ are placed the numbers $q<i\le 2q$ have a fixed position. Thus the problem of computing the cardinality of $\mathcal{I}_{p,q}$ is equivalent to counting the number of possibilities to place adjacently $q$ pairs of numbers inside $n$ boxes. The result follows by induction.
\end{proof}
\noindent
It thus follows that the homology class $[\mathcal{C}]$ of the cycle $\mathcal{C}$ inside the homology ring of $Z$ is given in terms of Schubert classes of elements in $I_{p,q}$ by 
\begin{equation}
[\mathcal{C}]=2^q\sum_{S\in \mathcal{I}_{p,q}}[S],
\label{eqcycle}
\end{equation}
and the sum is over $m_q$ summands. 
\noindent
\bigskip
\newline
\textbf{Remark.} Note that in both the case of the real form $SL(n,\R)$ and in the case of the real form $SU(p,q)$ there exists a fixed total number of Iwasawa-Schubert varieties which intersect at least one base cycle. This depends only on $n$ or $q$ respectively (and not on a specific open orbit). In the first case the number is $n!!$ and in the second case the number is $(n-1)\cdot(n-3)\dots(n-2q+1).$ Moreover, for each such Iwasawa-Schubert variety there is a canonical associated set of points of intersection with different base cycles and their number again only depends on $n$ and $q$ (and not on the Iwasawa-Schubert variety). In the first case the number is $2^{n/2}$ for $n$ even and $2^{(n-1)/2}$ for $n$ odd, while in the second case the number is $2^q$. Depending on how many open orbits there are, the points distribute themselves into the corresponding cycles in equal numbers. So in the case of $SL(n,\mathbb{R})$ for the odd dimensional case there is only one open orbit and thus the $2^{(n-1)/2}$ points of intersection all go to the same cycle. In the even dimensional case there are two open orbits so the points distribute themselves into half, $2^{n/2-1}$ in one orbit and $2^{n/2-1}$ in the other orbit. In the $SU(p,q)$ case there are more than $2^q$ open orbits but a maximal number of points, $2^q$, available to each Iwasawa-Schubert variety. So each Iwasawa-Schubert variety will intersect $2^q$ different open orbits and so the points will distribute as one in each orbit. 
\noindent
\bigskip 
\newline
The next result gives an algorithm that describes the Iwasawa-Schubert varieties together with their points of intersection for a given fixed open orbit.
\begin{thm}
Let $D_{\alpha}$ be an open orbit parametrized by the sequence $\alpha$. Then the following algorithm gives all the Schubert varieties that intersect the open orbit and are of dimension $pq$ as well as their points of intersection: 
\begin{itemize}
\item{Consider two copies of $\alpha$ denoted by $e$ and $w$ respectively. Step by step replace the $-$'s and $+$'s in $e$ with the pairs $(e_i,e_{2q-i+1})$ starting with $i=1$ and going down until $i=q$ and all pairs are used. In $w$ the replacement starts with the pairs $(n-i+1,i)$ for $i=1$ and goes on until $i=q$ and all pairs are used.}
\item{ The replacement of the pairs is carried out using the following rule: at each step in $e$, $e_i$ will replace a $-$ and $e_{2q-i+1}$ a $+$ and the $-$ and $+$ to be replaced need to be adjacent in $\alpha$, namely they need to be consecutive as $-+$ or $+-$ or they need to sit at the immediate right and left of other already introduced pairs. In $w$ the replacement is carried out in parallel with $e$ and in the same way as in $e$ with the condition that whenever $e_i$ sits before $e_{2q-i+1}$ in $e$, in $w$, $n-i+1$ is placed on the corresponding sit to $e_i$ and $i$ is placed on the corresponding sit to $e_{2q-i+1}$.}
\item{The remaining single elements $e_i$, for all  $p+1\le i \le n$, and $i$ for all $q+1\le i \le p$ are placed in the remaining spots in increasing order of their indices.}
\end{itemize} 
Each $w$ obtained from this algorithm parametrizes an Iwasawa-Schubert variety $S_w$ that intersects the open orbit in exactly one point, its corresponding $e$. If $T_\alpha$ denotes the set of intersection points associated to an open orbit $D_{\alpha}$ and $\mathcal{D}$ denotes the set of all open orbits $D_\alpha$, then $f:\mathcal{D}\rightarrow \mathcal{P}({\fix(T)})$, given by $D_\alpha\mapsto T_\alpha$, is injective. 
\label{theorem38}
\end{thm}
\begin{figure}
\centerline{\includegraphics[scale=0.8]{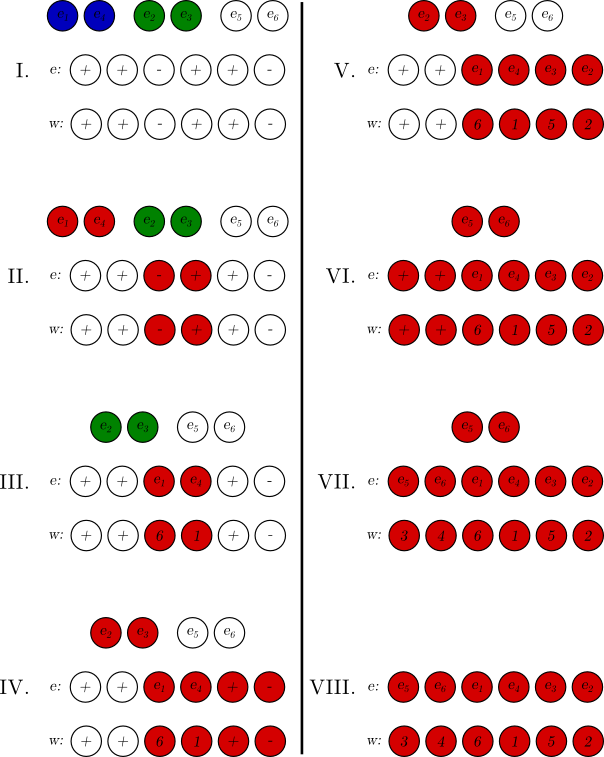}}
\caption{Graphic representation of Theorem \ref{theorem38} for an open $SU(4,2)$-orbit.}
\label{fig43}
\end{figure}
\begin{proof}
The algorithm constructs all Schubert varieties associated to the open orbit that satisfy the necessary and suffiecient conditions of \textbf{Theorem 39} in order for such Schubert varieties to intersect the open orbit.
\end{proof}
\noindent
Figure \ref{fig43} constructs one of the Schubert varieties which intersects the open orbit parametrized by $\alpha=++-++-$ together with its point of intersection following the algorithm described in \ref{theorem38}. 
\noindent
\bigskip
\newline
Let $C_0^\alpha$ be the base cycle associated to the open orbit $D_\alpha$ and let $\mathcal{S}_{C_0}^\alpha$ be the set of Iwasawa-Schubert varieties associated to $C_0^\alpha$. It then follows that the homology class $[C_0^\alpha]$ of the base cycle inside the homology ring of $Z$ is given in terms of the Schubert classes of elements in $\mathcal{S}_{C_0}^\alpha$ by: 
\begin{equation}
[C_0^\alpha]=\sum_{S\in\mathcal{S}_{C_0}^\alpha}\,[S].
\label{eqbasecycle}
\end{equation}
\noindent
\textbf{Example1.} The orbit parametrized by $\alpha=++\dots+--\dots-$, namely $p$ consecutive $+$'s and $q$ consecutive $-$'s is intersected only by one Schubert variety. If $p>q$, the Schubert variety is parametrized by $(q+1)\dots (2q)(n-q+1)\dots (n)(1)\dots (q)$ and the point of intersection is $e_{n-q+1}\dots e_{n}e_{q+1}\dots e_{2q}e_{1}\dots e_{q}$. If $p=q$, the Schubert variety is parametrized by $(q+1)\dots (n)(1)\dots (q)$ and the point of intersection is $e_{q+1}\dots e_n e_1\dots e_q.$
\noindent
\bigskip
\newline
\textbf{Example2.} The open orbit parametrized by $\alpha=+\dots+(+-)(+-)\dots(+-)$, a sequence of $p-q$ $+$'s for the case when $p>q$ and $q$ pairs $(+-)$ is intersected by $(2q)!!=(2q-1)\cdot (2q-3)\dots 1$ Iwasawa-Schubert varieties. Note that if $p=q$ this is the same number as the cardinality of $\mathcal{I}_{p,q}$.
\noindent
\bigskip
\newline
\textbf{Example3.} The open orbit parametrized by $\alpha=(+-+)(+-+)\dots (+-+)(+-+)+\dots+$, for $p\ge q$, is intersected by $2q\cdot(2q-2)\cdot(2q-4)\dots2.$

\section{Main results for $Z=G/P$}
In contrast to both the case of open $\SLR$-orbits and of open $SL(m,\H)$-orbits, in the case of $SU(p,q)$ every open orbit in $Z=G/P$ is measurable. As usual, the Weyl group element $w$ that parametrizes a Schubert variety $S_w$ is divided into \textbf{blocks} $B_i$ corresponding to the dimension sequence $(d_1,\dots, d_s)$ defining the parabolic subgroup $P$.
\begin{defi}
A permutation $w$ is said to satisfy the \textbf{generalized pairing condition} if for each block $B_i$ the elements inside $B_i$ can be rearranged to form a new permutation $\hat{w}$ that satisfies the pairing condition.  
\end{defi}
\begin{comment}
\begin{defi}
A permutation $w$ is said to satisfy the \textbf{generalized strictly pairing condition} if for each block $B_i$ the elements inside $B_i$ can be rearranged to form a new permutation $\hat{w}$ that satisfies the strictly pairing condition and whenever numbers $n-i+1$ for $1\le i \le q$ sit in the same box their completion to their corresponding pair $i$ is done starting with the bigger number.
\end{defi}
\noindent
\textbf{Example.} For example $(3)(4)(156)(2)$ satisfies the generalized strictly pairing condition for $SU(4,2)$ and $d=(1,1,3,1)$, but $(3)(4)(256)(1)$ does not satisfy it because $5$ got its completion before $6$ inside the same box.
\end{comment}
\begin{thm}
A Schubert variety $S_w$ has non-empty intersection with at least one open $SU(p,q)$-orbit if and only if $w$ satisfies the generalized pairing condition, i.e. if and only if there exists a lift of $S_w$ to a Schubert variety $\hat{S}_{\hat{w}}$ that intersects at least one open orbit in $Z=G/B$.
\end{thm}
\begin{proof}
If $w$ satisfies the generalized pairing condition, then one can find another representative $\hat{w}$ of the parametrization coset of $S_w$ that satisfies the pairing condition and thus a Schubert variety $\hat{S}_{\hat{w}}$ that intersects at least one open orbit. Because the projection map $\pi$ is equivariant it follows that $\pi(S_{\hat{w}})=S_w$ has non-empty intersection with at least one open orbit.
\noindent
\bigskip
\newline
Conversely, suppose that $S_w$ has non-empty intersection with at least one open orbit, say $D$, and consequently with at least one base cycle, say $C_0$. Then for every point $p \in S_w\cap C_0 $ there exists $\hat{p}\in \hat{S}_{\hat{w}} \cap \hat{C}_0$, where $\hat{C}_0$ is the base cycle inside a choice of lifting $\hat{D}$ for $D$, with $\pi (\hat{p})=p$ and $\pi(\hat{S}_{\hat{w}})=S_w$ for some Schubert variety in $G/B$ indexed by $\hat{w}$. It follows that $\hat{w}$ satisfies the pairing condition and $w$ is obtained from $\hat{w}$ by dividing $\hat{w}$ into blocks $B_1\dots B_s$ and arranging the elements in each such block in increasing order. This shows that $w$ satisfies the generalized pairing condition.  
\end{proof}
\noindent
%\begin{comment}
Recall that in the $G/P$ case the dimension of each base cycle is dependent on the dimension sequence that defines the corresponding open orbit. Thus the dimension of the Iwasawa-Schubert varieties will vary with the open orbit. However, the canonical parametrization of an open orbit will turn out to be a useful tool in giving a unifying algorithm to describe the Iwasawa-Schubert varieties for an arbitrary open orbit.   
\begin{defi}
Let $D_{a,b}$ be an open orbit in $Z=G/P$ and start with $e$ and $w$ being two copies of the canonical parametrization $\alpha$ associated to $D_{a,b}$. Then $e$ and $w$ are said to satisfy the \textbf{generalized strictly pairing condition for} $D_{a,b}$ if they are constructed by the following algorithm:
\begin{itemize}
\item{Let $f_j=min(a_j,b_j)$ for all $1\le j \le s$ and consider the pairs $(e_i,e_{2q-i+1})$, for all $1\le i \le \sum_{j=1}^sf_j$. For each block $j$ in $e$ choose $f_j$ arbitrary pairs from the above considered pairs. Replace the first $f_j$ $-'s$, with the $e_i$'s of the chosen pairs arranged in increasing order of their indices, and replace the last $f_j$ $+$'s with the $e_{2q-i+1}$'s of the chosen pairs arranged in increasing order of their indices. For $w$ one considers the pairs $(i,n-i+1)$, for all $1\le i \le \sum_{j=1}^sf_j$ and the replacement is done in parallel with $e$ and in the same way as in $e$.}
\item{The remaining pairs $(e_i,e_{2q-i+1})$ and $(i,n-i+1)$, for all $\sum_{j=1}^sf_j<i\le q$, together with the single elements $e_i$, for all $2q+1\le i \le n$, and $i$, for all $q+1\le i \le p$ corresponding to $e$ and $w$ respectively, are placed by using the algorithm of Theorem 41 for the sequences $e$ and $w$ obtained by ignoring the already replaced spots and seen as parametrizing an open orbit in $G/B$.  }
\end{itemize}
\end{defi}
\noindent
If $w$ satisfies the generalized strictly pairing condition for a fixed open orbit $D_{a,b}$, let $\hat{w}$ be the permutation obtained from $w$ by making the following choice of rearrangements inside each block. Assume without loss of generality that $p>q$ and consider an arbitrary block $d_i=a_i+b_i$, with $f_i=min(a_i,b_i)$, which is of course arranged in increasing order.
\noindent
\bigskip
\newline
By definition the block $d_i$ contains $f_i$ pairs $(h_j,g_j)$ with $h_j<q$  and $g_j>n-q$ for all $1\le j \le f_i$ and $k_j$ numbers that belong to the set $\{q+1,\dots, p\}$, for all $1\le j \le max(a_i,b_i)-min(a_i,b_i)$. Rearrange this block by placing the numbers $h_j$ in increasing order at the right of the biggest number among the $g_j$'s. Note that now this block satisfies the strictly pairing condition. Moreover, in order to go back to the rearrangement of the numbers in increasing order one needs to loose $a_ib_i$ in length, because each $h_j$ is smaller than each $g_j$ and also than each $k_j$. Call $\hat{w}$ the \textbf{canonical rearrangement} of $w$.
\noindent
\bigskip
\newline
Let $\hat{\alpha}$ be the sequence obtained from the canonical parametrization $\alpha$ associated to $D_{a,b}$ by placing in each block $i$, the first $f_i$ $-$'s to the end of the block. By definition and Theorem 41, $\hat{w}$ parametrizes an Iwasawa-Schubert variety corresponding to $D_{\hat{\alpha}}$. Call the base cycle $\hat{C}_0$ inside $D_{\hat{\alpha}}$ the \textbf{canonical lifting} of the base cycle $C_0$.
\begin{thm}
A Schubert variety $S_w$ is an Iwasawa-Schubert variety for the open orbit $D_{a,b}$ if and only if $w$ satisfies the generalized strictly pairing condition for $D_{a,b}$, i.e. if and only if $S_w$ lifts to $S_{\hat{w}}$ in $\mathcal{S}_{\hat{C}_0}$, where $\hat{w}$ is the canonical rearrangement of $w$ and $\hat{C}_0$ is the canonical lifting of $C_0$. Under this condition, $S_w$ intersects $D_{a,b}$ in precisely one point constructed by the generalized strictly pairing condition algorithm together with $w$. 
\end{thm}
\begin{proof}
If $w$ satisfies the generalized strictly pairing condition for $D_{a,b}$, then by the above observation the canonical rearrangement $\hat{w}$ of $w$ is just another representative of the parametrization coset of $S_w$. Moreover, $\hat{w}$ satisfies the strictly pairing condition and thus parametrizes a Schubert variety $S_{\hat{w}}$ that intersects $\hat{C}_0$. Since the projection map $\pi$ is equivariant, it follows that $\pi(S_{\tilde{w}})=S_w$ intersects $C_0$.
\noindent
\bigskip
\newline
Conversely, assume that $S_w$ lifts to $S_{\hat{w}}$ such that $\hat{w}$ satisfies the strictly pairing condition and $S_{\hat{w}}$ is an Iwasawa-Schubert variety for $\hat{C}_0$ which is the canonical lift of $C_0$. 
Then, $\dim S_{\hat{w}}-\dim S_w=(\sum_{i=1}^sa_i)(\sum_{j=1}^sb_j)-\sum_{1\le i <j\le s}(a_ib_j+b_ia_j)=\sum_{i=1}^sa_ib_i$ and this happens precisely when $w$ satisfies the generalized strictly pairing condition. 
%Observe that when one divides $\alpha_^$ into blocks $i$ of length $d_i$, then the first $max(a_i-b_i)-min(a_i-b_i)$ members of block $i$ are either $+$ or $-$. In the first case the only possibility of placing pairs following the algorithm of Theorem 20 is on the last $f_i$ pairs and this pushes down to the generalized strictly pairing condition. In the second case one could also use the first $max(a_i-b_i)-min(a_i-b_i)$ to place pairs 
%\noindent
%\bigskip
%\newline
%kk
\end{proof}
\noindent
Let $C_0^\alpha$ be the base cycle associated to the open orbit $D_\alpha$ with $\alpha$ the canonical parametrization for open orbits in $G/P$ and let $\mathcal{S}_{C_0}^\alpha$ be the set of Iwasawa-Schubert varieties associated to $C_0^\alpha$. It then follows that the homology class $[C_0^\alpha]$ of the base cycle inside the homology ring of $Z$ is given in terms of the Schubert classes of elements in $\mathcal{S}_{C_0}^\alpha$ by: $$[C_0^\alpha]=\sum_{S\in\mathcal{S}_{C_0}^\alpha}\,[S].$$
\noindent
\textbf{Example (Maximal parabolics).} Let $Z$ be identified with $Gr(k,n)$, the Grassmannian of $k$ planes in $\C^n$. Equivalently, $P$ is a maximal parabolic subgroup defined by the dimension sequence $(k,n-k)$. Let $D_{a,b}$ be an arbitrary open orbit parametrized by $a: 0\le a_1\le a_2\le q$ and $b: 0\le b_1\le b_2\le p$, with $a_1+b_1=k$, $a_2+b_2=n-k$ and as usually $a_1+a_2=q$ and $b_1+b_2=p$. Let $f_1:=min(a_1,b_1)$, $f_2:=min(a_2,b_2)$ and $f:=max(f_1,f_2)$. The two blocks inside the canonical parametrization of $D_{a,b}$ look like $[(--\dots--)(+\dots+)(++\dots--)]$ or $[(--\dots --)(-\dots -)(++\dots ++)]$. Then, it follows from the above theorem that the number of Schubert varieties of dimension $pq$ intersecting this open orbit is $$f_1+f_2 \choose f$$
\noindent
Let us write down the concrete Schubert varieties for $Z=Gr(5,11)$ and $G_0=SU(7,4)$ and $D_{a,b}$ is parametrized by $a_1=3$, $a_2=1$ and $b_1=2$, $b_2=5$. Then $f_1=2$, $f_2=1$ and $\alpha=[(--)(-)(++)][(-)(++++)(+)]$. The first $f_1+f_2$ pairs are $(11,1)$, $(10,2)$ and $(9,3)$. Following the algorithm described in the above theorem, one can choose two arbitrary pairs among these for the first block of $\alpha$ and one pair for the second block of $\alpha$. Then the remaining pair $(8,4)$ together with the numbers $5,6,7$ are placed using the algorithm for the $G/B$ case for the orbit parametrized by $-++++$. For a graphic representation of this see Figure \ref{figure7}. This gives the following $3$ Schubert varieties: $1281011345679$, $1389112456710$, $2389101456711$.  

\begin{figure}
\includegraphics[scale=0.9]{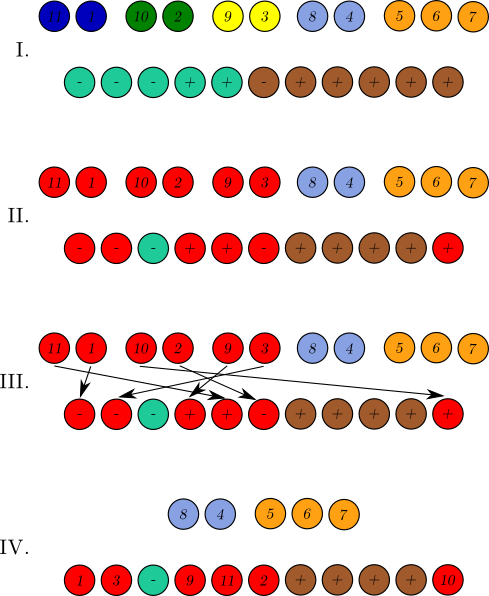}
\end{figure}

\begin{figure}
\includegraphics[scale=0.9]{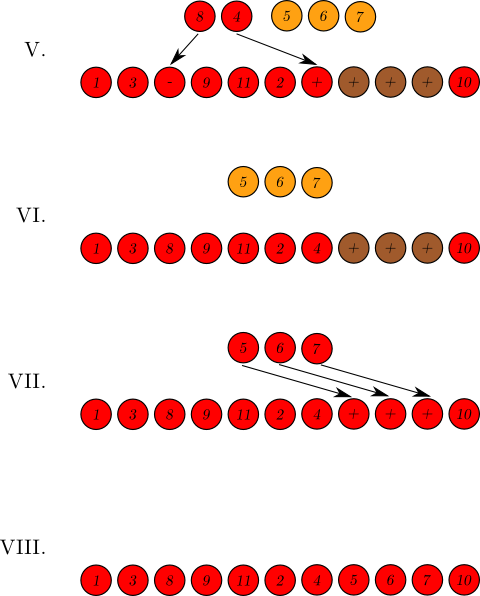}
\caption{Generalised strictly pairing condition example.}
\label{figure7}
\end{figure}

\newpage
\bibliographystyle{alphanum}
\bibliography{citations}

\begin{thebibliography}{FHW}

\bibitem[Bar]{Barlet1975}
D.~Barlet.
\newblock {Espace analytique r\'eduit des cycles analytiques complexes compacts
  d'un espace analytique complexe de dimension finie}.
\newblock In Fran\c{c}ois Norguet, editor, {\em Fonctions de Plusieurs
  Variables Complexes II SE - 1}, volume 482 of {\em Lecture Notes in
  Mathematics}, pages 1--158. Springer Berlin Heidelberg, 1975.

\bibitem[BK]{Barlet2000}
D.~Barlet and V.~Koziarz.
\newblock {Fonctions holomorphes sur l'espace des cycles: la m\'ethode
  d'intersection}.
\newblock {\em Mathematical Research Letters}, 7:537--549, 2000.

\bibitem[FHW]{Fels2006}
G.~Fels, A.~Huckleberry, and J.~A. Wolf.
\newblock {\em {Cycle Spaces of Flag Domains}}, volume 245 of {\em Progress in
  Mathematics}.
\newblock Birkh\"{a}user-Verlag, Boston, 2006.

\bibitem[HS]{Huckleberry2001}
A.~Huckleberry and A.~Simon, (Appendix by D.~Barlet).
\newblock {On cycle spaces of flag domains of $\mathrm{SL}_n(\R)$}.
\newblock {\em Journal f\"{u}r die reine und angewandte Mathematik},
  2001(541):171--208, 2001.

\bibitem[HW]{Huckleberry2002}
A.~Huckleberry and J.~A. Wolf.
\newblock {Cycle spaces of real forms of $\mathrm{SL}_n(\C)$}.
\newblock In {\em Complex Geometry}, pages 111--133. Springer Verlag, Berlin,
  2002.

\bibitem[Kna]{Knapp2002}
A.~W. Knapp.
\newblock {\em {Lie Groups Beyond an Introduction}}.
\newblock Progress in Mathematics. Birkh\"{a}user Boston, Boston, 2 edition,
  2002.

\bibitem[Wol1]{Wolf1969}
J.~A. Wolf.
\newblock {The action of a real semisimple group on a complex flag manifold. I:
  Orbit structure and holomorphic arc components}.
\newblock {\em Bulletin of the American Mathematical Society},
  75(6):1121--1237, November 1969.

\bibitem[Wol2]{Wolf1972}
J.~A. Wolf.
\newblock {Fine structure of Hermitian symmetric spaces}.
\newblock In W.~M. Boothby and G.L. Weiss, editors, {\em Symmetric Spaces},
  pages 271--357. Marcel Dekker, New York, 1972.

\bibitem[Wol3]{Wolf1995}
J.~A. Wolf.
\newblock {Exhaustion functions and cohomology vanishing theorems for open
  orbits on complex flag manifolds}.
\newblock {\em Mathematical Research Letters}, 2(2), 1995.

\end{thebibliography}
\end{document}